\documentclass[11pt]{amsart}

\usepackage{amsmath}
\usepackage{amssymb}
\usepackage{amsthm}
\usepackage{cancel}
\usepackage{tcolorbox}
\usepackage[margin=1.2in]{geometry}
\usepackage{tikz,tikz-cd}
\usetikzlibrary{external}
\usetikzlibrary{calc,arrows.meta}
\usepackage{cases}
\usepackage[]{mdframed}
\usepackage{url}
\usepackage{hyperref}
\usepackage[normalem]{ulem}

\newcommand{\Z}{\mathbb{Z}}
\newcommand{\C}{\mathbb{C}}
\newcommand{\R}{\mathbb{R}}

\newcommand{\mN}{L}
\newcommand{\mng}{r}
\newcommand{\oG}{M}
\newcommand{\oH}{N}
\newcommand{\mwo}{\omega_1}
\newcommand{\mwt}{\omega_2}
\newcommand{\whl}{P}
\newcommand{\wl}{Q}

\theoremstyle{definition}
\newtheorem{definition}{Definition}

\theoremstyle{remark}
\newtheorem{remark}{Remark}
\theoremstyle{plain}
\newtheorem{theorem}{Theorem}
\newtheorem{proposition}{Proposition}

\newtheorem{corollary}{Corollary}

\numberwithin{equation}{section}
\numberwithin{lemma}{section}
\numberwithin{definition}{section}
\numberwithin{proposition}{section}
\numberwithin{corollary}{section}
\numberwithin{theorem}{section}
\numberwithin{remark}{section}

\title[Heat kernels on Cayley graphs \dots]{Heat kernels on Cayley graphs of free products of finite groups}

\author{Kamila Kashaeva}
\address{Section de mathématiques, Université de Genève, rue du Conseil-Général 7-9, 1205 Genève, Suisse}
\email{kamila.kashaeva@unige.ch}
\date{}
\begin{document}

\maketitle

\begin{abstract}
We obtain explicit formulas for heat kernels of the normalized Laplacians for two infinite families of Cayley graphs. The first family corresponds to free products of finite groups $G=G_1*\dots*G_{\mng}$, with $\mng \ge 2$, and where all factors $G_i$ have the same order $\mN\ge 2$. The second family corresponds to free products of two arbitrary nontrivial finite groups $G*H$ of unequal order. In both families, the generating set consists of all non-identity elements in each factor. 
Extending the approach of Chung--Yau, we show that the Cayley graph in each family
strongly and regularly covers a weighted half-line and a weighted line, respectively. We solve the spectral problems on the half-line and the line, and then apply spectral transfer principles for strong and regular coverings to derive formulas for the heat kernels on the Cayley graphs. 
For the first family, the spectrum of the Laplacian consists of a single interval together with, when $\mN>\mng$, one additional isolated eigenvalue. When $\mN=2$, we recover the well-known results for regular trees. For the second family, the spectrum of the Laplacian consists of two intervals and two eigenvalues.

\end{abstract}

\section{Introduction}

There are relatively few examples of infinite graphs for which the spectrum of the Laplacian and the associated heat kernel are explicitly known. Heat kernels are fundamental solutions of the heat equation and are also one way of encoding the spectral decomposition of the Laplacian. 
Explicit formulas can be useful in counting problems on graphs, such as counting paths and geodesics. They are also particularly convenient under coverings, where the heat kernels of the covering and quotient graphs can be related directly by summation or periodization formulas; see, for example,~\cite{MR3394421,MR1667452,MR1852183}.
Explicit expressions are known for lattice graphs and regular trees; see~\cite{MR5005075} and the references therein. For regular trees, explicit formulas for the heat kernel were obtained by Chung--Yau~\cite{MR1667452}, Cowling--Meda--Setti~\cite{MR1653343}, and Chinta--Jorgenson--Karlsson~\cite{MR3394421}. Related spectral results for infinite Cayley and Schreier graphs can be found in the works of Kesten~\cite{MR109367}, Bartholdi--Grigorchuk~\cite{MR1841750} and Grigorchuk--Nagnibeda--Pérez~\cite{MR4458570}.

In a recent work~\cite{MR5101585}, we obtained an explicit formula for the heat kernel of a Cayley graph of $\operatorname{PSL}_2(\Z)$, where we used techniques inspired by Chung--Yau to reduce the problem to a simpler weighted graph. The purpose of this paper is to extend this approach to a broader class of Cayley graphs arising from free products of finite groups. By successively taking quotients by finite groups of symmetries, we reduce these graphs to a weighted half-line and a weighted line. We then solve the corresponding spectral problems and use the covering relations to recover the heat kernels on the original Cayley graphs. Our approach can be compared to work by Figà--Talamanca--Picardello~\cite{MR710827}, Iozzi--Picardello~\cite{MR729363}, Faraut--Picardello~\cite{MR779540}, Cartwright--Soardi~\cite{MR820357,MR846137}, Aomoto--Kato~\cite{MR948999}, and Woess~\cite{MR1743100} on random walks and harmonic analysis on free products and related symmetric graphs.

Recall that, if $\Gamma$ is a Cayley graph of a group $G$ with a symmetric generating set $S=S^{-1}$, the normalized Laplacian $\mathcal{L}_{\Gamma}$ on $\Gamma$ is a bounded self-adjoint operator in the Hilbert space $\ell^2(G)$ defined by
\begin{equation}\label{Laplacian_G}
(\mathcal{L}_{\Gamma}f)(x)=f(x)-\frac{1}{|S|}\sum_{a\in S}f(xa),\quad \forall x\in G.
\end{equation}
The heat operator $h_t^{\Gamma}$ of $\Gamma$ is the exponential of the Laplacian: $h_t^{\Gamma}=e^{-t\mathcal{L}_{\Gamma}}$. The invariance of $\Gamma$ under left translations implies that the heat operator is described in terms of a \emph{heat kernel function} $k_t^{\Gamma}\colon G\to \R$ defined through the formula
\begin{equation}\label{heat_kernel_det} 
(h_t^{\Gamma}f)(x)=\sum_{y\in G}k_t^{\Gamma}(y^{-1}x)f(y),\quad f\in \ell^2(G).
\end{equation}

We consider two families of Cayley graphs. The first family comes from free products
$
G_1*\dots*G_r
$
where $r\ge 2$, and each $G_i$ is a finite group of a given order $\mN\ge2$. The second family is associated to free products
$
G*H
$,
where 
$2\le|G|<|H|$.
In both cases, the generating set is the union of all non-identity elements in each free factor. This choice of generators leads to Cayley graphs with a tree-like structure obtained by gluing complete graphs along vertices. 

For both families, we obtain explicit formulas for the heat kernel functions, from which we extract the corresponding spectra. For the first family, the spectrum consists of one interval, together with an additional isolated eigenvalue when $\mN>\mng$. In the special case $\mN=2$, the Cayley graph reduces to a regular tree, and our formula recovers the heat kernel of the $\mng$-regular tree first obtained by Chung--Yau~\cite{MR1667452}. For the second family, the spectrum consists of two intervals and two isolated eigenvalues.
The occurrence of continuous spectrum, together with discrete eigenvalues, also appears in more general spectral problems, see for example Aomoto~\cite{MR1084706}, Sunada~\cite{MR1152674}, and more recently Avni--Breuer--Simon~\cite{MR4103777}.

We now describe our main results in more detail.

Let 
$$
\mathcal X=\mathcal X_{\mN}^r= \mathrm{Cay}(\mathcal A, S)
$$ denote the Cayley graph associated with the free product
$$
\mathcal A=\mathcal A_{\mN}^r:=G_1*\cdots *G_{\mng},\quad r\ge2,\quad |G_1|=\dots=|G_{\mng}|=\mN\ge2,
$$ 
and the generating set
$S = \bigsqcup_{i=1}^\mng \bigl(G_i \setminus \{e\}\bigr)$. The case $\mN=5$ and $\mng=3$ is
illustrated in Fig.~\ref{fig:cayley_graph}.

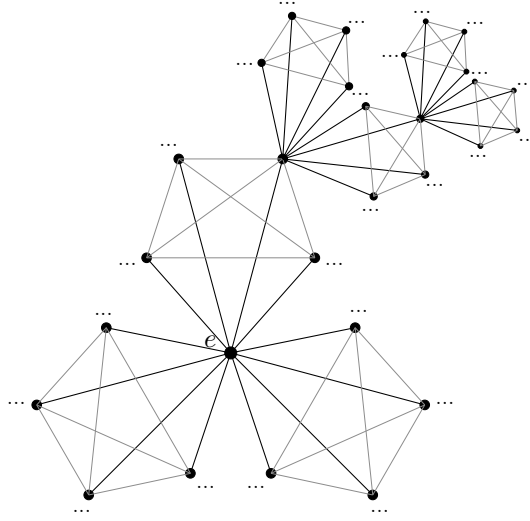
\begin{figure}[h]
\center
\begin{tikzpicture}[scale=0.9]

\tikzset{
  treeedge/.style={line width=0.6pt},
  cliqueedge/.style={gray!90, line width=0.3pt}
}

\coordinate (e) at (0,0);

\coordinate (C1) at (90:1.8);

\coordinate (a0) at ($ (C1) + (270:1.3) $); 
\coordinate (a1) at ($ (C1) + (342:1.3) $); 
\coordinate (a2) at ($ (C1) + (54:1.3) $);
\coordinate (a3) at ($ (C1) + (126:1.3) $);
\coordinate (a4) at ($ (C1) + (198:1.3) $);

\coordinate (C2) at (210:1.8);

\coordinate (b0) at ($ (C2) + (30:1.3) $); 
\coordinate (b1) at ($ (C2) + (102:1.3) $); 
\coordinate (b2) at ($ (C2) + (174:1.3) $);
\coordinate (b3) at ($ (C2) + (246:1.3) $);
\coordinate (b4) at ($ (C2) + (318:1.3) $);

\coordinate (C3) at (330:1.8);

\coordinate (c0) at ($ (C3) + (150:1.3) $); 
\coordinate (c1) at ($ (C3) + (222:1.3) $); 
\coordinate (c2) at ($ (C3) + (294:1.3) $);
\coordinate (c3) at ($ (C3) + (6:1.3) $);
\coordinate (c4) at ($ (C3) + (78:1.3) $);

\draw[fill=black] (e) circle [radius=2.5pt];
\node at (-0.3,0.17) {\small$e$};

\draw[fill=black] (a1) circle [radius=2pt];
\draw[fill=black] (a2) circle [radius=2pt];
\draw[fill=black] (a3) circle [radius=2pt];
\draw[fill=black] (a4) circle [radius=2pt];

\draw(e)--(a1); \draw(e)--(a2); \draw(e)--(a3); \draw(e)--(a4);
\draw[cliqueedge](a1)--(a2); \draw[cliqueedge](a1)--(a3); \draw[cliqueedge](a1)--(a4); 
\draw[cliqueedge](a2)--(a3); \draw[cliqueedge](a2)--(a4); 
\draw[cliqueedge](a3)--(a4); 

\node at ($ (C1) + (342:1.6) $) {\tiny$...$};
\node at ($ (C1) + (126:1.55) $) {\tiny$...$};
\node at ($ (C1) + (198:1.6) $) {\tiny$...$};

\draw[fill=black] (b1) circle [radius=2pt];
\draw[fill=black] (b2) circle [radius=2pt];
\draw[fill=black] (b3) circle [radius=2pt];
\draw[fill=black] (b4) circle [radius=2pt];

\draw(e)--(b1); \draw(e)--(b2); \draw(e)--(b3); \draw(e)--(b4);
\draw[cliqueedge](b1)--(b2); \draw[cliqueedge](b1)--(b3); \draw[cliqueedge](b1)--(b4); 
\draw[cliqueedge](b2)--(b3); \draw[cliqueedge](b2)--(b4); 
\draw[cliqueedge](b3)--(b4); 

\node at ($ (C2) + (102:1.5) $) {\tiny$...$};
\node at ($ (C2) + (174:1.6) $) {\tiny$...$};
\node at ($ (C2) + (246:1.55) $) {\tiny$...$};
\node at ($ (C2) + (318:1.6) $) {\tiny$...$};

\draw[fill=black] (c1) circle [radius=2pt];
\draw[fill=black] (c2) circle [radius=2pt];
\draw[fill=black] (c3) circle [radius=2pt];
\draw[fill=black] (c4) circle [radius=2pt];

\draw(e)--(c1); \draw(e)--(c2); \draw(e)--(c3); \draw(e)--(c4);
\draw[cliqueedge](c1)--(c2); \draw[cliqueedge](c1)--(c3); \draw[cliqueedge](c1)--(c4); 
\draw[cliqueedge](c2)--(c3); \draw[cliqueedge](c2)--(c4); 
\draw[cliqueedge](c3)--(c4); 

\node at ($ (C3) + (222:1.6) $) {\tiny$...$};
\node at ($ (C3) + (294:1.55) $) {\tiny$...$};
\node at ($ (C3) + (6:1.6) $) {\tiny$...$};
\node at ($ (C3) + (78:1.55) $) {\tiny$...$};

\coordinate (Attach) at (a2);

\coordinate (D1) at ($ (Attach) + (5:1.5) $);
\coordinate (D2) at ($ (Attach) + (75:1.5) $);


\coordinate (d1_0) at ($ (D1) + (185:0.7) $);
\coordinate (d1_1) at ($ (D1) + (257:0.7) $);
\coordinate (d1_2) at ($ (D1) + (329:0.7) $);
\coordinate (d1_3) at ($ (D1) + (41:0.7) $);
\coordinate (d1_4) at ($ (D1) + (113:0.7) $);

\draw[fill=black] (d1_1) circle [radius=1.5pt];
\draw[fill=black] (d1_2) circle [radius=1.5pt];
\draw[fill=black] (d1_3) circle [radius=1.5pt];
\draw[fill=black] (d1_4) circle [radius=1.5pt];

\draw(Attach)--(d1_1); \draw(Attach)--(d1_2); \draw(Attach)--(d1_3); \draw(Attach)--(d1_4);
\draw[cliqueedge](d1_1)--(d1_2); \draw[cliqueedge](d1_1)--(d1_3); \draw[cliqueedge](d1_1)--(d1_4); 
\draw[cliqueedge](d1_2)--(d1_3); \draw[cliqueedge](d1_2)--(d1_4); 
\draw[cliqueedge](d1_3)--(d1_4); 

\node at ($ (D1) + (257:0.9) $) {\tiny$...$};
\node at ($ (D1) + (325:0.9) $) {\tiny$...$};
\node at ($ (D1) + (113:0.9) $) {\tiny$...$};

\coordinate (d2_0) at ($ (D2) + (255:0.7) $);
\coordinate (d2_1) at ($ (D2) + (327:0.7) $);
\coordinate (d2_2) at ($ (D2) + (39:0.7) $);
\coordinate (d2_3) at ($ (D2) + (111:0.7) $);
\coordinate (d2_4) at ($ (D2) + (183:0.7) $);

\draw(Attach)--(d2_1); \draw(Attach)--(d2_2); \draw(Attach)--(d2_3); \draw(Attach)--(d2_4);
\draw[cliqueedge](d2_1)--(d2_2); \draw[cliqueedge](d2_1)--(d2_3); \draw[cliqueedge](d2_1)--(d2_4); 
\draw[cliqueedge](d2_2)--(d2_3); \draw[cliqueedge](d2_2)--(d2_4); 
\draw[cliqueedge](d2_3)--(d2_4); 

\node at ($ (D2) + (39:0.9) $) {\tiny$...$};
\node at ($ (D2) + (111:0.9) $) {\tiny$...$};
\node at ($ (D2) + (183:0.95) $) {\tiny$...$};

\draw[fill=black] (d2_1) circle [radius=1.5pt];
\draw[fill=black] (d2_2) circle [radius=1.5pt];
\draw[fill=black] (d2_3) circle [radius=1.5pt];
\draw[fill=black] (d2_4) circle [radius=1.5pt];

\coordinate (Attach2) at (d1_3);

\coordinate (E1) at ($ (Attach2) + (5:1) $);
\coordinate (E2) at ($ (Attach2) + (75:1) $);


\coordinate (e1_0) at ($ (E1) + (185:0.5) $);
\coordinate (e1_1) at ($ (E1) + (257:0.5) $);
\coordinate (e1_2) at ($ (E1) + (329:0.5) $);
\coordinate (e1_3) at ($ (E1) + (41:0.5) $);
\coordinate (e1_4) at ($ (E1) + (113:0.5) $);

\draw[fill=black] (e1_1) circle [radius=1pt];
\draw[fill=black] (e1_2) circle [radius=1pt];
\draw[fill=black] (e1_3) circle [radius=1pt];
\draw[fill=black] (e1_4) circle [radius=1pt];

\draw(Attach2)--(e1_1); \draw(Attach2)--(e1_2); \draw(Attach2)--(e1_3); \draw(Attach2)--(e1_4);
\draw[cliqueedge](e1_1)--(e1_2); \draw[cliqueedge](e1_1)--(e1_3); \draw[cliqueedge](e1_1)--(e1_4); 
\draw[cliqueedge](e1_2)--(e1_3); \draw[cliqueedge](e1_2)--(e1_4); 
\draw[cliqueedge](e1_3)--(e1_4); 

\node at ($ (E1) + (257:0.65) $) {\tiny$...$};
\node at ($ (E1) + (326:0.7) $) {\tiny$...$};
\node at ($ (E1) + (38:0.7) $) {\tiny$...$};
\node at ($ (E1) + (102:0.6) $) {\tiny$...$};

\coordinate (e2_0) at ($ (E2) + (255:0.5) $);
\coordinate (e2_1) at ($ (E2) + (327:0.5) $);
\coordinate (e2_2) at ($ (E2) + (39:0.5) $);
\coordinate (e2_3) at ($ (E2) + (111:0.5) $);
\coordinate (e2_4) at ($ (E2) + (183:0.5) $);

\draw(Attach2)--(e2_1); \draw(Attach2)--(e2_2); \draw(Attach2)--(e2_3); \draw(Attach2)--(e2_4);
\draw[cliqueedge](e2_1)--(e2_2); \draw[cliqueedge](e2_1)--(e2_3); \draw[cliqueedge](e2_1)--(e2_4); 
\draw[cliqueedge](e2_2)--(e2_3); \draw[cliqueedge](e2_2)--(e2_4); 
\draw[cliqueedge](e2_3)--(e2_4); 

\draw[fill=black] (e2_1) circle [radius=1pt];
\draw[fill=black] (e2_2) circle [radius=1pt];
\draw[fill=black] (e2_3) circle [radius=1pt];
\draw[fill=black] (e2_4) circle [radius=1pt];

\node at ($ (E2) + (39:0.7) $) {\tiny$...$};
\node at ($ (E2) + (111:0.65) $) {\tiny$...$};
\node at ($ (E2) + (183:0.7) $) {\tiny$...$};
\end{tikzpicture}
\caption{The Cayley graph $\mathcal{X}_{5}^3$. The bold edges correspond to geodesics to the identity.} \label{fig:cayley_graph}
\end{figure}

Denote by $|g|$ the length of the reduced word representing $g\in \mathcal A$, so that $|e|=0$.
Our first main result gives an explicit formula for the heat kernel function $k_t^{\mathcal X}$, defined in \eqref{heat_kernel_det}.
\begin{theorem} \label{K1}
The heat kernel function $k_t^{\mathcal X}$ defined in~\eqref{heat_kernel_det} is given by
${k}_t^{\mathcal X}(g)=K_t^{\mathcal X}(|g|)$, where, for $n\in\Z_{\ge0}$,
\begin{multline} \label{K1_formula}
 K_t^{\mathcal X}(n) = \frac{\max(0,q-s)}{(q+1)(-q)^n} e^{-t\frac{q+1}{q}}\\
+ \frac{2s}{\pi \sqrt{qs}^{\,n}} \int_0^\pi
 \frac{e^{-t\frac{1+qs-2\sqrt{qs}\cos x}{q(s+1)}}\Bigl( s\sin((n+1)x)+(q-1)\sqrt{\frac{s}{q}}\sin nx-\sin((n-1)x) \Bigr) \sin x}{\bigl((q-1)\sqrt{\frac{s}{q}}+(s-1)\cos x\big)^2 +\big((s+1)\sin x\bigr)^2} \mathrm{d}x,
\end{multline}
where $q:=\mN-1$ and $s:=\mng-1$.
\end{theorem}
As a consequence, using the extremal values $0$ and $\pi$ of the integration variable $x$ in \eqref{K1_formula}, we obtain a complete description of the spectrum of the Laplacian on $\mathcal{X}$.
\begin{corollary}
The spectrum of the Laplacian $\mathcal{L}_{\mathcal{X}}$ on $\ell^2(\mathcal A)$ consists of the interval
$$
\Big[1-\frac{\mN-2+2\sqrt{(r-1)(\mN-1)}}{r(\mN-1)}, 1-\frac{\mN-2-2\sqrt{(r-1)(\mN-1)}}{r(\mN-1)} \Big]
$$
and, when $\mN>r$, the point $\frac{\mN}{\mN-1}$.
\end{corollary}

The isolated point is an eigenvalue. Proposition~\ref{spectrum_L^pr1} describes the eigenfunctions of the projected Laplacian, and Proposition~\ref{efct_cov} gives their lift to $\ell^2$-eigenfunctions of~$\mathcal{L}_{\mathcal{X}}$.

Closely related spectral information for graphs $\mathcal{X}_{L}^r$ has been obtained by different methods  
 in the work  
of Faraut--Picardello~\cite{MR779540} and in Woess's book on random walks 
~\cite[§9.C]{MR1743100}.

We now turn to the second family of Cayley graphs. 
Let
$$
 \mathcal{Y} =\mathcal Y_{\oG,\oH}
  = \mathrm{Cay}(\mathcal B, S),
$$
where 
$$
\mathcal B=\mathcal B_{\oG,\oH}:=G*H, \quad |G|=\oG+1, \ |H|=\oH+1,
$$
and $S=(G\setminus\{e\}) \sqcup (H\setminus\{e\})$.
The case $\oG=2$ and $\oH=4$ is illustrated in Fig.~\ref{fig:cayley_graph2}. 

Our second main result applies to a slightly more general weighted graph, where we assign positive weights $\mwo$ and $\mwt$ to the edges corresponding to $G$ and $H$, respectively. 
In this case, the normalized Laplacian is given by
\begin{equation}\label{eq:laplacian-Y}
 (\mathcal{L}_{\mathcal{Y}}f)(x)=f(x)-\frac{1}{\oG\mwo+\oH\mwt} \Big(\sum_{a\in G\setminus{\{e\}}}\mwo f(xa) + \sum_{b\in H\setminus{\{e\}}}\mwt f(xb) \Big),\quad \forall x\in \mathcal{B}.
\end{equation}
Define the map
$\pi_e \colon \mathcal B \to\Z$
by 
\begin{equation}\label{CY_proj}
 \pi_e(g)=
\begin{cases} 
0 & \text{if $g=e$} \\
 |g| & \text{if the reduced form of $g$ starts with an element of $H$} \\
 -|g| & \text{if the reduced form of $g$ starts with an element of $G$}.
\end{cases}
\end{equation}

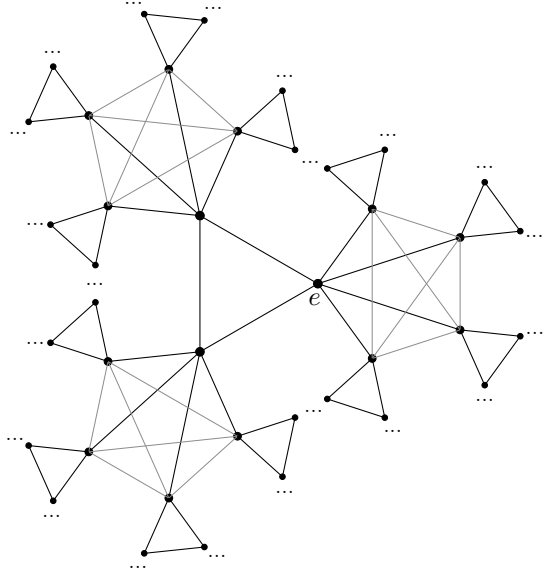
\begin{figure}[h]
\center
\begin{tikzpicture}[scale=0.8]

\tikzset{
  treeedge/.style={line width=0.6pt},
  cliqueedge/.style={gray!90, line width=0.3pt}
}

\coordinate (C0) at (0,0); 
\coordinate (t1) at ($(C0)+(0:1.3)$);
\coordinate (t3) at ($(C0)+(120:1.3)$);
\coordinate (t2) at ($(C0)+(240:1.3)$);

\coordinate (C1) at ($(t1)+(0:1.3)$);
\coordinate (a1) at ($(C1)+(180:1.3)$);
\coordinate (a2) at ($(C1)+(252:1.3)$);
\coordinate (a3) at ($(C1)+(324:1.3)$);
\coordinate (a4) at ($(C1)+(36:1.3)$); 
\coordinate (a5) at ($(C1)+(108:1.3)$); 

\coordinate (C2) at ($(t2)+(240:1.3)$);
\coordinate (b1) at ($(C2)+(60:1.3)$);
\coordinate (b2) at ($(C2)+(132:1.3)$);
\coordinate (b3) at ($(C2)+(204:1.3)$);
\coordinate (b4) at ($(C2)+(276:1.3)$);
\coordinate (b5) at ($(C2)+(348:1.3)$);

\coordinate (C3) at ($(t3)+(120:1.3)$);
\coordinate (c1) at ($(C3)+(-60:1.3)$);
\coordinate (c2) at ($(C3)+(12:1.3)$);
\coordinate (c3) at ($(C3)+(84:1.3)$);
\coordinate (c4) at ($(C3)+(156:1.3)$);
\coordinate (c5) at ($(C3)+(228:1.3)$);

\node at (1.25,-0.25) {\small$e$};

\draw[fill=black] (t1) circle [radius=2pt];
\draw[fill=black] (t2) circle [radius=2pt];
\draw[fill=black] (t3) circle [radius=2pt];

\draw[fill=black] (a1) circle [radius=1.8pt];
\draw[fill=black] (a2) circle [radius=1.8pt];
\draw[fill=black] (a3) circle [radius=1.8pt];
\draw[fill=black] (a4) circle [radius=1.8pt];
\draw[fill=black] (a5) circle [radius=1.8pt];

\draw[fill=black] (b1) circle [radius=1.8pt];
\draw[fill=black] (b2) circle [radius=1.8pt];
\draw[fill=black] (b3) circle [radius=1.8pt];
\draw[fill=black] (b4) circle [radius=1.8pt];
\draw[fill=black] (b5) circle [radius=1.8pt];

\draw[fill=black] (c1) circle [radius=1.8pt];
\draw[fill=black] (c2) circle [radius=1.8pt];
\draw[fill=black] (c3) circle [radius=1.8pt];
\draw[fill=black] (c4) circle [radius=1.8pt];
\draw[fill=black] (c5) circle [radius=1.8pt];

\draw(t1)--(t2)--(t3)--(t1); 

\draw(t1)--(a2); \draw(t1)--(a3); \draw(t1)--(a4); \draw(t1)--(a5);
\draw[cliqueedge](a2)--(a3); \draw[cliqueedge](a2)--(a4); \draw[cliqueedge](a2)--(a5); 
\draw[cliqueedge](a3)--(a4); \draw[cliqueedge](a3)--(a5); 
\draw[cliqueedge](a4)--(a5); 

\draw(t2)--(b2); \draw(t2)--(b3); \draw(t2)--(b4); \draw(t2)--(b5);
\draw[cliqueedge](b2)--(b3); \draw[cliqueedge](b2)--(b4); \draw[cliqueedge](b2)--(b5); 
\draw[cliqueedge](b3)--(b4); \draw[cliqueedge](b3)--(b5); 
\draw[cliqueedge](b4)--(b5); 

\draw(t3)--(c2); \draw(t3)--(c3); \draw(t3)--(c4); \draw(t3)--(c5);
\draw[cliqueedge](c2)--(c3); \draw[cliqueedge](c2)--(c4); \draw[cliqueedge](c2)--(c5); 
\draw[cliqueedge](c3)--(c4); \draw[cliqueedge](c3)--(c5); 
\draw[cliqueedge](c4)--(c5); 

\coordinate (a2t2) at ($(a2)+(222:1)$);
\coordinate (a2t3) at ($(a2)+(282:1)$);

\node at ($(a2t2)+(222:0.3)$){\tiny$...$};
\node at ($(a2t3)+(302:0.25)$){\tiny$...$};

\coordinate (a3t2) at ($(a3)+(294:1)$);
\coordinate (a3t3) at ($(a3)+(354:1)$);

\node at ($(a3t2)+(270:0.25)$){\tiny$...$};
\node at ($(a3t3)+(14:0.25)$){\tiny$...$};

\coordinate (a4t2) at ($(a4)+(6:1)$);
\coordinate (a4t3) at ($(a4)+(66:1)$);

\node at ($(a4t2)+(-20:0.25)$){\tiny$...$};
\node at ($(a4t3)+(90:0.25)$){\tiny$...$};

\coordinate (a5t2) at ($(a5)+(78:1)$);
\coordinate (a5t3) at ($(a5)+(138:1)$);

\node at ($(a5t2)+(78:0.25)$){\tiny$...$};
\node at ($(a5t3)+(158:0.3)$){\tiny$...$};

\draw[fill=black] (a2t2) circle [radius=1.3pt];
\draw[fill=black] (a2t3) circle [radius=1.3pt];
\draw(a2)--(a2t2)--(a2t3)--(a2); 

\draw[fill=black] (a3t2) circle [radius=1.3pt];
\draw[fill=black] (a3t3) circle [radius=1.3pt];
\draw(a3)--(a3t2)--(a3t3)--(a3); 

\draw[fill=black] (a4t2) circle [radius=1.3pt];
\draw[fill=black] (a4t3) circle [radius=1.3pt];
\draw(a4)--(a4t2)--(a4t3)--(a4); 

\draw[fill=black] (a5t2) circle [radius=1.3pt];
\draw[fill=black] (a5t3) circle [radius=1.3pt];
\draw(a5)--(a5t2)--(a5t3)--(a5);

\coordinate (b2t2) at ($(b2)+(102:1)$);
\coordinate (b2t3) at ($(b2)+(162:1)$);

\node at ($(b2t2)+(92:0.3)$){\tiny$...$};
\node at ($(b2t3)+(182:0.25)$){\tiny$...$};

\coordinate (b3t2) at ($(b3)+(174:1)$);
\coordinate (b3t3) at ($(b3)+(234:1)$);

\node at ($(b3t2)+(154:0.25)$){\tiny$...$};
\node at ($(b3t3)+(264:0.25)$){\tiny$...$};

\coordinate (b4t2) at ($(b4)+(246:1)$);
\coordinate (b4t3) at ($(b4)+(306:1)$);

\node at ($(b4t2)+(246:0.25)$){\tiny$...$};
\node at ($(b4t3)+(326:0.25)$){\tiny$...$};

\coordinate (b5t2) at ($(b5)+(318:1)$);
\coordinate (b5t3) at ($(b5)+(18:1)$);

\node at ($(b5t2)+(278:0.25)$){\tiny$...$};

\draw[fill=black] (b2t2) circle [radius=1.3pt];
\draw[fill=black] (b2t3) circle [radius=1.3pt];
\draw(b2)--(b2t2)--(b2t3)--(b2); 

\draw[fill=black] (b3t2) circle [radius=1.3pt];
\draw[fill=black] (b3t3) circle [radius=1.3pt];
\draw(b3)--(b3t2)--(b3t3)--(b3); 

\draw[fill=black] (b4t2) circle [radius=1.3pt];
\draw[fill=black] (b4t3) circle [radius=1.3pt];
\draw(b4)--(b4t2)--(b4t3)--(b4); 

\draw[fill=black] (b5t2) circle [radius=1.3pt];
\draw[fill=black] (b5t3) circle [radius=1.3pt];
\draw(b5)--(b5t2)--(b5t3)--(b5); 

\coordinate (c2t2) at ($(c2)+(-18:1)$);
\coordinate (c2t3) at ($(c2)+(42:1)$);

\node at ($(c2t3)+(82:0.25)$){\tiny$...$};

\coordinate (c3t2) at ($(c3)+(54:1)$);
\coordinate (c3t3) at ($(c3)+(114:1)$);

\node at ($(c3t2)+(54:0.25)$){\tiny$...$};
\node at ($(c3t3)+(134:0.25)$){\tiny$...$};

\coordinate (c4t2) at ($(c4)+(126:1)$);
\coordinate (c4t3) at ($(c4)+(186:1)$);

\node at ($(c4t2)+(92:0.25)$){\tiny$...$};
\node at ($(c4t3)+(226:0.25)$){\tiny$...$};

\coordinate (c5t2) at ($(c5)+(198:1)$);
\coordinate (c5t3) at ($(c5)+(258:1)$);

\node at ($(c5t2)+(178:0.25)$){\tiny$...$};

\draw[fill=black] (c2t2) circle [radius=1.3pt];
\draw[fill=black] (c2t3) circle [radius=1.3pt];
\draw(c2)--(c2t2)--(c2t3)--(c2); 

\draw[fill=black] (c3t2) circle [radius=1.3pt];
\draw[fill=black] (c3t3) circle [radius=1.3pt];
\draw(c3)--(c3t2)--(c3t3)--(c3); 

\draw[fill=black] (c4t2) circle [radius=1.3pt];
\draw[fill=black] (c4t3) circle [radius=1.3pt];
\draw(c4)--(c4t2)--(c4t3)--(c4); 

\draw[fill=black] (c5t2) circle [radius=1.3pt];
\draw[fill=black] (c5t3) circle [radius=1.3pt];
\draw(c5)--(c5t2)--(c5t3)--(c5); 

\end{tikzpicture}
\caption{The Cayley graph $\mathcal{Y}_{2,4}$. The bold edges correspond to geodesics to the identity.} \label{fig:cayley_graph2}
\end{figure}

\begin{theorem} \label{K2}
The heat kernel function $k_t^{\mathcal Y}$ defined in~\eqref{heat_kernel_det} for the Laplacian~\eqref{eq:laplacian-Y} is given by
$$
{k}_t^{\mathcal Y}(x)=K_t^{\mathcal Y}(\pi_e(x)),
$$ where, for $n\in\Z$,
\begin{multline*}
 K_t^{\mathcal Y}(n) = \frac{1}{\sqrt{|\pi_e^{-1}(n)|}} \Bigg( \sum_{k=1}^2 e^{-t\lambda_k} f_k(0)f_k(n)+ \int_0^\pi \sum_{\epsilon\in\{\pm1\}}  e^{-t\lambda_{\epsilon,x}} F_{\epsilon,x}(0)^T \mathcal{H}_{\epsilon}(x)^{-1} F_{\epsilon,x}(n) \mathrm{d}x \Bigg).
\end{multline*}
Here, $\lambda_k, \lambda_{\epsilon,x} \in \R$, $f_k \colon \Z \to \R$, $F_{\epsilon,x} \colon \Z \to \R^2$, and $2\times 2$ matrix $\mathcal{H}_{\epsilon}(x)$ are given explicitly in~Section~\ref{section2}. 
\end{theorem}

\begin{corollary}\label{Cor_spec2}
 The spectrum of the Laplacian $\mathcal{L}_{\mathcal Y}$ on $\ell^2(\mathcal B)$ defined in~\eqref{eq:laplacian-Y} is the following closed subset of $\R$, see Fig.~\ref{fig:specL2}:
$$
\Big[c-\frac{R_0}{2\kappa},c-\frac{R_{\pi}}{2\kappa}\Big] \sqcup \{\lambda_1\} \sqcup \Big[c+\frac{R_{\pi}}{2\kappa},c+\frac{R_0}{2\kappa}\Big] \sqcup \{\lambda_2\} 
$$
where
$$
 \kappa:= M\mwo+N\mwt, \quad c=\frac{(\oG+1)\mwo+(\oH+1)\mwt}{2\kappa},
$$
\begin{equation*}
R_x=\sqrt{\big((\oG+1)\mwo-(\oH+1)\mwt\big)^2+4(\oG+\oH)\mwo\mwt+8\sqrt{\oG\oH}\mwo\mwt\cos x},
\end{equation*}
and
\begin{equation*}
 \lambda_1=\frac{(\oH+1)\mwt}{\kappa}, \quad
  \lambda_2=1+\frac{\mwo+\mwt}{\kappa}.
\end{equation*}
\end{corollary}

In the unweighted case ($\mwo=\mwt$), Corollary~\ref{Cor_spec2} recovers the spectral result of Cartwright--Soardi~\cite{MR820357}, after the transformation 
from the adjacency operator to the normalized Laplacian, while in general it corresponds to the spectrum obtained by Cartwright--Soardi in~\cite[§4.3]{MR846137}. Related spectral results for weighted random walks on free products of cyclic groups were obtained by Aomoto--Kato~\cite{MR948999}.

\begin{figure}[h]
\center
\begin{tikzpicture}[xscale=6,line/.style = {draw,ultra thick, shorten >=-2pt, shorten <=-2pt}]

\draw[-latex] (0,0) -- (2,0);


\node at (0.8,-0.4) {\small$c$}; \draw (0.8,0.15) -- (0.8,-0.15);

\node at (0.88,-0.4) {$\lambda_1$};
\node at (1.6,-0.4) {$\lambda_2$}; 

\draw[line, {Circle[length=4pt]}-{Circle[length=4pt]}]  (0.1,0) -- (0.6,0);
\draw[line, {Circle[length=4pt]}-{Circle[length=4pt]}]  (0.88,0) -- (0.88,0);

\draw[line, {Circle[length=4pt]}-{Circle[length=4pt]}]  (1,0) -- (1.5,0);
\draw[line, {Circle[length=4pt]}-{Circle[length=4pt]}]  (1.6,0) -- (1.6,0);

\end{tikzpicture}
\caption{Schematic form of the spectrum of $\mathcal{L}_\mathcal{Y}$ for $\mwo=\mwt$: the continuous spectrum consists of two intervals and two discrete eigenvalues. 
The diagram is not to scale.}  \label{fig:specL2} 
\end{figure}
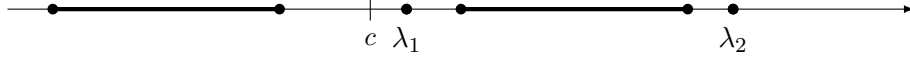

{\bf Outline.} 
The paper is organized as follows.

In Section 2, following~\cite{MR5101585}, we briefly review the necessary background on weighted graph coverings, Laplacians and heat kernels.

In Section 3, we study free products of finite groups of equal order.
After describing the associated Cayley graph and its projection onto a weighted half-line,
we determine the spectrum of the projected Laplacian and derive an explicit formula
for the heat kernel.

In Section 4, we study free products of two finite groups of unequal orders, allowing different positive weights on the two types of edges. By quotienting the Cayley graph by its symmetries, we obtain a weighted line whose projected Laplacian has periodic coefficients of period two. We determine its discrete and continuous spectra, obtain the spectral resolution, and use the covering relation to derive the heat kernel on the original Cayley graph.

\section{Preliminaries}\label{sec2}
In this section, we briefly review the necessary background. We recall the notions of weighted graphs and coverings that will be used throughout. In particular, we present the quotient construction of coverings via group actions, and the lifting relation for heat kernels under coverings.
We refer to our previous work \cite{MR5101585} for further details.

\subsection{Weighted graphs, Laplacians and heat kernels} 
\begin{definition}
 A \textit{weighted graph} is a set of \textit{vertices} $V$ equipped with a non-negative symmetric \textit{weight function} 
 $$
 w\colon V\times V\to \R_{\geq 0},\quad w(u,v)=w(v,u)\quad \forall u,v \in V.
 $$
 \end{definition}
 In a weighted graph, an \textit{edge} is a pair of vertices with strictly positive weight, and two vertices connected by an edge are called \textit{neighbours}. The weighted graph is called \emph{connected} if for any vertices $u$ and $v$, there exists a finite sequence of edges connecting $u$ and $v$. We also say that a weighted graph contains a \emph{loop} at $v$, if $w(v,v)>0$.
 
 The \textit{degree} of a vertex $u\in V$ of a weighted graph $(V,w)$ is defined as 
 $$
 d_u=\sum_{v\in V}w(u,v).
 $$
 Throughout this work, all graphs are assumed connected and locally finite.

The \emph{matrix coefficients} $A(u,v)$ of a linear map (operator) $A\colon \C^V \to \C^V$ are defined by 
$$
(Af)(u)=\sum_{v\in V}A(u,v)f(v).
$$

\begin{definition} \label{Laplacian}
 Given a weighted graph $(V,w)$, the \textit{normalized Laplacian} $\mathcal{L}$ of $(V,w)$ is an operator with the matrix coefficients  
\begin{equation}\label{normLaplacian}
  \mathcal{L}(u,v)=\delta_{u,v}-\frac{w(u,v)}{\sqrt{d_ud_v}}.
\end{equation}
\end{definition}
This is a self-adjoint (bounded) operator on the Hilbert space $\ell^2(V)$.
We write
\begin{equation}\label{operatorM}
 \mathcal{L}=I-\mathcal{M},
\end{equation}
where 
\begin{equation}\label{operatorM_mat}
 \mathcal{M}(u,v)=\frac{w(u,v)}{\sqrt{d_ud_v}}.
\end{equation}
The operators $\mathcal{L}$ and $\mathcal{M}$ have the same eigenvectors. If $\mu$ is an eigenvalue of $\mathcal{M}$, then $1-\mu$ is an eigenvalue of $\mathcal{L}$.

\begin{definition}\label{ht}
 Given a weighted graph $(V,w)$, the \textit{heat operator} $h_t$ of $(V,w)$ is an operator defined for $t\geq 0$ as 
 $$
 h_t=e^{-t\mathcal{L}}=\sum_{k=0}^{\infty}\frac{(-t)^k}{k!}\mathcal{L}^k.
 $$
 Its matrix coefficients $h_t(x,y)$ form the heat kernel.
\end{definition}
 
\begin{remark}
For a Cayley graph $\Gamma(G)$ of a group $G$, left invariance implies that
\begin{equation}\label{kernel}
 h_t(x,y)=k_t(y^{-1}x), 
\end{equation}
where $k_t(x):=h_t(x,e)$. Thus, the heat kernel is determined by the function $k_t$.
\end{remark}

\subsection{Coverings of weighted graphs}\label{sec3}
We recall the notion of covering of weighted graphs, following~\cite{MR1667452,MR5101585}.

\begin{definition}
Given two weighted graphs $(\tilde{V},\tilde{w})$, $(V,w)$, a map $\pi \colon \tilde{V} \to V$ is called a \textit{covering of weighted graphs} if $\pi(\tilde{V})=V$ and if there exists a function $\lambda: V \to \R_{>0}$ such that for any $u\in V$ and any $y\in \tilde{V}$, the following equation holds
 \begin{equation}\label{new_def}
\sum_{x\in \pi^{-1}(u)}\tilde{w}(x,y)=\lambda(\pi(y)) w(u,\pi(y)).
\end{equation}
We say that $(\tilde{V},\tilde{w})$ \textit{covers} $(V,w)$ through $\pi$. 
\end{definition} 

\begin{remark}\label{new_def_finite_fibers}
For a covering of (connected) weighted graphs $\pi \colon \tilde{V} \to V$, if $\pi^{-1}(u)$ is a finite set for any $u\in V$, then $\lambda(u)=\frac{c}{|\pi^{-1}(u)|}$ for some constant $c>0$.

Any topological covering is a covering of weighted graphs with $\lambda=1$. However, a covering of weighted graphs is not necessarily a topological covering since the preimages (fibers) of vertices can have different cardinalities.
\end{remark}

A more specific class of coverings will be especially useful for our purposes.

\begin{definition}
Given two weighted graphs $(\tilde{V},\tilde{w})$ and $(V,w)$, we say $(\tilde{V},\tilde{w})$ \textit{covers} $(V,w)$ \textit{strongly} and \textit{regularly} if there exists a vertex $u_0\in V$, called \textit{distinguished vertex}, such that, for any vertex $x\in \tilde{V}$ there exists a covering of weighted graphs $\pi_x \colon (\tilde{V},\tilde{w})\to (V,w)$ such that $\pi_x^{-1}(u_0)=\{x\}$.
\end{definition}

\subsection{Quotient weighted graphs}
An \emph{automorphism} of a weighted graph $(V,w)$ is a bijective map $\sigma\colon V\to V$ such that $w(\sigma(u),\sigma(v))=w(u,v)$ for any $u,v\in V$.
The construction of quotient weighted graphs is based on group actions by automorphisms of weighted graphs.

\begin{definition}
 Let $(\tilde{V},\tilde{w})$ be a weighted graph and $G\subset \operatorname{Aut}(\tilde{V},\tilde{w})$ a subgroup such that the orbit $Gx$ of $x$ is a finite set for any $x\in\tilde{V}$. The \textit{quotient weighted graph} of the weighted graph $(\tilde{V},\tilde{w})$ with respect to the group $G$ is a weighted graph $(V,w)$  defined as
 
\begin{equation} \label{qgraph}
 V=\tilde{V}/G, \quad w(u,v)=\sum_{\substack{x\in u\\ y\in v}}\tilde{w}(x,y) \quad \forall u,v\in V.
\end{equation}
\end{definition}

\begin{proposition}{\cite[Proposition 2.5]{MR5101585}}
\label{Prop2}
 Let $(\tilde{V},\tilde{w})$ be a weighted graph, $G\subset \operatorname{Aut}(\tilde{V},\tilde{w})$ a subgroup such that, for any $x\in\tilde{V}$, its orbit $Gx$  is a finite set  and $(V,w)$ the \textit{quotient weighted graph} of $(\tilde{V},\tilde{w})$ with respect to the group $G$. Then, the canonical projection map to the quotient space $\pi\colon\tilde{V}\to V$ is a covering of weighted graphs.
\end{proposition}

\subsection{Eigenfunctions and heat kernels under coverings}
The next propositions describe the lifting of eigenfunctions and the corresponding relation between heat kernels under coverings.
\begin{proposition}{\cite[Proposition 4.1]{MR5101585}}\label{efct_cov}
 Let $(\tilde{V},\tilde{w})$ cover $(V,w)$ through $\pi$ with associated function $\lambda: V \to \R_{>0}$, and assume that all fibers $\pi^{-1}(v)$ are finite. Let $\tilde{\mathcal{L}}$ and $\mathcal{L}$ denote the corresponding (normalized) Laplacians. If $f\in \ell^2(V)$ is an eigenfunction of $\mathcal{L}$ with  eigenvalue $\mu$, then the pull-back $\pi^{*}\sqrt{\lambda}f \in \ell^2(\tilde{V})$ is an eigenfunction of $\tilde{\mathcal{L}}$ with the same eigenvalue $\mu$. Explicitly, 
 $$
 (\pi^{*}\sqrt{\lambda}f)(x)=\sqrt{\lambda(\pi(x))}f(\pi(x)),
 $$
 and 
 the function $\lambda$ is given by
$$
\lambda(v)=\frac{c}{|\pi^{-1}(v)|}
$$
 for a constant $c>0$, as described in Remark~\ref{new_def_finite_fibers}.
\end{proposition}

\begin{proposition}{\cite[Proposition 4.2]{MR5101585}} \label{rel_covering_heat_kernel}
Let $(\tilde{V},\tilde{w})$ cover $(V,w)$ through $\pi$ with associated function $\lambda: V \to \R_{>0}$. Let $\tilde{h}_t$ and $h_t$ denote the corresponding heat kernels.
 Then, for any $u,v\in V$ and any $y\in\pi^{-1}(v)$, one has the equality
\begin{equation}\label{eq5}
\sum_{x\in\pi^{-1}(u)}\tilde{h}_t(x,y)=\sqrt{\frac{\lambda(v)}{\lambda(u)}}h_t(u,v).
\end{equation}
In particular, if the covering is strong and regular and $u_0\in V$ is a distinguished vertex with $\pi_x^{-1}(u_0)=\{x\}$, then, for any $v\in V$ and $y\in\pi_x^{-1}(v)$, we have
 \begin{equation}\label{new_hk}
 \tilde{h}_t(x,y)=\sqrt{\frac{\lambda(v)}{\lambda(u_0)}}h_t(u_0,v)=\sqrt{\frac{\lambda(\pi_x(y))}{\lambda(\pi_x(x))}}h_t(\pi_x(x),\pi_x(y)).
 \end{equation}
\end{proposition}

\begin{remark}
 If the covering $\pi \colon \tilde{V} \to V$ is strong and regular with finite fibers, then~\eqref{new_hk} becomes 
\begin{equation}\label{eq6}
 \tilde{h}_t(x,y)=\frac{1}{\sqrt{\vert\pi_x^{-1}(\pi_x(y))\vert}}h_t(\pi_x(x),\pi_x(y)).
\end{equation}
\end{remark}

\section{Free products of finite groups of equal order}\label{family1}
\subsection{The Cayley graph and its projection}\label{projection}
Let 
$$
\mathcal A=\mathcal A_{\mN}^{\mng}:=G_1*G_2*\cdots *G_{\mng},\quad r\ge2,\quad |G_1|=\dots=|G_{\mng}|=\mN\ge2,
$$ 
and let $\mathcal X =\mathcal X_{\mN}^{\mng}= \mathrm{Cay}(\mathcal A, S)$ denote the Cayley graph with respect to the generating set
$$
S = \bigsqcup_{i=1}^\mng \bigl(G_i \setminus \{e\}\bigr).
$$
All weights in the graph $\mathcal X$ are 1.
We illustrate $\mathcal X$ in Figure~\ref{fig:cayley_graph} in the case $\mN=5$ and $\mng=3$, and we show below that the graph $\mathcal X$ has a recursive symmetry structure.

In what follows, we implicitly fix for each pair of groups $G_i$ and $G_j$, an identification of the corresponding pointed complete graphs. 

At the identity vertex $e$, there are $\mng$ branches in $\mathcal X$, which start by complete graphs on $\mN$ vertices 
(corresponding to the $\mng$ free factors). 
The symmetric group $S_{\mng}$ acts on $\mathcal X$ fixing $e$ and permuting these $\mng$ branches.
Quotienting by this action, we obtain the first projection $\pi_1\colon \mathcal{X} \to \mathcal{X}_1$, under which the $\mng$ branches are identified, and we are left with a single branch starting with a complete graph containing $v_0=\pi_1(e)$, see Fig.~\ref{fig:step1}. 

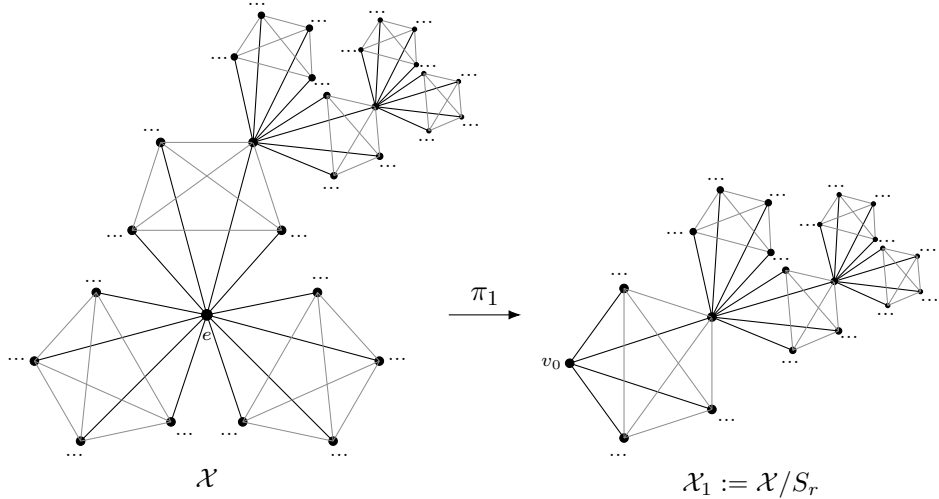
\begin{figure}[h]
\begin{tikzpicture}[scale=0.8]

\tikzset{
  treeedge/.style={line width=0.6pt},
  cliqueedge/.style={gray!90, line width=0.3pt}
}

\coordinate (e) at (0,0);

\node at (0,-2.7) {\small $\mathcal X$};

\coordinate (C1) at (90:1.8);

\coordinate (a0) at ($ (C1) + (270:1.3) $); 
\coordinate (a1) at ($ (C1) + (342:1.3) $); 
\coordinate (a2) at ($ (C1) + (54:1.3) $);
\coordinate (a3) at ($ (C1) + (126:1.3) $);
\coordinate (a4) at ($ (C1) + (198:1.3) $);

\coordinate (C2) at (210:1.8);

\coordinate (b0) at ($ (C2) + (30:1.3) $); 
\coordinate (b1) at ($ (C2) + (102:1.3) $); 
\coordinate (b2) at ($ (C2) + (174:1.3) $);
\coordinate (b3) at ($ (C2) + (246:1.3) $);
\coordinate (b4) at ($ (C2) + (318:1.3) $);

\coordinate (C3) at (330:1.8);

\coordinate (c0) at ($ (C3) + (150:1.3) $); 
\coordinate (c1) at ($ (C3) + (222:1.3) $); 
\coordinate (c2) at ($ (C3) + (294:1.3) $);
\coordinate (c3) at ($ (C3) + (6:1.3) $);
\coordinate (c4) at ($ (C3) + (78:1.3) $);

\draw[fill=black] (e) circle [radius=2.5pt];
\node at (0,-0.35) {\tiny$e$};

\draw[fill=black] (a1) circle [radius=2pt];
\draw[fill=black] (a2) circle [radius=2pt];
\draw[fill=black] (a3) circle [radius=2pt];
\draw[fill=black] (a4) circle [radius=2pt];

\draw(e)--(a1); \draw(e)--(a2); \draw(e)--(a3); \draw(e)--(a4);
\draw[cliqueedge](a1)--(a2); \draw[cliqueedge](a1)--(a3); \draw[cliqueedge](a1)--(a4); 
\draw[cliqueedge](a2)--(a3); \draw[cliqueedge](a2)--(a4); 
\draw[cliqueedge](a3)--(a4); 

\node at ($ (C1) + (342:1.6) $) {\tiny$...$};
\node at ($ (C1) + (126:1.55) $) {\tiny$...$};
\node at ($ (C1) + (198:1.6) $) {\tiny$...$};

\draw[fill=black] (b1) circle [radius=2pt];
\draw[fill=black] (b2) circle [radius=2pt];
\draw[fill=black] (b3) circle [radius=2pt];
\draw[fill=black] (b4) circle [radius=2pt];

\draw(e)--(b1); \draw(e)--(b2); \draw(e)--(b3); \draw(e)--(b4);
\draw[cliqueedge](b1)--(b2); \draw[cliqueedge](b1)--(b3); \draw[cliqueedge](b1)--(b4); 
\draw[cliqueedge](b2)--(b3); \draw[cliqueedge](b2)--(b4); 
\draw[cliqueedge](b3)--(b4); 

\node at ($ (C2) + (102:1.5) $) {\tiny$...$};
\node at ($ (C2) + (174:1.6) $) {\tiny$...$};
\node at ($ (C2) + (246:1.55) $) {\tiny$...$};
\node at ($ (C2) + (318:1.6) $) {\tiny$...$};

\draw[fill=black] (c1) circle [radius=2pt];
\draw[fill=black] (c2) circle [radius=2pt];
\draw[fill=black] (c3) circle [radius=2pt];
\draw[fill=black] (c4) circle [radius=2pt];

\draw(e)--(c1); \draw(e)--(c2); \draw(e)--(c3); \draw(e)--(c4);
\draw[cliqueedge](c1)--(c2); \draw[cliqueedge](c1)--(c3); \draw[cliqueedge](c1)--(c4); 
\draw[cliqueedge](c2)--(c3); \draw[cliqueedge](c2)--(c4); 
\draw[cliqueedge](c3)--(c4); 

\node at ($ (C3) + (222:1.6) $) {\tiny$...$};
\node at ($ (C3) + (294:1.55) $) {\tiny$...$};
\node at ($ (C3) + (6:1.6) $) {\tiny$...$};
\node at ($ (C3) + (78:1.55) $) {\tiny$...$};

\coordinate (Attach) at (a2);

\coordinate (D1) at ($ (Attach) + (5:1.5) $);
\coordinate (D2) at ($ (Attach) + (75:1.5) $);


\coordinate (d1_0) at ($ (D1) + (185:0.7) $);
\coordinate (d1_1) at ($ (D1) + (257:0.7) $);
\coordinate (d1_2) at ($ (D1) + (329:0.7) $);
\coordinate (d1_3) at ($ (D1) + (41:0.7) $);
\coordinate (d1_4) at ($ (D1) + (113:0.7) $);

\draw[fill=black] (d1_1) circle [radius=1.5pt];
\draw[fill=black] (d1_2) circle [radius=1.5pt];
\draw[fill=black] (d1_3) circle [radius=1.5pt];
\draw[fill=black] (d1_4) circle [radius=1.5pt];

\draw(Attach)--(d1_1); \draw(Attach)--(d1_2); \draw(Attach)--(d1_3); \draw(Attach)--(d1_4);
\draw[cliqueedge](d1_1)--(d1_2); \draw[cliqueedge](d1_1)--(d1_3); \draw[cliqueedge](d1_1)--(d1_4); 
\draw[cliqueedge](d1_2)--(d1_3); \draw[cliqueedge](d1_2)--(d1_4); 
\draw[cliqueedge](d1_3)--(d1_4); 

\node at ($ (D1) + (257:0.9) $) {\tiny$...$};
\node at ($ (D1) + (325:0.9) $) {\tiny$...$};
\node at ($ (D1) + (113:0.9) $) {\tiny$...$};

\coordinate (d2_0) at ($ (D2) + (255:0.7) $);
\coordinate (d2_1) at ($ (D2) + (327:0.7) $);
\coordinate (d2_2) at ($ (D2) + (39:0.7) $);
\coordinate (d2_3) at ($ (D2) + (111:0.7) $);
\coordinate (d2_4) at ($ (D2) + (183:0.7) $);

\draw(Attach)--(d2_1); \draw(Attach)--(d2_2); \draw(Attach)--(d2_3); \draw(Attach)--(d2_4);
\draw[cliqueedge](d2_1)--(d2_2); \draw[cliqueedge](d2_1)--(d2_3); \draw[cliqueedge](d2_1)--(d2_4); 
\draw[cliqueedge](d2_2)--(d2_3); \draw[cliqueedge](d2_2)--(d2_4); 
\draw[cliqueedge](d2_3)--(d2_4); 

\node at ($ (D2) + (39:0.9) $) {\tiny$...$};
\node at ($ (D2) + (111:0.9) $) {\tiny$...$};
\node at ($ (D2) + (183:0.95) $) {\tiny$...$};

\draw[fill=black] (d2_1) circle [radius=1.5pt];
\draw[fill=black] (d2_2) circle [radius=1.5pt];
\draw[fill=black] (d2_3) circle [radius=1.5pt];
\draw[fill=black] (d2_4) circle [radius=1.5pt];

\coordinate (Attach2) at (d1_3);

\coordinate (E1) at ($ (Attach2) + (5:1) $);
\coordinate (E2) at ($ (Attach2) + (75:1) $);


\coordinate (e1_0) at ($ (E1) + (185:0.5) $);
\coordinate (e1_1) at ($ (E1) + (257:0.5) $);
\coordinate (e1_2) at ($ (E1) + (329:0.5) $);
\coordinate (e1_3) at ($ (E1) + (41:0.5) $);
\coordinate (e1_4) at ($ (E1) + (113:0.5) $);

\draw[fill=black] (e1_1) circle [radius=1pt];
\draw[fill=black] (e1_2) circle [radius=1pt];
\draw[fill=black] (e1_3) circle [radius=1pt];
\draw[fill=black] (e1_4) circle [radius=1pt];

\draw(Attach2)--(e1_1); \draw(Attach2)--(e1_2); \draw(Attach2)--(e1_3); \draw(Attach2)--(e1_4);
\draw[cliqueedge](e1_1)--(e1_2); \draw[cliqueedge](e1_1)--(e1_3); \draw[cliqueedge](e1_1)--(e1_4); 
\draw[cliqueedge](e1_2)--(e1_3); \draw[cliqueedge](e1_2)--(e1_4); 
\draw[cliqueedge](e1_3)--(e1_4); 

\node at ($ (E1) + (257:0.65) $) {\tiny$...$};
\node at ($ (E1) + (326:0.7) $) {\tiny$...$};
\node at ($ (E1) + (38:0.7) $) {\tiny$...$};
\node at ($ (E1) + (102:0.6) $) {\tiny$...$};

\coordinate (e2_0) at ($ (E2) + (255:0.5) $);
\coordinate (e2_1) at ($ (E2) + (327:0.5) $);
\coordinate (e2_2) at ($ (E2) + (39:0.5) $);
\coordinate (e2_3) at ($ (E2) + (111:0.5) $);
\coordinate (e2_4) at ($ (E2) + (183:0.5) $);

\draw(Attach2)--(e2_1); \draw(Attach2)--(e2_2); \draw(Attach2)--(e2_3); \draw(Attach2)--(e2_4);
\draw[cliqueedge](e2_1)--(e2_2); \draw[cliqueedge](e2_1)--(e2_3); \draw[cliqueedge](e2_1)--(e2_4); 
\draw[cliqueedge](e2_2)--(e2_3); \draw[cliqueedge](e2_2)--(e2_4); 
\draw[cliqueedge](e2_3)--(e2_4); 

\draw[fill=black] (e2_1) circle [radius=1pt];
\draw[fill=black] (e2_2) circle [radius=1pt];
\draw[fill=black] (e2_3) circle [radius=1pt];
\draw[fill=black] (e2_4) circle [radius=1pt];

\node at ($ (E2) + (39:0.7) $) {\tiny$...$};
\node at ($ (E2) + (111:0.65) $) {\tiny$...$};
\node at ($ (E2) + (183:0.7) $) {\tiny$...$};

\draw[-{Latex}] (4,0) -- node[above]{$\pi_1$}(5.2,0);

\begin{scope}[shift={(6,-0.8)}]

\node at (3,-2) {\small$\mathcal{X}_1:=\mathcal{X}/S_{\mng}$};

\coordinate (e_q1) at (0,0);
\draw[fill=black] (e_q1) circle [radius=2pt];
\node at (-0.3,0) {\tiny$v_0$};

\coordinate (C_q1) at (0:1.3);

\coordinate (a0_q1) at ($ (C_q1) + (180:1.3) $); 
\coordinate (a1_q1) at ($ (C_q1) + (252:1.3) $); 
\coordinate (a2_q1) at ($ (C_q1) + (324:1.3) $);
\coordinate (a3_q1) at ($ (C_q1) + (36:1.3) $);
\coordinate (a4_q1) at ($ (C_q1) + (108:1.3) $);

\draw[fill=black] (a1_q1) circle [radius=2pt];
\draw[fill=black] (a2_q1) circle [radius=2pt];
\draw[fill=black] (a3_q1) circle [radius=2pt];
\draw[fill=black] (a4_q1) circle [radius=2pt];

\draw(e_q1)--(a1_q1); \draw(e_q1)--(a2_q1); \draw(e_q1)--(a3_q1); \draw(e_q1)--(a4_q1);
\draw[cliqueedge](a1_q1)--(a2_q1); \draw[cliqueedge](a1_q1)--(a3_q1); \draw[cliqueedge](a1_q1)--(a4_q1); 
\draw[cliqueedge](a2_q1)--(a3_q1); \draw[cliqueedge](a2_q1)--(a4_q1); 
\draw[cliqueedge](a3_q1)--(a4_q1); 

\node at ($ (C_q1) + (252:1.6) $) {\tiny$...$};
\node at ($ (C_q1) + (324:1.6) $) {\tiny$...$};
\node at ($ (C_q1) + (108:1.55) $) {\tiny$...$};

\coordinate (Attach_q1) at (a3_q1);

\coordinate (D1_q1) at ($ (Attach_q1) + (5:1.5) $);
\coordinate (D2_q1) at ($ (Attach_q1) + (75:1.5) $);


\coordinate (d1_q1_0) at ($ (D1_q1) + (185:0.7) $);
\coordinate (d1_q1_1) at ($ (D1_q1) + (257:0.7) $);
\coordinate (d1_q1_2) at ($ (D1_q1) + (329:0.7) $);
\coordinate (d1_q1_3) at ($ (D1_q1) + (41:0.7) $);
\coordinate (d1_q1_4) at ($ (D1_q1) + (113:0.7) $);

\draw[fill=black] (d1_q1_1) circle [radius=1.5pt];
\draw[fill=black] (d1_q1_2) circle [radius=1.5pt];
\draw[fill=black] (d1_q1_3) circle [radius=1.5pt];
\draw[fill=black] (d1_q1_4) circle [radius=1.5pt];

\draw(Attach_q1)--(d1_q1_1); \draw(Attach_q1)--(d1_q1_2); \draw(Attach_q1)--(d1_q1_3); \draw(Attach_q1)--(d1_q1_4);
\draw[cliqueedge](d1_q1_1)--(d1_q1_2); \draw[cliqueedge](d1_q1_1)--(d1_q1_3); \draw[cliqueedge](d1_q1_1)--(d1_q1_4); 
\draw[cliqueedge](d1_q1_2)--(d1_q1_3); \draw[cliqueedge](d1_q1_2)--(d1_q1_4); 
\draw[cliqueedge](d1_q1_3)--(d1_q1_4); 

\node at ($ (D1_q1) + (257:0.9) $) {\tiny$...$};
\node at ($ (D1_q1) + (329:0.9) $) {\tiny$...$};

\coordinate (d2_q1_0) at ($ (D2_q1) + (255:0.7) $);
\coordinate (d2_q1_1) at ($ (D2_q1) + (327:0.7) $);
\coordinate (d2_q1_2) at ($ (D2_q1) + (39:0.7) $);
\coordinate (d2_q1_3) at ($ (D2_q1) + (111:0.7) $);
\coordinate (d2_q1_4) at ($ (D2_q1) + (183:0.7) $);

\draw(Attach_q1)--(d2_q1_1); \draw(Attach_q1)--(d2_q1_2); \draw(Attach_q1)--(d2_q1_3); \draw(Attach_q1)--(d2_q1_4);
\draw[cliqueedge](d2_q1_1)--(d2_q1_2); \draw[cliqueedge](d2_q1_1)--(d2_q1_3); \draw[cliqueedge](d2_q1_1)--(d2_q1_4); 
\draw[cliqueedge](d2_q1_2)--(d2_q1_3); \draw[cliqueedge](d2_q1_2)--(d2_q1_4); 
\draw[cliqueedge](d2_q1_3)--(d2_q1_4); 

\draw[fill=black] (d2_q1_1) circle [radius=1.5pt];
\draw[fill=black] (d2_q1_2) circle [radius=1.5pt];
\draw[fill=black] (d2_q1_3) circle [radius=1.5pt];
\draw[fill=black] (d2_q1_4) circle [radius=1.5pt];

\node at ($ (D2_q1) + (327:0.9) $) {\tiny$...$};
\node at ($ (D2_q1) + (39:0.9) $) {\tiny$...$};
\node at ($ (D2_q1) + (111:0.9) $) {\tiny$...$};
\node at ($ (D2_q1) + (183:1) $) {\tiny$...$};

\coordinate (Attach2_q1) at (d1_q1_3);

\coordinate (E1_q1) at ($ (Attach2_q1) + (5:1) $);
\coordinate (E2_q1) at ($ (Attach2_q1) + (75:1) $);


\coordinate (e1_q1_0) at ($ (E1_q1) + (185:0.5) $);
\coordinate (e1_q1_1) at ($ (E1_q1) + (257:0.5) $);
\coordinate (e1_q1_2) at ($ (E1_q1) + (329:0.5) $);
\coordinate (e1_q1_3) at ($ (E1_q1) + (41:0.5) $);
\coordinate (e1_q1_4) at ($ (E1_q1) + (113:0.5) $);

\draw[fill=black] (e1_q1_1) circle [radius=1pt];
\draw[fill=black] (e1_q1_2) circle [radius=1pt];
\draw[fill=black] (e1_q1_3) circle [radius=1pt];
\draw[fill=black] (e1_q1_4) circle [radius=1pt];

\draw(Attach2_q1)--(e1_q1_1); \draw(Attach2_q1)--(e1_q1_2); \draw(Attach2_q1)--(e1_q1_3); \draw(Attach2_q1)--(e1_q1_4);
\draw[cliqueedge](e1_q1_1)--(e1_q1_2); \draw[cliqueedge](e1_q1_1)--(e1_q1_3); \draw[cliqueedge](e1_q1_1)--(e1_q1_4); 
\draw[cliqueedge](e1_q1_2)--(e1_q1_3); \draw[cliqueedge](e1_q1_2)--(e1_q1_4); 
\draw[cliqueedge](e1_q1_3)--(e1_q1_4); 

\node at ($ (E1_q1) + (257:0.65) $) {\tiny$...$};
\node at ($ (E1_q1) + (326:0.7) $) {\tiny$...$};
\node at ($ (E1_q1) + (38:0.7) $) {\tiny$...$};
\node at ($ (E1_q1) + (102:0.6) $) {\tiny$...$};

\coordinate (e2_q1_0) at ($ (E2_q1) + (255:0.5) $);
\coordinate (e2_q1_1) at ($ (E2_q1) + (327:0.5) $);
\coordinate (e2_q1_2) at ($ (E2_q1) + (39:0.5) $);
\coordinate (e2_q1_3) at ($ (E2_q1) + (111:0.5) $);
\coordinate (e2_q1_4) at ($ (E2_q1) + (183:0.5) $);

\draw(Attach2_q1)--(e2_q1_1); \draw(Attach2_q1)--(e2_q1_2); \draw(Attach2_q1)--(e2_q1_3); \draw(Attach2_q1)--(e2_q1_4);
\draw[cliqueedge](e2_q1_1)--(e2_q1_2); \draw[cliqueedge](e2_q1_1)--(e2_q1_3); \draw[cliqueedge](e2_q1_1)--(e2_q1_4); 
\draw[cliqueedge](e2_q1_2)--(e2_q1_3); \draw[cliqueedge](e2_q1_2)--(e2_q1_4); 
\draw[cliqueedge](e2_q1_3)--(e2_q1_4); 

\draw[fill=black] (e2_q1_1) circle [radius=1pt];
\draw[fill=black] (e2_q1_2) circle [radius=1pt];
\draw[fill=black] (e2_q1_3) circle [radius=1pt];
\draw[fill=black] (e2_q1_4) circle [radius=1pt];

\node at ($ (E2_q1) + (39:0.7) $) {\tiny$...$};
\node at ($ (E2_q1) + (111:0.65) $) {\tiny$...$};
\node at ($ (E2_q1) + (183:0.7) $) {\tiny$...$};

\end{scope}

\end{tikzpicture}
\caption{Covering $\pi_1\colon \mathcal{X} \to \mathcal{X}_1$, with $v_0=\pi_1(e)$. All edges in $\mathcal{X}$ (respectively $\mathcal{X}_1$) have weight $w_0=1$ (respectively $w_1=\mng$).} \label{fig:step1}
\end{figure}

Starting from $\mathcal{X}_1$, we construct recursively projections $\pi_{i}\colon \mathcal{X}_{i-1}\to \mathcal{X}_{i}$ for all $i\ge2$, where
$$
\mathcal{X}_i=
\begin{cases} 
   \mathcal{X}_{i-1}/S_{\mN-1} &\text{if $i$ is even} \\
   \mathcal{X}_{i-1}/S_{r-1} &\text{if $i$ is odd},
\end{cases} 
$$
and each $\mathcal{X}_i$ contains vertices $v_0,v_1,\dots, v_{\lfloor i/2 \rfloor }$ organised linearly, in the sense that the only neighbours of $v_j$ for $j<\lfloor i/2 \rfloor$ are $v_1$ if $j=0$, and $v_{j\pm1}$ if $j>0$. $\mathcal{X}_i$ also contains loops at $v_1,\dots, v_{\lfloor i/2 \rfloor }$.

For $i\ge1$, the symmetric group $S_{\mN-1}$ acts on $\mathcal{X}_{2i-1}$, fixing $v_0,\dots,v_{i-1}$ and permuting all $\mN-1$ neighbours of $v_{i-1}$, except $v_{i-2}$, together with the branches attached to them, see Fig.~\ref{fig:step_even}.  

For $i\ge1$, the symmetric group $S_{r-1}$ acts on $\mathcal{X}_{2i}$, fixing $v_0,\dots,v_{i}$ and permuting the $r-1$ branches
starting with complete graphs containing $v_i$, see Fig.~\ref{fig:step_odd}.

\begin{figure}[h]
\begin{tikzpicture}[scale=1]

\tikzset{
  treeedge/.style={line width=0.6pt},
  cliqueedge/.style={gray!90, line width=0.3pt}
}

\node at (-1.2,0) {\small$\mathcal{X}_{2i-1}$};

\coordinate (p0) at (0,0);
\coordinate (p1) at ($(p0) + (1.8,0) $);
\coordinate (pd1) at ($(p1) + (0.65,0) $); 
\coordinate (pd2) at ($(pd1) + (0.65,0) $); 
\coordinate (p_i-2) at ($(pd2) + (0.65,0) $);

\node at (0,-0.3) {\tiny$v_0$};
\node at (1.8,-0.3) {\tiny$v_1$};
\node at (2.8,0) {\tiny$\dots$};
\node at (3.75,-0.3) {\tiny$v_{i-1}$};

\draw[fill=black] (p0) circle [radius=2pt];
\draw[fill=black] (p1) circle [radius=2pt];
\draw[fill=black] (p_i-2) circle [radius=2pt];

\draw(p0)--node[above]{\tiny{\color{black} $w_2$}}(p1)--(pd1); \draw(pd2)--(p_i-2); 

\makeatletter
\tikzset{my loop/.style =  {to path={
  \pgfextra{\let\tikztotarget=\tikztostart}
  [looseness=6,min distance=10mm]
  \tikz@to@curve@path}
  }}  
\makeatletter 
\path (p1) edge[my loop] node[above] {\tiny{$(\mN-2)w_2$}} (p1);
\path (p_i-2) edge[my loop] node[above] {\tiny{$(\mN-2)w_{2i-2}$}} (p_i-2);

\coordinate (C_q1) at ($ (p_i-2) + (0:1.8) $);

\coordinate (a0_q1) at ($ (C_q1) + (180:0.8) $); 
\coordinate (a1_q1) at ($ (C_q1) + (252:0.8) $); 
\coordinate (a2_q1) at ($ (C_q1) + (324:0.8) $);
\coordinate (a3_q1) at ($ (C_q1) + (36:0.8) $);
\coordinate (a4_q1) at ($ (C_q1) + (108:0.8) $);

\draw[fill=black] (a1_q1) circle [radius=2pt];
\draw[fill=black] (a2_q1) circle [radius=2pt];
\draw[fill=black] (a3_q1) circle [radius=2pt];
\draw[fill=black] (a4_q1) circle [radius=2pt];

\draw(p_i-2)--(a1_q1); \draw(p_i-2)--(a2_q1); \draw(p_i-2)--(a3_q1); \draw(p_i-2)--(a4_q1);
\draw[cliqueedge](a1_q1)--(a2_q1); \draw[cliqueedge](a1_q1)--(a3_q1); \draw[cliqueedge](a1_q1)--(a4_q1); 
\draw[cliqueedge](a2_q1)--(a3_q1); \draw[cliqueedge](a2_q1)--(a4_q1); 
\draw[cliqueedge](a3_q1)--(a4_q1); 

\node at ($ (C_q1) + (252:1.1) $) {\tiny$...$};
\node at ($ (C_q1) + (324:1.1) $) {\tiny$...$};
\node at ($ (C_q1) + (108:1.05) $) {\tiny$...$};

\coordinate (Attach_q1) at (a3_q1);

\coordinate (D1_q1) at ($ (Attach_q1) + (5:1.5) $);
\coordinate (D2_q1) at ($ (Attach_q1) + (75:1.5) $);


\coordinate (d1_q1_0) at ($ (D1_q1) + (185:0.7) $);
\coordinate (d1_q1_1) at ($ (D1_q1) + (257:0.7) $);
\coordinate (d1_q1_2) at ($ (D1_q1) + (329:0.7) $);
\coordinate (d1_q1_3) at ($ (D1_q1) + (41:0.7) $);
\coordinate (d1_q1_4) at ($ (D1_q1) + (113:0.7) $);

\draw[fill=black] (d1_q1_1) circle [radius=1.5pt];
\draw[fill=black] (d1_q1_2) circle [radius=1.5pt];
\draw[fill=black] (d1_q1_3) circle [radius=1.5pt];
\draw[fill=black] (d1_q1_4) circle [radius=1.5pt];

\draw(Attach_q1)--(d1_q1_1); \draw(Attach_q1)--(d1_q1_2); \draw(Attach_q1)--(d1_q1_3); \draw(Attach_q1)--(d1_q1_4);
\draw[cliqueedge](d1_q1_1)--(d1_q1_2); \draw[cliqueedge](d1_q1_1)--(d1_q1_3); \draw[cliqueedge](d1_q1_1)--(d1_q1_4); 
\draw[cliqueedge](d1_q1_2)--(d1_q1_3); \draw[cliqueedge](d1_q1_2)--(d1_q1_4); 
\draw[cliqueedge](d1_q1_3)--(d1_q1_4); 

\node at ($ (D1_q1) + (257:0.9) $) {\tiny$...$};
\node at ($ (D1_q1) + (329:0.9) $) {\tiny$...$};

\coordinate (d2_q1_0) at ($ (D2_q1) + (255:0.7) $);
\coordinate (d2_q1_1) at ($ (D2_q1) + (327:0.7) $);
\coordinate (d2_q1_2) at ($ (D2_q1) + (39:0.7) $);
\coordinate (d2_q1_3) at ($ (D2_q1) + (111:0.7) $);
\coordinate (d2_q1_4) at ($ (D2_q1) + (183:0.7) $);

\draw(Attach_q1)--(d2_q1_1); \draw(Attach_q1)--(d2_q1_2); \draw(Attach_q1)--(d2_q1_3); \draw(Attach_q1)--(d2_q1_4);
\draw[cliqueedge](d2_q1_1)--(d2_q1_2); \draw[cliqueedge](d2_q1_1)--(d2_q1_3); \draw[cliqueedge](d2_q1_1)--(d2_q1_4); 
\draw[cliqueedge](d2_q1_2)--(d2_q1_3); \draw[cliqueedge](d2_q1_2)--(d2_q1_4); 
\draw[cliqueedge](d2_q1_3)--(d2_q1_4); 

\draw[fill=black] (d2_q1_1) circle [radius=1.5pt];
\draw[fill=black] (d2_q1_2) circle [radius=1.5pt];
\draw[fill=black] (d2_q1_3) circle [radius=1.5pt];
\draw[fill=black] (d2_q1_4) circle [radius=1.5pt];

\node at ($ (D2_q1) + (327:0.9) $) {\tiny$...$};
\node at ($ (D2_q1) + (39:0.9) $) {\tiny$...$};
\node at ($ (D2_q1) + (111:0.9) $) {\tiny$...$};
\node at ($ (D2_q1) + (183:1) $) {\tiny$...$};

\coordinate (Attach2_q1) at (d1_q1_3);

\coordinate (E1_q1) at ($ (Attach2_q1) + (5:1) $);
\coordinate (E2_q1) at ($ (Attach2_q1) + (75:1) $);


\coordinate (e1_q1_0) at ($ (E1_q1) + (185:0.5) $);
\coordinate (e1_q1_1) at ($ (E1_q1) + (257:0.5) $);
\coordinate (e1_q1_2) at ($ (E1_q1) + (329:0.5) $);
\coordinate (e1_q1_3) at ($ (E1_q1) + (41:0.5) $);
\coordinate (e1_q1_4) at ($ (E1_q1) + (113:0.5) $);

\draw[fill=black] (e1_q1_1) circle [radius=1pt];
\draw[fill=black] (e1_q1_2) circle [radius=1pt];
\draw[fill=black] (e1_q1_3) circle [radius=1pt];
\draw[fill=black] (e1_q1_4) circle [radius=1pt];

\draw(Attach2_q1)--(e1_q1_1); \draw(Attach2_q1)--(e1_q1_2); \draw(Attach2_q1)--(e1_q1_3); \draw(Attach2_q1)--(e1_q1_4);
\draw[cliqueedge](e1_q1_1)--(e1_q1_2); \draw[cliqueedge](e1_q1_1)--(e1_q1_3); \draw[cliqueedge](e1_q1_1)--(e1_q1_4); 
\draw[cliqueedge](e1_q1_2)--(e1_q1_3); \draw[cliqueedge](e1_q1_2)--(e1_q1_4); 
\draw[cliqueedge](e1_q1_3)--(e1_q1_4); 

\node at ($ (E1_q1) + (257:0.65) $) {\tiny$...$};
\node at ($ (E1_q1) + (326:0.7) $) {\tiny$...$};
\node at ($ (E1_q1) + (38:0.7) $) {\tiny$...$};
\node at ($ (E1_q1) + (102:0.6) $) {\tiny$...$};

\coordinate (e2_q1_0) at ($ (E2_q1) + (255:0.5) $);
\coordinate (e2_q1_1) at ($ (E2_q1) + (327:0.5) $);
\coordinate (e2_q1_2) at ($ (E2_q1) + (39:0.5) $);
\coordinate (e2_q1_3) at ($ (E2_q1) + (111:0.5) $);
\coordinate (e2_q1_4) at ($ (E2_q1) + (183:0.5) $);

\draw(Attach2_q1)--(e2_q1_1); \draw(Attach2_q1)--(e2_q1_2); \draw(Attach2_q1)--(e2_q1_3); \draw(Attach2_q1)--(e2_q1_4);
\draw[cliqueedge](e2_q1_1)--(e2_q1_2); \draw[cliqueedge](e2_q1_1)--(e2_q1_3); \draw[cliqueedge](e2_q1_1)--(e2_q1_4); 
\draw[cliqueedge](e2_q1_2)--(e2_q1_3); \draw[cliqueedge](e2_q1_2)--(e2_q1_4); 
\draw[cliqueedge](e2_q1_3)--(e2_q1_4); 

\draw[fill=black] (e2_q1_1) circle [radius=1pt];
\draw[fill=black] (e2_q1_2) circle [radius=1pt];
\draw[fill=black] (e2_q1_3) circle [radius=1pt];
\draw[fill=black] (e2_q1_4) circle [radius=1pt];

\node at ($ (E2_q1) + (39:0.7) $) {\tiny$...$};
\node at ($ (E2_q1) + (111:0.65) $) {\tiny$...$};
\node at ($ (E2_q1) + (183:0.7) $) {\tiny$...$};

\draw[-{Latex}] (0,-1.3) --(0,-2.5);
\node at (-0.4,-1.9) {$\pi_{2i}$};

\begin{scope}[shift={(0,-3.5)}]

\node at (-2,0) {\small$\mathcal{X}_{2i}=\mathcal{X}_{2i-1}/S_{\mN-1}$};

\coordinate (p0) at (0,0);
\coordinate (p1) at ($(p0) + (1.8,0) $);
\coordinate (pd1) at ($(p1) + (0.65,0) $); 
\coordinate (pd2) at ($(pd1) + (0.65,0) $); 
\coordinate (p_i-2) at ($(pd2) + (0.65,0) $);
\coordinate (p_i-1) at ($(p_i-2) + (1.8,0) $); 

\node at (0,-0.3) {\tiny$v_0$};
\node at (1.8,-0.3) {\tiny$v_1$};
\node at (2.8,0) {\tiny$\dots$};
\node at (3.75,-0.3) {\tiny$v_{i-1}$};
\node at (5.3,-0.3) {\tiny$v_{i}$};

\draw[fill=black] (p0) circle [radius=2pt];
\draw[fill=black] (p1) circle [radius=2pt];
\draw[fill=black] (p_i-2) circle [radius=2pt];
\draw[fill=black] (p_i-1) circle [radius=2pt];

\draw(p0)--node[above]{\tiny{\color{black} $w_2$}}(p1)--(pd1); \draw(pd2)--(p_i-2)--node[above]{\tiny{\color{black} $w_{2i}$}}(p_i-1); 

\makeatletter
\tikzset{my loop/.style =  {to path={
  \pgfextra{\let\tikztotarget=\tikztostart}
  [looseness=6,min distance=10mm]
  \tikz@to@curve@path}
  }}  
\makeatletter 
\path (p1) edge[my loop] node[above] {\tiny{$(\mN-2)w_2$}} (p1);
\path (p_i-2) edge[my loop] node[above] {\tiny{$(\mN-2)w_{2i-2}$}} (p_i-2);
\path (p_i-1) edge[my loop] node[above] {\tiny{$(\mN-2)w_{2i}$}} (p_i-1);

\coordinate (Attach_q2) at (p_i-1);

\coordinate (D1_q2) at ($ (Attach_q2) + (0:1.5) $);
\coordinate (D2_q2) at ($ (Attach_q2) + (-75:1.5) $);


\coordinate (d1_q2_0) at ($ (D1_q2) + (180:0.7) $);
\coordinate (d1_q2_1) at ($ (D1_q2) + (252:0.7) $);
\coordinate (d1_q2_2) at ($ (D1_q2) + (324:0.7) $);
\coordinate (d1_q2_3) at ($ (D1_q2) + (36:0.7) $);
\coordinate (d1_q2_4) at ($ (D1_q2) + (108:0.7) $);

\draw[fill=black] (d1_q2_1) circle [radius=1.5pt];
\draw[fill=black] (d1_q2_2) circle [radius=1.5pt];
\draw[fill=black] (d1_q2_3) circle [radius=1.5pt];
\draw[fill=black] (d1_q2_4) circle [radius=1.5pt];

\draw(Attach_q2)--(d1_q2_1); \draw(Attach_q2)--(d1_q2_2); \draw(Attach_q2)--(d1_q2_3); \draw(Attach_q2)--(d1_q2_4);
\draw[cliqueedge](d1_q2_1)--(d1_q2_2); \draw[cliqueedge](d1_q2_1)--(d1_q2_3); \draw[cliqueedge](d1_q2_1)--(d1_q2_4); 
\draw[cliqueedge](d1_q2_2)--(d1_q2_3); \draw[cliqueedge](d1_q2_2)--(d1_q2_4); 
\draw[cliqueedge](d1_q2_3)--(d1_q2_4); 

\node at ($ (D1_q2) + (252:0.9) $) {\tiny$...$};
\node at ($ (D1_q2) + (324:0.9) $) {\tiny$...$};

\coordinate (d2_q2_0) at ($ (D2_q2) + (-255:0.7) $);
\coordinate (d2_q2_1) at ($ (D2_q2) + (-327:0.7) $);
\coordinate (d2_q2_2) at ($ (D2_q2) + (-39:0.7) $);
\coordinate (d2_q2_3) at ($ (D2_q2) + (-111:0.7) $);
\coordinate (d2_q2_4) at ($ (D2_q2) + (-183:0.7) $);

\draw(Attach_q2)--(d2_q2_1); \draw(Attach_q2)--(d2_q2_2); \draw(Attach_q2)--(d2_q2_3); \draw(Attach_q2)--(d2_q2_4);
\draw[cliqueedge](d2_q2_1)--(d2_q2_2); \draw[cliqueedge](d2_q2_1)--(d2_q2_3); \draw[cliqueedge](d2_q2_1)--(d2_q2_4); 
\draw[cliqueedge](d2_q2_2)--(d2_q2_3); \draw[cliqueedge](d2_q2_2)--(d2_q2_4); 
\draw[cliqueedge](d2_q2_3)--(d2_q2_4); 

\draw[fill=black] (d2_q2_1) circle [radius=1.5pt];
\draw[fill=black] (d2_q2_2) circle [radius=1.5pt];
\draw[fill=black] (d2_q2_3) circle [radius=1.5pt];
\draw[fill=black] (d2_q2_4) circle [radius=1.5pt];

\node at ($ (D2_q2) + (-39:0.9) $) {\tiny$...$};
\node at ($ (D2_q2) + (-111:0.9) $) {\tiny$...$};
\node at ($ (D2_q2) + (-183:1) $) {\tiny$...$};

\coordinate (Attach2_q2) at (d1_q2_3);

\coordinate (E1_q2) at ($ (Attach2_q2) + (5:1) $);
\coordinate (E2_q2) at ($ (Attach2_q2) + (75:1) $);


\coordinate (e1_q2_0) at ($ (E1_q2) + (185:0.5) $);
\coordinate (e1_q2_1) at ($ (E1_q2) + (257:0.5) $);
\coordinate (e1_q2_2) at ($ (E1_q2) + (329:0.5) $);
\coordinate (e1_q2_3) at ($ (E1_q2) + (41:0.5) $);
\coordinate (e1_q2_4) at ($ (E1_q2) + (113:0.5) $);

\draw[fill=black] (e1_q2_1) circle [radius=1pt];
\draw[fill=black] (e1_q2_2) circle [radius=1pt];
\draw[fill=black] (e1_q2_3) circle [radius=1pt];
\draw[fill=black] (e1_q2_4) circle [radius=1pt];

\draw(Attach2_q2)--(e1_q2_1); \draw(Attach2_q2)--(e1_q2_2); \draw(Attach2_q2)--(e1_q2_3); \draw(Attach2_q2)--(e1_q2_4);
\draw[cliqueedge](e1_q2_1)--(e1_q2_2); \draw[cliqueedge](e1_q2_1)--(e1_q2_3); \draw[cliqueedge](e1_q2_1)--(e1_q2_4); 
\draw[cliqueedge](e1_q2_2)--(e1_q2_3); \draw[cliqueedge](e1_q2_2)--(e1_q2_4); 
\draw[cliqueedge](e1_q2_3)--(e1_q2_4); 

\node at ($ (E1_q2) + (257:0.65) $) {\tiny$...$};
\node at ($ (E1_q2) + (326:0.7) $) {\tiny$...$};
\node at ($ (E1_q2) + (38:0.7) $) {\tiny$...$};
\node at ($ (E1_q2) + (102:0.6) $) {\tiny$...$};

\coordinate (e2_q2_0) at ($ (E2_q2) + (255:0.5) $);
\coordinate (e2_q2_1) at ($ (E2_q2) + (327:0.5) $);
\coordinate (e2_q2_2) at ($ (E2_q2) + (39:0.5) $);
\coordinate (e2_q2_3) at ($ (E2_q2) + (111:0.5) $);
\coordinate (e2_q2_4) at ($ (E2_q2) + (183:0.5) $);

\draw(Attach2_q2)--(e2_q2_1); \draw(Attach2_q2)--(e2_q2_2); \draw(Attach2_q2)--(e2_q2_3); \draw(Attach2_q2)--(e2_q2_4);
\draw[cliqueedge](e2_q2_1)--(e2_q2_2); \draw[cliqueedge](e2_q2_1)--(e2_q2_3); \draw[cliqueedge](e2_q2_1)--(e2_q2_4); 
\draw[cliqueedge](e2_q2_2)--(e2_q2_3); \draw[cliqueedge](e2_q2_2)--(e2_q2_4); 
\draw[cliqueedge](e2_q2_3)--(e2_q2_4); 

\draw[fill=black] (e2_q2_1) circle [radius=1pt];
\draw[fill=black] (e2_q2_2) circle [radius=1pt];
\draw[fill=black] (e2_q2_3) circle [radius=1pt];
\draw[fill=black] (e2_q2_4) circle [radius=1pt];

\node at ($ (E2_q2) + (39:0.7) $) {\tiny$...$};
\node at ($ (E2_q2) + (111:0.65) $) {\tiny$...$};
\node at ($ (E2_q2) + (183:0.7) $) {\tiny$...$};

\end{scope}
\end{tikzpicture}

\caption{Covering $\pi_{2i}\colon \mathcal{X}_{2i-1}\to\mathcal{X}_{2i}$, where $v_j=\pi_i(v_j)$. All unlabelled edges in $\mathcal{X}_{2i}$ have weight $w_{2i}=(\mN-1)w_{2i-1}$.
} \label{fig:step_even}
\end{figure}

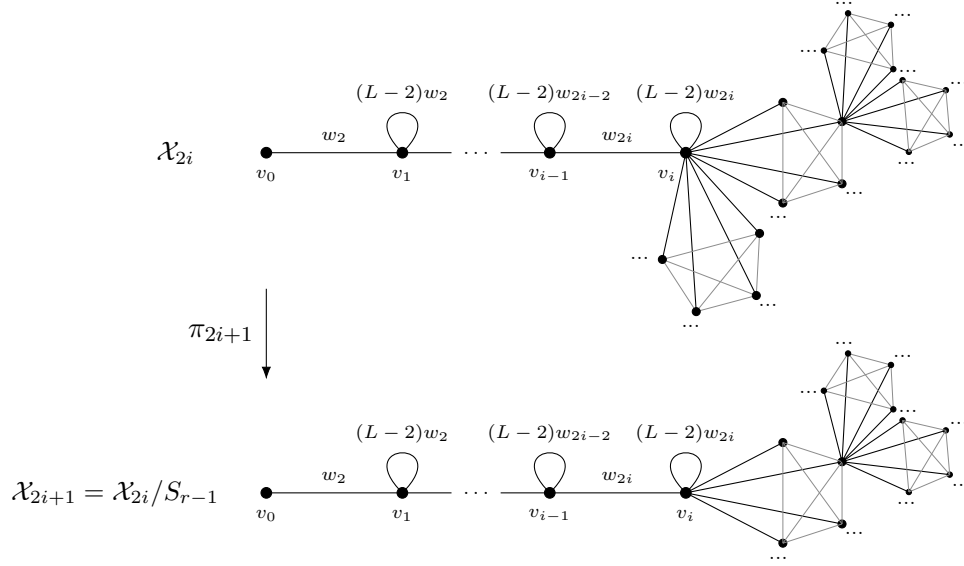
\begin{figure}[h]
\begin{tikzpicture}[scale=1]

\tikzset{
  treeedge/.style={line width=0.6pt},
  cliqueedge/.style={gray!90, line width=0.3pt}
}

\node at (-1.2,0) {\small$\mathcal{X}_{2i}$};

\coordinate (p0) at (0,0);
\coordinate (p1) at ($(p0) + (1.8,0) $);
\coordinate (pd1) at ($(p1) + (0.65,0) $); 
\coordinate (pd2) at ($(pd1) + (0.65,0) $); 
\coordinate (p_i-2) at ($(pd2) + (0.65,0) $);
\coordinate (p_i-1) at ($(p_i-2) + (1.8,0) $); 

\node at (0,-0.3) {\tiny$v_0$};
\node at (1.8,-0.3) {\tiny$v_1$};
\node at (2.8,0) {\tiny$\dots$};
\node at (3.75,-0.3) {\tiny$v_{i-1}$};
\node at (5.3,-0.3) {\tiny$v_{i}$};

\draw[fill=black] (p0) circle [radius=2pt];
\draw[fill=black] (p1) circle [radius=2pt];
\draw[fill=black] (p_i-2) circle [radius=2pt];
\draw[fill=black] (p_i-1) circle [radius=2pt];

\draw(p0)--node[above]{\tiny{\color{black} $w_2$}}(p1)--(pd1); \draw(pd2)--(p_i-2)--node[above]{\tiny{\color{black} $w_{2i}$}}(p_i-1); 

\makeatletter
\tikzset{my loop/.style =  {to path={
  \pgfextra{\let\tikztotarget=\tikztostart}
  [looseness=6,min distance=10mm]
  \tikz@to@curve@path}
  }}  
\makeatletter 
\path (p1) edge[my loop] node[above] {\tiny{$(\mN-2)w_2$}} (p1);
\path (p_i-2) edge[my loop] node[above] {\tiny{$(\mN-2)w_{2i-2}$}} (p_i-2);
\path (p_i-1) edge[my loop] node[above] {\tiny{$(\mN-2)w_{2i}$}} (p_i-1);

\coordinate (Attach_q2) at (p_i-1);

\coordinate (D1_q2) at ($ (Attach_q2) + (0:1.5) $);
\coordinate (D2_q2) at ($ (Attach_q2) + (-75:1.5) $);


\coordinate (d1_q2_0) at ($ (D1_q2) + (180:0.7) $);
\coordinate (d1_q2_1) at ($ (D1_q2) + (252:0.7) $);
\coordinate (d1_q2_2) at ($ (D1_q2) + (324:0.7) $);
\coordinate (d1_q2_3) at ($ (D1_q2) + (36:0.7) $);
\coordinate (d1_q2_4) at ($ (D1_q2) + (108:0.7) $);

\draw[fill=black] (d1_q2_1) circle [radius=1.5pt];
\draw[fill=black] (d1_q2_2) circle [radius=1.5pt];
\draw[fill=black] (d1_q2_3) circle [radius=1.5pt];
\draw[fill=black] (d1_q2_4) circle [radius=1.5pt];

\draw(Attach_q2)--(d1_q2_1); \draw(Attach_q2)--(d1_q2_2); \draw(Attach_q2)--(d1_q2_3); \draw(Attach_q2)--(d1_q2_4);
\draw[cliqueedge](d1_q2_1)--(d1_q2_2); \draw[cliqueedge](d1_q2_1)--(d1_q2_3); \draw[cliqueedge](d1_q2_1)--(d1_q2_4); 
\draw[cliqueedge](d1_q2_2)--(d1_q2_3); \draw[cliqueedge](d1_q2_2)--(d1_q2_4); 
\draw[cliqueedge](d1_q2_3)--(d1_q2_4); 

\node at ($ (D1_q2) + (252:0.9) $) {\tiny$...$};
\node at ($ (D1_q2) + (324:0.9) $) {\tiny$...$};

\coordinate (d2_q2_0) at ($ (D2_q2) + (-255:0.7) $);
\coordinate (d2_q2_1) at ($ (D2_q2) + (-327:0.7) $);
\coordinate (d2_q2_2) at ($ (D2_q2) + (-39:0.7) $);
\coordinate (d2_q2_3) at ($ (D2_q2) + (-111:0.7) $);
\coordinate (d2_q2_4) at ($ (D2_q2) + (-183:0.7) $);

\draw(Attach_q2)--(d2_q2_1); \draw(Attach_q2)--(d2_q2_2); \draw(Attach_q2)--(d2_q2_3); \draw(Attach_q2)--(d2_q2_4);
\draw[cliqueedge](d2_q2_1)--(d2_q2_2); \draw[cliqueedge](d2_q2_1)--(d2_q2_3); \draw[cliqueedge](d2_q2_1)--(d2_q2_4); 
\draw[cliqueedge](d2_q2_2)--(d2_q2_3); \draw[cliqueedge](d2_q2_2)--(d2_q2_4); 
\draw[cliqueedge](d2_q2_3)--(d2_q2_4); 

\draw[fill=black] (d2_q2_1) circle [radius=1.5pt];
\draw[fill=black] (d2_q2_2) circle [radius=1.5pt];
\draw[fill=black] (d2_q2_3) circle [radius=1.5pt];
\draw[fill=black] (d2_q2_4) circle [radius=1.5pt];

\node at ($ (D2_q2) + (-39:0.9) $) {\tiny$...$};
\node at ($ (D2_q2) + (-111:0.9) $) {\tiny$...$};
\node at ($ (D2_q2) + (-183:1) $) {\tiny$...$};

\coordinate (Attach2_q2) at (d1_q2_3);

\coordinate (E1_q2) at ($ (Attach2_q2) + (5:1) $);
\coordinate (E2_q2) at ($ (Attach2_q2) + (75:1) $);


\coordinate (e1_q2_0) at ($ (E1_q2) + (185:0.5) $);
\coordinate (e1_q2_1) at ($ (E1_q2) + (257:0.5) $);
\coordinate (e1_q2_2) at ($ (E1_q2) + (329:0.5) $);
\coordinate (e1_q2_3) at ($ (E1_q2) + (41:0.5) $);
\coordinate (e1_q2_4) at ($ (E1_q2) + (113:0.5) $);

\draw[fill=black] (e1_q2_1) circle [radius=1pt];
\draw[fill=black] (e1_q2_2) circle [radius=1pt];
\draw[fill=black] (e1_q2_3) circle [radius=1pt];
\draw[fill=black] (e1_q2_4) circle [radius=1pt];

\draw(Attach2_q2)--(e1_q2_1); \draw(Attach2_q2)--(e1_q2_2); \draw(Attach2_q2)--(e1_q2_3); \draw(Attach2_q2)--(e1_q2_4);
\draw[cliqueedge](e1_q2_1)--(e1_q2_2); \draw[cliqueedge](e1_q2_1)--(e1_q2_3); \draw[cliqueedge](e1_q2_1)--(e1_q2_4); 
\draw[cliqueedge](e1_q2_2)--(e1_q2_3); \draw[cliqueedge](e1_q2_2)--(e1_q2_4); 
\draw[cliqueedge](e1_q2_3)--(e1_q2_4); 

\node at ($ (E1_q2) + (257:0.65) $) {\tiny$...$};
\node at ($ (E1_q2) + (326:0.7) $) {\tiny$...$};
\node at ($ (E1_q2) + (38:0.7) $) {\tiny$...$};
\node at ($ (E1_q2) + (102:0.6) $) {\tiny$...$};

\coordinate (e2_q2_0) at ($ (E2_q2) + (255:0.5) $);
\coordinate (e2_q2_1) at ($ (E2_q2) + (327:0.5) $);
\coordinate (e2_q2_2) at ($ (E2_q2) + (39:0.5) $);
\coordinate (e2_q2_3) at ($ (E2_q2) + (111:0.5) $);
\coordinate (e2_q2_4) at ($ (E2_q2) + (183:0.5) $);

\draw(Attach2_q2)--(e2_q2_1); \draw(Attach2_q2)--(e2_q2_2); \draw(Attach2_q2)--(e2_q2_3); \draw(Attach2_q2)--(e2_q2_4);
\draw[cliqueedge](e2_q2_1)--(e2_q2_2); \draw[cliqueedge](e2_q2_1)--(e2_q2_3); \draw[cliqueedge](e2_q2_1)--(e2_q2_4); 
\draw[cliqueedge](e2_q2_2)--(e2_q2_3); \draw[cliqueedge](e2_q2_2)--(e2_q2_4); 
\draw[cliqueedge](e2_q2_3)--(e2_q2_4); 

\draw[fill=black] (e2_q2_1) circle [radius=1pt];
\draw[fill=black] (e2_q2_2) circle [radius=1pt];
\draw[fill=black] (e2_q2_3) circle [radius=1pt];
\draw[fill=black] (e2_q2_4) circle [radius=1pt];

\node at ($ (E2_q2) + (39:0.7) $) {\tiny$...$};
\node at ($ (E2_q2) + (111:0.65) $) {\tiny$...$};
\node at ($ (E2_q2) + (183:0.7) $) {\tiny$...$};

\draw[-{Latex}] (0,-1.8) --(0,-3);
\node at (-0.6,-2.4) {$\pi_{2i+1}$};

\begin{scope}[shift={(0,-4.5)}]
\node at (-2,0) {\small$\mathcal{X}_{2i+1}=\mathcal{X}_{2i}/S_{r-1}$};

\coordinate (p0) at (0,0);
\coordinate (p1) at ($(p0) + (1.8,0) $);
\coordinate (pd1) at ($(p1) + (0.65,0) $); 
\coordinate (pd2) at ($(pd1) + (0.65,0) $); 
\coordinate (p_i-2) at ($(pd2) + (0.65,0) $);
\coordinate (p_i-1) at ($(p_i-2) + (1.8,0) $); 

\node at (0,-0.3) {\tiny$v_0$};
\node at (1.8,-0.3) {\tiny$v_1$};
\node at (2.8,0) {\tiny$\dots$};
\node at (3.75,-0.3) {\tiny$v_{i-1}$};
\node at (5.55,-0.3) {\tiny$v_{i}$};

\draw[fill=black] (p0) circle [radius=2pt];
\draw[fill=black] (p1) circle [radius=2pt];
\draw[fill=black] (p_i-2) circle [radius=2pt];
\draw[fill=black] (p_i-1) circle [radius=2pt];

\draw(p0)--node[above]{\tiny{\color{black} $w_2$}}(p1)--(pd1); \draw(pd2)--(p_i-2)--node[above]{\tiny{\color{black} $w_{2i}$}}(p_i-1); 

\makeatletter
\tikzset{my loop/.style =  {to path={
  \pgfextra{\let\tikztotarget=\tikztostart}
  [looseness=6,min distance=10mm]
  \tikz@to@curve@path}
  }}  
\makeatletter 
\path (p1) edge[my loop] node[above] {\tiny{$(\mN-2)w_2$}} (p1);
\path (p_i-2) edge[my loop] node[above] {\tiny{$(\mN-2)w_{2i-2}$}} (p_i-2);
\path (p_i-1) edge[my loop] node[above] {\tiny{$(\mN-2)w_{2i}$}} (p_i-1);

\coordinate (Attach_q2) at (p_i-1);

\coordinate (D1_q2) at ($ (Attach_q2) + (0:1.5) $);
\coordinate (D2_q2) at ($ (Attach_q2) + (-75:1.5) $);


\coordinate (d1_q2_0) at ($ (D1_q2) + (180:0.7) $);
\coordinate (d1_q2_1) at ($ (D1_q2) + (252:0.7) $);
\coordinate (d1_q2_2) at ($ (D1_q2) + (324:0.7) $);
\coordinate (d1_q2_3) at ($ (D1_q2) + (36:0.7) $);
\coordinate (d1_q2_4) at ($ (D1_q2) + (108:0.7) $);

\draw[fill=black] (d1_q2_1) circle [radius=1.5pt];
\draw[fill=black] (d1_q2_2) circle [radius=1.5pt];
\draw[fill=black] (d1_q2_3) circle [radius=1.5pt];
\draw[fill=black] (d1_q2_4) circle [radius=1.5pt];

\draw(Attach_q2)--(d1_q2_1); \draw(Attach_q2)--(d1_q2_2); \draw(Attach_q2)--(d1_q2_3); \draw(Attach_q2)--(d1_q2_4);
\draw[cliqueedge](d1_q2_1)--(d1_q2_2); \draw[cliqueedge](d1_q2_1)--(d1_q2_3); \draw[cliqueedge](d1_q2_1)--(d1_q2_4); 
\draw[cliqueedge](d1_q2_2)--(d1_q2_3); \draw[cliqueedge](d1_q2_2)--(d1_q2_4); 
\draw[cliqueedge](d1_q2_3)--(d1_q2_4); 

\node at ($ (D1_q2) + (252:0.9) $) {\tiny$...$};
\node at ($ (D1_q2) + (324:0.9) $) {\tiny$...$};

\coordinate (Attach2_q2) at (d1_q2_3);

\coordinate (E1_q2) at ($ (Attach2_q2) + (5:1) $);
\coordinate (E2_q2) at ($ (Attach2_q2) + (75:1) $);


\coordinate (e1_q2_0) at ($ (E1_q2) + (185:0.5) $);
\coordinate (e1_q2_1) at ($ (E1_q2) + (257:0.5) $);
\coordinate (e1_q2_2) at ($ (E1_q2) + (329:0.5) $);
\coordinate (e1_q2_3) at ($ (E1_q2) + (41:0.5) $);
\coordinate (e1_q2_4) at ($ (E1_q2) + (113:0.5) $);

\draw[fill=black] (e1_q2_1) circle [radius=1pt];
\draw[fill=black] (e1_q2_2) circle [radius=1pt];
\draw[fill=black] (e1_q2_3) circle [radius=1pt];
\draw[fill=black] (e1_q2_4) circle [radius=1pt];

\draw(Attach2_q2)--(e1_q2_1); \draw(Attach2_q2)--(e1_q2_2); \draw(Attach2_q2)--(e1_q2_3); \draw(Attach2_q2)--(e1_q2_4);
\draw[cliqueedge](e1_q2_1)--(e1_q2_2); \draw[cliqueedge](e1_q2_1)--(e1_q2_3); \draw[cliqueedge](e1_q2_1)--(e1_q2_4); 
\draw[cliqueedge](e1_q2_2)--(e1_q2_3); \draw[cliqueedge](e1_q2_2)--(e1_q2_4); 
\draw[cliqueedge](e1_q2_3)--(e1_q2_4); 

\node at ($ (E1_q2) + (257:0.65) $) {\tiny$...$};
\node at ($ (E1_q2) + (326:0.7) $) {\tiny$...$};
\node at ($ (E1_q2) + (38:0.7) $) {\tiny$...$};
\node at ($ (E1_q2) + (102:0.6) $) {\tiny$...$};

\coordinate (e2_q2_0) at ($ (E2_q2) + (255:0.5) $);
\coordinate (e2_q2_1) at ($ (E2_q2) + (327:0.5) $);
\coordinate (e2_q2_2) at ($ (E2_q2) + (39:0.5) $);
\coordinate (e2_q2_3) at ($ (E2_q2) + (111:0.5) $);
\coordinate (e2_q2_4) at ($ (E2_q2) + (183:0.5) $);

\draw(Attach2_q2)--(e2_q2_1); \draw(Attach2_q2)--(e2_q2_2); \draw(Attach2_q2)--(e2_q2_3); \draw(Attach2_q2)--(e2_q2_4);
\draw[cliqueedge](e2_q2_1)--(e2_q2_2); \draw[cliqueedge](e2_q2_1)--(e2_q2_3); \draw[cliqueedge](e2_q2_1)--(e2_q2_4); 
\draw[cliqueedge](e2_q2_2)--(e2_q2_3); \draw[cliqueedge](e2_q2_2)--(e2_q2_4); 
\draw[cliqueedge](e2_q2_3)--(e2_q2_4); 

\draw[fill=black] (e2_q2_1) circle [radius=1pt];
\draw[fill=black] (e2_q2_2) circle [radius=1pt];
\draw[fill=black] (e2_q2_3) circle [radius=1pt];
\draw[fill=black] (e2_q2_4) circle [radius=1pt];

\node at ($ (E2_q2) + (39:0.7) $) {\tiny$...$};
\node at ($ (E2_q2) + (111:0.65) $) {\tiny$...$};
\node at ($ (E2_q2) + (183:0.7) $) {\tiny$...$};

\end{scope}
\end{tikzpicture}

\caption{
Covering $\pi_{2i+1}\colon \mathcal{X}_{2i} \to \mathcal{X}_{2i+1}$, where $v_j=\pi_{2i+1}(v_j)$. All unlabelled edges in $\mathcal{X}_{2i+1}$ have weight $w_{2i+1}=(r-1)w_{2i}$.} \label{fig:step_odd}
\end{figure}

\begin{figure}[h]
\begin{tikzpicture}[scale=1]

\tikzset{
  treeedge/.style={line width=0.6pt},
  cliqueedge/.style={gray!90, line width=0.3pt}
}

\coordinate (e) at (0,0);
\node at (2.7,-1.5) {\small $\mathcal{X}$};

\begin{scope}[scale=0.8, transform shape]

\coordinate (C1) at (90:1.8);

\coordinate (a0) at ($ (C1) + (270:1.3) $); 
\coordinate (a1) at ($ (C1) + (342:1.3) $); 
\coordinate (a2) at ($ (C1) + (54:1.3) $);
\coordinate (a3) at ($ (C1) + (126:1.3) $);
\coordinate (a4) at ($ (C1) + (198:1.3) $);

\coordinate (C2) at (210:1.8);

\coordinate (b0) at ($ (C2) + (30:1.3) $); 
\coordinate (b1) at ($ (C2) + (102:1.3) $); 
\coordinate (b2) at ($ (C2) + (174:1.3) $);
\coordinate (b3) at ($ (C2) + (246:1.3) $);
\coordinate (b4) at ($ (C2) + (318:1.3) $);

\coordinate (C3) at (330:1.8);

\coordinate (c0) at ($ (C3) + (150:1.3) $); 
\coordinate (c1) at ($ (C3) + (222:1.3) $); 
\coordinate (c2) at ($ (C3) + (294:1.3) $);
\coordinate (c3) at ($ (C3) + (6:1.3) $);
\coordinate (c4) at ($ (C3) + (78:1.3) $);

\draw[fill=black] (e) circle [radius=2.5pt];
\node at (0,-0.35) {\tiny$e$};

\draw[fill=black] (a1) circle [radius=2pt];
\draw[fill=black] (a2) circle [radius=2pt];
\draw[fill=black] (a3) circle [radius=2pt];
\draw[fill=black] (a4) circle [radius=2pt];

\draw(e)--(a1); \draw(e)--(a2); \draw(e)--(a3); \draw(e)--(a4);
\draw[cliqueedge](a1)--(a2); \draw[cliqueedge](a1)--(a3); \draw[cliqueedge](a1)--(a4); 
\draw[cliqueedge](a2)--(a3); \draw[cliqueedge](a2)--(a4); 
\draw[cliqueedge](a3)--(a4); 

\node at ($ (C1) + (342:1.6) $) {\tiny$...$};
\node at ($ (C1) + (126:1.55) $) {\tiny$...$};
\node at ($ (C1) + (198:1.6) $) {\tiny$...$};

\draw[fill=black] (b1) circle [radius=2pt];
\draw[fill=black] (b2) circle [radius=2pt];
\draw[fill=black] (b3) circle [radius=2pt];
\draw[fill=black] (b4) circle [radius=2pt];

\draw(e)--(b1); \draw(e)--(b2); \draw(e)--(b3); \draw(e)--(b4);
\draw[cliqueedge](b1)--(b2); \draw[cliqueedge](b1)--(b3); \draw[cliqueedge](b1)--(b4); 
\draw[cliqueedge](b2)--(b3); \draw[cliqueedge](b2)--(b4); 
\draw[cliqueedge](b3)--(b4); 

\node at ($ (C2) + (102:1.5) $) {\tiny$...$};
\node at ($ (C2) + (174:1.6) $) {\tiny$...$};
\node at ($ (C2) + (246:1.55) $) {\tiny$...$};
\node at ($ (C2) + (318:1.6) $) {\tiny$...$};

\draw[fill=black] (c1) circle [radius=2pt];
\draw[fill=black] (c2) circle [radius=2pt];
\draw[fill=black] (c3) circle [radius=2pt];
\draw[fill=black] (c4) circle [radius=2pt];

\draw(e)--(c1); \draw(e)--(c2); \draw(e)--(c3); \draw(e)--(c4);
\draw[cliqueedge](c1)--(c2); \draw[cliqueedge](c1)--(c3); \draw[cliqueedge](c1)--(c4); 
\draw[cliqueedge](c2)--(c3); \draw[cliqueedge](c2)--(c4); 
\draw[cliqueedge](c3)--(c4); 

\node at ($ (C3) + (222:1.6) $) {\tiny$...$};
\node at ($ (C3) + (294:1.55) $) {\tiny$...$};
\node at ($ (C3) + (6:1.6) $) {\tiny$...$};
\node at ($ (C3) + (78:1.55) $) {\tiny$...$};


\coordinate (Attach) at (a2);

\coordinate (D1) at ($ (Attach) + (5:1.5) $);
\coordinate (D2) at ($ (Attach) + (75:1.5) $);


\coordinate (d1_0) at ($ (D1) + (185:0.7) $);
\coordinate (d1_1) at ($ (D1) + (257:0.7) $);
\coordinate (d1_2) at ($ (D1) + (329:0.7) $);
\coordinate (d1_3) at ($ (D1) + (41:0.7) $);
\coordinate (d1_4) at ($ (D1) + (113:0.7) $);

\draw[fill=black] (d1_1) circle [radius=1.5pt];
\draw[fill=black] (d1_2) circle [radius=1.5pt];
\draw[fill=black] (d1_3) circle [radius=1.5pt];
\draw[fill=black] (d1_4) circle [radius=1.5pt];

\draw(Attach)--(d1_1); \draw(Attach)--(d1_2); \draw(Attach)--(d1_3); \draw(Attach)--(d1_4);
\draw[cliqueedge](d1_1)--(d1_2); \draw[cliqueedge](d1_1)--(d1_3); \draw[cliqueedge](d1_1)--(d1_4); 
\draw[cliqueedge](d1_2)--(d1_3); \draw[cliqueedge](d1_2)--(d1_4); 
\draw[cliqueedge](d1_3)--(d1_4); 

\node at ($ (D1) + (257:0.9) $) {\tiny$...$};
\node at ($ (D1) + (325:0.9) $) {\tiny$...$};
\node at ($ (D1) + (113:0.9) $) {\tiny$...$};

\coordinate (d2_0) at ($ (D2) + (255:0.7) $);
\coordinate (d2_1) at ($ (D2) + (327:0.7) $);
\coordinate (d2_2) at ($ (D2) + (39:0.7) $);
\coordinate (d2_3) at ($ (D2) + (111:0.7) $);
\coordinate (d2_4) at ($ (D2) + (183:0.7) $);

\draw(Attach)--(d2_1); \draw(Attach)--(d2_2); \draw(Attach)--(d2_3); \draw(Attach)--(d2_4);
\draw[cliqueedge](d2_1)--(d2_2); \draw[cliqueedge](d2_1)--(d2_3); \draw[cliqueedge](d2_1)--(d2_4); 
\draw[cliqueedge](d2_2)--(d2_3); \draw[cliqueedge](d2_2)--(d2_4); 
\draw[cliqueedge](d2_3)--(d2_4); 

\node at ($ (D2) + (39:0.9) $) {\tiny$...$};
\node at ($ (D2) + (111:0.9) $) {\tiny$...$};
\node at ($ (D2) + (183:0.95) $) {\tiny$...$};

\draw[fill=black] (d2_1) circle [radius=1.5pt];
\draw[fill=black] (d2_2) circle [radius=1.5pt];
\draw[fill=black] (d2_3) circle [radius=1.5pt];
\draw[fill=black] (d2_4) circle [radius=1.5pt];

\coordinate (Attach2) at (d1_3);

\coordinate (E1) at ($ (Attach2) + (5:1) $);
\coordinate (E2) at ($ (Attach2) + (75:1) $);


\coordinate (e1_0) at ($ (E1) + (185:0.5) $);
\coordinate (e1_1) at ($ (E1) + (257:0.5) $);
\coordinate (e1_2) at ($ (E1) + (329:0.5) $);
\coordinate (e1_3) at ($ (E1) + (41:0.5) $);
\coordinate (e1_4) at ($ (E1) + (113:0.5) $);

\draw[fill=black] (e1_1) circle [radius=1pt];
\draw[fill=black] (e1_2) circle [radius=1pt];
\draw[fill=black] (e1_3) circle [radius=1pt];
\draw[fill=black] (e1_4) circle [radius=1pt];

\draw(Attach2)--(e1_1); \draw(Attach2)--(e1_2); \draw(Attach2)--(e1_3); \draw(Attach2)--(e1_4);
\draw[cliqueedge](e1_1)--(e1_2); \draw[cliqueedge](e1_1)--(e1_3); \draw[cliqueedge](e1_1)--(e1_4); 
\draw[cliqueedge](e1_2)--(e1_3); \draw[cliqueedge](e1_2)--(e1_4); 
\draw[cliqueedge](e1_3)--(e1_4); 

\node at ($ (E1) + (257:0.65) $) {\tiny$...$};
\node at ($ (E1) + (326:0.7) $) {\tiny$...$};
\node at ($ (E1) + (38:0.7) $) {\tiny$...$};
\node at ($ (E1) + (102:0.6) $) {\tiny$...$};

\coordinate (e2_0) at ($ (E2) + (255:0.5) $);
\coordinate (e2_1) at ($ (E2) + (327:0.5) $);
\coordinate (e2_2) at ($ (E2) + (39:0.5) $);
\coordinate (e2_3) at ($ (E2) + (111:0.5) $);
\coordinate (e2_4) at ($ (E2) + (183:0.5) $);

\draw(Attach2)--(e2_1); \draw(Attach2)--(e2_2); \draw(Attach2)--(e2_3); \draw(Attach2)--(e2_4);
\draw[cliqueedge](e2_1)--(e2_2); \draw[cliqueedge](e2_1)--(e2_3); \draw[cliqueedge](e2_1)--(e2_4); 
\draw[cliqueedge](e2_2)--(e2_3); \draw[cliqueedge](e2_2)--(e2_4); 
\draw[cliqueedge](e2_3)--(e2_4); 

\draw[fill=black] (e2_1) circle [radius=1pt];
\draw[fill=black] (e2_2) circle [radius=1pt];
\draw[fill=black] (e2_3) circle [radius=1pt];
\draw[fill=black] (e2_4) circle [radius=1pt];

\node at ($ (E2) + (39:0.7) $) {\tiny$...$};
\node at ($ (E2) + (111:0.65) $) {\tiny$...$};
\node at ($ (E2) + (183:0.7) $) {\tiny$...$};

\end{scope}

\draw[-{Latex}] (3.5,0) -- node[above]{$\pi_e$}(4.7,0);


\begin{scope}[shift={(5.5,0)}]

\node at (2.5,-1.5) {$\whl$};

\coordinate (p0) at (0,0);
\coordinate (p1) at ($ (p0) + (1.5,0) $); 
\coordinate (p2) at ($ (p1) + (1.5,0) $); 
\coordinate (p3) at ($ (p2) + (1.5,0) $); 
\coordinate (etc) at ($ (p3) + (1.2,0) $); 

\node at ($ (p0) + (0,-0.4) $) {\small $0$};
\node at ($ (p1) + (0,-0.4) $) {\small $1$};
\node at ($ (p2) + (0,-0.4) $) {\small $2$};
\node at ($ (p3) + (0,-0.4) $) {\small $3$};

\draw[fill=black] (p0) circle [radius=2pt];
\draw[fill=black] (p1) circle [radius=2pt];
\draw[fill=black] (p2) circle [radius=2pt];
\draw[fill=black] (p3) circle [radius=2pt];

\draw(p0)--node[above]{\tiny{\color{black} $w_2$}}(p1); 
\draw(p1)--node[above]{\tiny{\color{black} $w_4$}}(p2); 
\draw(p2)--node[above]{\tiny{\color{black} $w_6$}}(p3); 
\draw(p3)--(etc); 


\makeatletter
\tikzset{my loop/.style =  {to path={
  \pgfextra{\let\tikztotarget=\tikztostart}
  [looseness=6,min distance=10mm]
  \tikz@to@curve@path}
  }}  
\makeatletter 

\path (p1) edge[my loop] node[above] {\tiny{\color{black} $(\mN-2)w_2$}} (p1);
\path (p2) edge[my loop] node[above] {\tiny{\color{black} $(\mN-2)w_4$}} (p2);
\path (p3) edge[my loop] node[above] {\tiny{\color{black} $(\mN-2)w_6$}} (p3);

\node at (6.2,0) {$\dots$};
\end{scope}
\end{tikzpicture}
\caption{$\mathcal{X}$ covers the weighted path $\whl$ strongly and regularly. The distinguished vertex is $0$.} \label{fig:covering}
\end{figure}

By  iterating this procedure infinitely, all vertices of $\mathcal{X}$ at distance $i$ from $e$ are identified with a single vertex $v_i$. The resulting quotient graph is a weighted half-line $\whl$, which the sequence of graphs $(\mathcal{X}_i)_{i\ge1}$ stabilizes to, and
 the vertex set $\{v_i\mid i\in \Z_{\ge0}\}$ of $\whl$ is naturally identified with $\Z_{\ge0}$. 
\begin{proposition}
 The weight function $w=w_{\whl}$ of $\whl$ is given by, for $0\le m\le n$,
 \begin{equation}\label{weight_fct}
 w(m,n)=w(n,m)=
\begin{cases}
 r(\mN-1)^n(r-1)^{n-1} & \text{if } m=n-1 \\
 r(\mN-1)^n(r-1)^{n-1}(\mN-2) & \text{if } m=n\ge1 \\
 0 & \text{otherwise}.
\end{cases}
\end{equation}
\end{proposition}
\begin{proof}
From Figs.~\ref{fig:step_even}--\ref{fig:covering}, we have $w_1=r$ and the recurrence equations, for $i\ge1$,
\begin{equation}\label{w1_rec}
 w_{2i}=(\mN-1)w_{2i-1},\quad w_{2i+1}=(r-1)w_{2i}, 
\end{equation}
and the equalities $w(0,0)=0$ and, for $i\ge1$,
\begin{equation}\label{w2_rec}
 w(i-1,i)=w_{2i}, \quad w(i,i)=(\mN-2)w_{2i}.
\end{equation}
Using \eqref{w1_rec}, we obtain, for $i\ge2$,
\begin{align*}
 w_{2i}=(\mN-1)(r-1)w_{2(i-1)},
\end{align*}
with the solution 
$$
w_{2i}=(\mN-1)^{i-1}(r-1)^{i-1}w_2.
$$
From \eqref{w1_rec}, we also have $w_2=(\mN-1)w_1=r(\mN-1)$,
so that 
$$
w_{2i}=r(\mN-1)^{i}(r-1)^{i-1}.
$$
Using this in~\eqref{w2_rec}, we  complete the proof.
\end{proof}

 As the vertex set of $\mathcal{X}$ is $\mathcal{A}$ and that of $\whl$ is $\Z_{\ge0}$, the covering map $\pi_e\colon \mathcal{X} \to \whl$ is identical to the map $|\cdot|\colon \mathcal{A} \to \Z_{\ge0}$, where $|x|$ is the word length of $x\in \mathcal{A}$. 
 
Proposition~\ref{Prop2} implies that $\mathcal{X}$ covers $\whl$. Since $\mathcal{X}$ is vertex-transitive, any vertex can be chosen to map to $0$. Hence $X$ covers $\whl$ strongly and regularly, where the distinguished vertex is $0$, see Fig.~\ref{fig:covering}.

\subsection{The projected Laplacian}
Let $\mathcal{L}=\mathcal{L}_{\whl}$ be the normalized Laplacian on $\whl$ and 
let $\pi_e\colon \mathcal{A} \to \Z_{\ge0}$ be the covering defined above.
By definition of the covering, see Section~\ref{projection}, we have
\begin{equation}\label{fibers}
 |\pi_e^{-1}(n)|=
\begin{cases}
 r(\mN-1)^n(r-1)^{n-1} & \text{if } n\ge1 \\
 1 & \text{if } n=0.
\end{cases}
\end{equation}
The weight function $w$ is given by~\eqref{weight_fct}. 
The degree of a vertex is
$$
d_n=
\begin{cases} 
   r^2(\mN-1)^{n+1}(r-1)^{n-1} &\text{if } n\ge1  \\
   r(\mN-1) &\text{if } n=0,
\end{cases} 
$$
so that the non-zero matrix coefficients of the operator $\mathcal{M}=I-\mathcal{L}$ in~\eqref{operatorM} 
are given by
\begin{align*}
 & \mathcal{M}(n,n)=
 \begin{cases} 
   \frac{\mN-2}{r(\mN-1)} &\text{if } n\ge1  \\
   0 &\text{if } n=0
\end{cases} 
  \\
 & \mathcal{M}(n-1,n)=\mathcal{M}(n,n-1)=
\begin{cases} 
   \frac{\sqrt{r-1}}{r\sqrt{\mN-1}} &\text{if } n\ge2 \\
   \frac{1}{\sqrt{r(\mN-1)}} &\text{if } n=1.
\end{cases} 
\end{align*}
In this way, we obtain the following.
\begin{proposition} \label{propM}
The operator $\mathcal{M}$ of the graph $\whl$ acts explicitly on a function $f\colon\Z_{\ge0} \to \C$, as follows:
 for $n\geq 2$, 
\begin{equation} 
 (\mathcal{M}f)(n)=\frac{1}{r(\mN-1)}\Big( \sqrt{(r-1)(\mN-1)}(f(n-1)+f(n+1))+(\mN-2)f(n) \Big)
\end{equation}
and
\begin{align}
 & (\mathcal{M}f)(1)=\frac{1}{r(\mN-1)}\Big( \sqrt{r(\mN-1)}f(0)+(\mN-2)f(1))+\sqrt{(r-1)(\mN-1)}f(2) \Big) \\
 & (\mathcal{M}f)(0)=\frac{1}{\sqrt{r(\mN-1)}}f(1).
\end{align}
\end{proposition}
\begin{remark}
 Recall that $\mathcal{M}$ is a self-adjoint operator on the Hilbert space $\ell^2(\Z_{\ge0})$ with respect to the standard scalar product.
\end{remark}

\subsection{Spectral analysis of the projected Laplacian}
As a step towards the proof of Theorem~\ref{K1}, we solve the spectral problem for the self-adjoint operator $\mathcal{L}$ in the Hilbert space $\ell^2(\Z_{\ge0})$.
By~\ref{operatorM}, the spectral problem for $\mathcal{L}$ directly translates to that of $\mathcal{M}$. 

\subsubsection{Eigenvalues and eigenfunctions}
\begin{proposition}\label{spectrum_L^pr1}
The $\pi_e$-projected Laplacian $\mathcal{L}$ of $\whl$ has at most one eigenvalue, and it has one eigenvalue if and only if $\mN>r$, given by
\begin{equation}
  \lambda_*=1-\mu_*=\frac{\mN}{\mN-1}
\end{equation}
 with the corresponding real $\ell^2$-eigenfunction given by
\begin{equation}\label{disc_efct}
 f_*(n)=
\begin{cases}
 (-1)^n \sqrt{\frac{r(\mN-r)}{p(r-1)}} \sqrt{\frac{r-1}{\mN-1}}^n & \text{if } n\ge1 \\
  \sqrt{\frac{\mN-r}{\mN}} & \text{if } n=0.
\end{cases}
\end{equation}
It also has the following generalized eigenvalues, parameterized by $x\in(0,\pi)$,
\begin{equation} \label{continuous_spectrum}
 \lambda_x=1-\mu_x=1-\frac{\mN-2+2\cos(x)\sqrt{(r-1)(\mN-1)}}{r(\mN-1)}.
\end{equation}
The corresponding real generalized eigenfunctions are given by
\begin{equation} \label{gen_efct}
 f_x(n)=
\begin{cases}
 (r-1)\sin((n+1)x)+(\mN-2)\sqrt{\frac{r-1}{\mN-1}}\sin(nx)-\sin((n-1)x) & \text{if } n\ge1 \\
 \sqrt{r(r-1)}\sin(x) & \text{if } n=0.
\end{cases}
\end{equation}
\end{proposition}

\begin{proof}
Using Proposition~\ref{propM}, the equation $\mathcal{M}f=\mu f$ is equivalent to the following system of difference equations:
 \begin{align} { }
& (r(\mN-1)\mu-(\mN-2))f(n)=\sqrt{(r-1)(\mN-1)}(f(n-1)+f(n+1)), \ n\ge2 \label{l1} \\
& (r(\mN-1)\mu-(\mN-2))f(1)=\sqrt{r(\mN-1)}f(0)+\sqrt{(r-1)(\mN-1)}f(2) \label{l2} \\
& \mu f(0)=\frac{1}{\sqrt{r(\mN-1)}}f(1). \label{l3}
\end{align}
Since~\eqref{l1} is a second-order linear recurrence with constant coefficients, we look for solutions of the form $f(n)=\xi^n$, which leads to the spectral (dispersion) relation
\begin{equation}\label{speceq1}
\mu=\frac{\mN-2+\sqrt{(r-1)(\mN-1)}(\xi^{-1}+\xi)}{r(\mN-1)}.
\end{equation}
Note that the right-hand side of \eqref{speceq1} depends only on $\xi+\xi^{-1}$, so if $\xi$ is a solution, so is $\xi^{-1}$. 
Therefore the general solution of~\eqref{l1} is of the form, for $n\ge1$,
\begin{equation}\label{gensol}
 f(n)=\alpha\xi^n+\beta\xi^{-n}, \quad \alpha,\beta\in\C.
\end{equation}
Substituting~\eqref{gensol} into the boundary equation~\eqref{l2} and using~\eqref{speceq1}, we solve for $f(0)$ to obtain
\begin{equation}\label{f0}
 f(0)=\sqrt{\frac{r-1}{r}}(\alpha+\beta).
\end{equation}
Using~\eqref{speceq1}--\eqref{f0} 
in~\eqref{l3}, we obtain
\begin{align}
 \alpha \Big(\xi-(\mN-2)\sqrt{\frac{r-1}{\mN-1}}-(r-1)\xi^{-1}\Big)=\beta \Big(-\xi^{-1}+(\mN-2)\sqrt{\frac{r-1}{\mN-1}}+(r-1)\xi\Big). \label{hom_sys}
\end{align}
We distinguish the cases $\vert \xi \vert \ne1$ and $\vert \xi \vert=1$.
\vskip 0.3cm
{\it The case $\vert \xi \vert \ne1$.}
Since~\eqref{speceq1} is invariant under the change $\xi \leftrightarrow \xi^{-1}$, assume, without loss of generality, that $\vert \xi \vert>1$. 
The condition $f\in \ell^2(\Z_{\ge0})$ forces $\alpha=0$.
Hence, from~\eqref{hom_sys}, we obtain
$$
-\xi^{-1}+(\mN-2)\sqrt{\frac{r-1}{\mN-1}}+(r-1)\xi=0 \ \iff \  \xi^2+\frac{(\mN-2)}{\sqrt{(r-1)(\mN-1)}}\xi=\frac{1}{r-1},
$$
with solutions
$$
\xi_{\pm}=\frac{2-\mN\pm \mN}{2\sqrt{(r-1)(\mN-1)}}.
$$
Observe that $|\xi_{+}|\le 1$ for all $\mN,r\ge2$, 
while $|\xi_{-}|>1$ if and only if $\mN>r$. Thus, for $\mN>r$, there exists a unique solution with $|\xi|>1$, given by
\begin{equation}\label{sol_xi}
 \xi_{-}=\frac{1-\mN}{\sqrt{(r-1)(\mN-1)}}=-\sqrt{\frac{\mN-1}{r-1}}.
\end{equation}
For $\mN>r$, substituting $\xi_{-}$ into~\eqref{speceq1} gives the eigenvalue
\begin{equation}
 \mu=\mu_*:=\frac{1}{1-\mN}.
\end{equation}
Using $\alpha=0$ and $\xi=\xi_{-}$, the corresponding $\ell^2$-eigenfunction is, up to a scalar $\beta$, given by
\begin{align*}
& f_*(0)=\beta \sqrt{\frac{r-1}{r}}, \quad f_*(n)=\beta \xi^{-n}=(-1)^n\beta \sqrt{\frac{r-1}{\mN-1}}^n, \quad n\ge1.
\end{align*}
Since $\mN>r$, summing the geometric series gives
$$
\|f_*\|^2_{\ell^2(\Z_{\ge0})}
=|\beta|^2 \frac{\mN(r-1)}{r(\mN-r)}.
$$ 
Thus, choosing $\beta=\sqrt{\frac{r(\mN-r)}{\mN(r-1)}}$ gives the normalized eigenfunction~\eqref{disc_efct}.

\vskip 0.3cm
{\it The case $\vert \xi \vert=1$.}
 In this case, write $\xi=e^{ix}$. Using the symmetry of the spectral equation under $\xi \leftrightarrow \xi^{-1}$, we can always assume that $x\in(0,\pi)$. 
Then, \eqref{speceq1} gives
\begin{equation}\label{mu_x}
 \mu_x=\frac{\mN-2+2\cos x\sqrt{(r-1)(\mN-1)}}{r(\mN-1)}.
\end{equation}
Solving~\eqref{hom_sys} determines~$\alpha$ and $\beta$ up to a common factor $\gamma$. Substituting these solutions into~\eqref{gensol}, we obtain, for $n\ge1$,
\begin{align*}
 f_x(n) 
 & = 2i\gamma \Big((r-1)\sin((n+1)x)+(\mN-2)\sqrt{\frac{r-1}{\mN-1}}\sin(nx)-\sin((n-1)x) \Big).
\end{align*}
Substituting the same solutions
into~\eqref{f0}, we obtain
\begin{align*}
 f_x(0) 
 =2i\gamma\sqrt{r(r-1)} \sin(x).
\end{align*}
Choosing $\gamma=-\frac{i}{2}$, we obtain the generalized eigenfunctions in~\eqref{gen_efct}, corresponding to the generalized eigenvalues $\mu_x$ in~\eqref{mu_x}.
\end{proof}

\subsubsection{Scalar products of (generalized) eigenfunctions}
The scalar products of the generalized eigenfunctions, interpreted in a distributional sense, allow us
to determine the spectral measure associated with this operator. 

\begin{proposition}\label{scalar_products1}
We have the following scalar products, for $x,y\in(0,\pi)$,
\begin{equation}\label{sc_pr_H}
  \langle f_x,  f_y \rangle=H(x)\delta(x-y)
\end{equation}
where 
\begin{equation}\label{Hx}
 H(x):= \frac{\pi}{2} \Big(\big((\mN-2)\sqrt{\frac{r-1}{\mN-1}}+(r-2)\cos x\big)^2+(r\sin x)^2 \Big),
\end{equation}
and $\delta(x)$ is Dirac's delta function.
\end{proposition}
\begin{proof}
 
Let $x,y \in (0,\pi)$. Using~\eqref{gen_efct}, we have 
\begin{align*}
 \langle f_x, f_y \rangle 
& = r(r-1)\sin x \sin y \notag \\
& \quad + \sum_{n\ge1} \Big( (r-1)\sin((n+1)x)+(\mN-2)\sqrt{\frac{r-1}{\mN-1}}\sin nx-\sin((n-1)x) \Big) \\ 
& \quad \quad \quad \quad \times \Big( (r-1)\sin((n+1)y)+(\mN-2)\sqrt{\frac{r-1}{\mN-1}}\sin ny-\sin((n-1)y) \notag \Big).
 \end{align*}
We expand the product and evaluate the resulting sums separately.
Using the trigonometric identity
\begin{equation} \label{2sin_id}
  2\sin a\sin b=\cos(a-b)-\cos(a+b),
\end{equation}
 the distributional form of the Poisson summation formula,
\begin{equation}\label{poisson_summation}
 \sum_{n\in\Z}e^{2\pi inx}=\sum_{k\in\Z}\delta(x+k),
\end{equation}
and the equality, for any continuous function $g$,
\begin{equation} \label{delta_cont_fct}
 g(x)\delta(x-y)=g(y)\delta(x-y),
\end{equation}
we have the following distributional identities, for $x,y\in(0,\pi)$,
\begin{align} 
& \sum_{n\ge0} \sin nx \sin ny =\frac{\pi}{2} \delta(x-y), \label{n,n} \\ 
& \sum_{n\ge0} \Big(\sin nx \sin((n+1)y) + \sin((n+1)x)\sin(ny) \Big) =\pi \cos x \delta(x-y), \label{n,n+1} \\
& \sum_{n\ge0} \Big(\sin((n+1)x) \sin((n-1)y) + \sin((n-1)x)\sin((n+1)y) \Big) \label{n+1,n-1} \\
& \qquad \qquad \qquad \qquad \qquad \qquad \qquad \qquad \qquad \qquad \quad 
=\pi\cos(2x)\delta(x-y)-\sin x \sin y.  \notag
\end{align}
Proofs of~\eqref{n,n} and \eqref{n,n+1} can be found in \cite{MR5101585}, while~\eqref{n+1,n-1} is proved in Appendix~\ref{appA}. 
Using~\eqref{n,n}--\eqref{n+1,n-1}, and reindexing 
where necessary, we obtain
\begin{align*}
  \langle f_x, f_y \rangle &= r(r-1)\sin x \sin y + \frac{(\mN-2)^2(r-1)}{\mN-1} \frac{\pi}{2}\delta(x-y) 
  +(r-1)^2\Big(\frac{\pi}{2} \delta(x-y)-\sin x\sin y\Big) \\
  &\quad + \frac{\pi}{2}\delta(x-y)
  +(r-1)(\mN-2)\sqrt{\frac{r-1}{\mN-1}}\pi\cos x \delta(x-y) \\
  & \quad -(\mN-2)\sqrt{\frac{r-1}{\mN-1}}\pi\cos x \delta(x-y)
   -(r-1) \Big(\pi\cos(2x)\delta(x-y) + \sin x\sin y \Big).
\end{align*}
The terms containing $\sin x \sin y$ cancel. Thus, we obtain
\begin{equation*}
  \langle f_x, f_y \rangle = \frac{\pi}{2} A(x)\delta(x-y),
\end{equation*}
where 
\begin{align*}
 A(x) & :=\frac{(\mN-2)^2(r-1)}{\mN-1}+(r-1)^2+1+2(r-2)(\mN-2)\sqrt{\frac{r-1}{\mN-1}}\cos x-2(r-1)\cos(2x).
\end{align*}
Using $\cos(2x)=2\cos^2 x-1$, this simplifies to 
\begin{align*}
  A(x) 
  &= \Big((\mN-2)\sqrt{\frac{r-1}{\mN-1}}+(r-2)\cos x\Big)^2 +(r\sin x)^2, 
\end{align*}
proving~\eqref{sc_pr_H} and \eqref{Hx}.
\end{proof}

\subsubsection{Completeness}
\begin{proposition}\label{completeness1_prop}  
The following completeness condition is satisfied: for all $m,n\in\Z_{\ge0}$, 
\begin{equation}\label{completeness1}
\delta_{m,n}=\int_0^\pi f_x(m)f_x(n)\frac{\mathrm{d}x}{H(x)}+
\begin{cases}
 f_*(m)f_*(n)&\text{if } \mN>r\\
 0&\text{if } \mN\le r
\end{cases},
\end{equation}
where $f_*$, $f_{x}$ are given in Proposition~\ref{spectrum_L^pr1}, and $H(x)$ is given in Proposition~\ref{scalar_products1}.

\end{proposition}
The proof is based on generating functions, which allow to reduce the integral to a contour integral on the complex unit circle where all poles are simple. The details are given in Appendix~\ref{appB}.

\subsection{Heat kernel on the projected graph and proof of Theorem~\ref{K1}}

\begin{proposition} \label{prop:hk_proj}
The heat kernel of the $\pi_e$-projected Laplacian $\mathcal{L}_{\whl}$ on $\ell^2(\Z_{\ge0})$ is given by 
 $$
 h^{\whl}_t(m,n)
 = \int_0^\pi e^{-t\lambda_x} f_x(m)f_x(n)\frac{\mathrm{d}x}{H(x)}+e^{-t\lambda_*}
\begin{cases}
 f_*(m)f_*(n)&\text{if } \mN>r\\
 0&\text{if } \mN\le r
\end{cases} ,
 $$
 where $\lambda_*$, $\lambda_x$ are given in Proposition~\ref{spectrum_L^pr1}.
\end{proposition}

We now have all the preparations for proving Theorem~\ref{K1}.
\begin{proof}[Proof of Theorem~\ref{K1}]
Let $h^{\mathcal{X}}_t=e^{-t\mathcal{L}_{\mathcal{X}}}$ be the heat kernel of the Laplacian $\mathcal{L}_{\mathcal{X}}$. 
Then, formula~\eqref{eq6} implies that the function $k^{\mathcal{X}}_t(x)$ in~\eqref{heat_kernel_det} is given by
$$
k^{\mathcal{X}}_t(x)=K^{\mathcal{X}}_t(\pi_e(x)), \quad K^{\mathcal{X}}_t(n):=\frac{1}{\sqrt{|\pi_e^{-1}(n)|}} h^{\whl}_t(0,n),
$$
where $|\pi_e^{-1}(n)|$ is given in~\eqref{fibers}.
More specifically, using Proposition~\ref{prop:hk_proj}, we obtain
\begin{multline*}
 K^{\mathcal{X}}_t(n)  = \frac{1}{\sqrt{r(\mN-1)^n(r-1)^{n-1}}}\bigg( \chi_{(0,+\infty)}(\mN-r) e^{-t\frac{\mN}{\mN-1}} f(*,n)f(*,0) \ + \\
  \int_0^\pi e^{-t(1-\frac{\mN-2+2\cos x\sqrt{(r-1)(\mN-1)}}{r(\mN-1)})} f(x,n) f(x,0) \frac{\mathrm{d}x}{H(x)} \bigg) 
\end{multline*}
\begin{multline*}
= \frac{\max(\mN-r,0)}{\mN(1-\mN)^n} e^{-t\frac{\mN}{\mN-1}} + \frac{2}{\pi\sqrt{(\mN-1)^n(r-1)^{n-2}}} \times \\
 \int_0^\pi \frac{e^{-t(1-\frac{\mN-2+2\cos x\sqrt{(r-1)(\mN-1)}}{r(\mN-1)})} \Big( (r-1)\sin(n+1)x+(\mN-2)\sqrt{\frac{r-1}{\mN-1}}\sin nx-\sin(n-1)x \Big)\sin x}{\Big((\mN-2)\sqrt{\frac{r-1}{\mN-1}}+(r-2)\cos x\Big)^2 +(r\sin x)^2} \mathrm{d}x 
\end{multline*}
and
\begin{align*}
 K^{\mathcal{X}}_t(0) & =\chi_{(0,+\infty)}(\mN-r) e^{-t\frac{\mN}{\mN-1}} f(*,0)f(*,0) + \int_0^\pi e^{-t(1-\frac{\mN-2+2\cos(x)\sqrt{(r-1)(\mN-1)}}{r(\mN-1)})} f(x,0)f(x,0)
  \frac{\mathrm{d}x}{H(x)} \\
  &= \frac{\max(\mN-r,0)}{\mN} e^{-t\frac{\mN}{\mN-1}} + \frac{2r(r-1)}{\pi} \int_0^\pi  
\frac{e^{-t(1-\frac{\mN-2+2\cos x\sqrt{(r-1)(\mN-1)}}{r(\mN-1)})} \sin^2 x}{\Big((\mN-2)\sqrt{\frac{r-1}{\mN-1}}+(r-2)\cos x\Big)^2 +(r\sin x)^2} \mathrm{d}x.
\end{align*}
Using the substitutions $r=s+1$ and $\mN=q+1$, these formulas can be rewritten as the one in Theorem~\ref{K1}.
\end{proof}

\section{Free products of two finite groups of unequal order} \label{section2}

In this section we extend the covering approach developed in the previous section
to free products of two arbitrary finite groups of unequal orders.
Let
$$
\mathcal B=\mathcal B_{\oG,\oH}:=G*H, \quad |G|=\oG+1,\quad |H|=\oH+1,\quad \oH>\oG\ge1,
$$
and consider the Cayley graph $\mathcal{Y} =\mathcal Y_{\oG,\oH}= \mathrm{Cay}(\mathcal B, S)$ with respect to the generating set
$$
S=(G\setminus\{e\}) \sqcup (H\setminus\{e\}).
$$
We also consider a slight generalization of the weights that introduces no substantial complications in the spectral problem. Specifically, each edge of $G$ has weight 
$\mwo$, while each edge of $H$ has weight $\mwt$.
Since the normalized Laplacian is invariant under a rescaling of all edge weights, it depends only on the ratio $\tau:=\frac{\mwo}{\mwt}$. Nevertheless, we still keep both parameters $\mwo$ and $\mwt$, since the quotient construction and resulting formulas are invariant under the exchange of pairs $(\oG,\mwo)$ and $(\oH,\mwt)$.

\subsection{The Cayley graph and its projection}~\label{projection2}
The Cayley graph $\mathcal Y_{2,4}$ is illustrated in Figure~\ref{fig:cayley_graph2}.
As in the equal-order case, the graph is obtained by gluing complete graphs along vertices.
However, the local geometry is no longer homogeneous since complete graphs of size $|G|$ and $|H|$ alternate.

With a similar quotient construction as in Section~\ref{family1}, we obtain a strong and regular covering 
$$
\pi_e : \mathcal{Y} \to \wl,
$$
where $\wl$ is a weighted line.

\begin{figure}[h]
\center
\begin{tikzpicture}[scale=1]

\tikzset{
  treeedge/.style={line width=0.6pt},
  cliqueedge/.style={gray!90, line width=0.3pt}
}

\begin{scope}[scale=0.8, transform shape]

\coordinate (C0) at (0,0); 
\coordinate (t1) at ($(C0)+(0:1.3)$);
\coordinate (t3) at ($(C0)+(120:1.3)$);
\coordinate (t2) at ($(C0)+(240:1.3)$);

\coordinate (C1) at ($(t1)+(0:1.3)$);
\coordinate (a1) at ($(C1)+(180:1.3)$);
\coordinate (a2) at ($(C1)+(252:1.3)$);
\coordinate (a3) at ($(C1)+(324:1.3)$);
\coordinate (a4) at ($(C1)+(36:1.3)$); 
\coordinate (a5) at ($(C1)+(108:1.3)$); 

\coordinate (C2) at ($(t2)+(240:1.3)$);
\coordinate (b1) at ($(C2)+(60:1.3)$);
\coordinate (b2) at ($(C2)+(132:1.3)$);
\coordinate (b3) at ($(C2)+(204:1.3)$);
\coordinate (b4) at ($(C2)+(276:1.3)$);
\coordinate (b5) at ($(C2)+(348:1.3)$);

\coordinate (C3) at ($(t3)+(120:1.3)$);
\coordinate (c1) at ($(C3)+(-60:1.3)$);
\coordinate (c2) at ($(C3)+(12:1.3)$);
\coordinate (c3) at ($(C3)+(84:1.3)$);
\coordinate (c4) at ($(C3)+(156:1.3)$);
\coordinate (c5) at ($(C3)+(228:1.3)$);

\node at (1.25,-0.25) {\small$e$};

\draw[fill=black] (t1) circle [radius=2pt];
\draw[fill=black] (t2) circle [radius=2pt];
\draw[fill=black] (t3) circle [radius=2pt];

\draw[fill=black] (a1) circle [radius=1.8pt];
\draw[fill=black] (a2) circle [radius=1.8pt];
\draw[fill=black] (a3) circle [radius=1.8pt];
\draw[fill=black] (a4) circle [radius=1.8pt];
\draw[fill=black] (a5) circle [radius=1.8pt];

\draw[fill=black] (b1) circle [radius=1.8pt];
\draw[fill=black] (b2) circle [radius=1.8pt];
\draw[fill=black] (b3) circle [radius=1.8pt];
\draw[fill=black] (b4) circle [radius=1.8pt];
\draw[fill=black] (b5) circle [radius=1.8pt];

\draw[fill=black] (c1) circle [radius=1.8pt];
\draw[fill=black] (c2) circle [radius=1.8pt];
\draw[fill=black] (c3) circle [radius=1.8pt];
\draw[fill=black] (c4) circle [radius=1.8pt];
\draw[fill=black] (c5) circle [radius=1.8pt];

\draw(t1)--(t2)--(t3)--(t1); 

\draw(t1)--(a2); \draw(t1)--(a3); \draw(t1)--(a4); \draw(t1)--(a5);
\draw[cliqueedge](a2)--(a3); \draw[cliqueedge](a2)--(a4); \draw[cliqueedge](a2)--(a5); 
\draw[cliqueedge](a3)--(a4); \draw[cliqueedge](a3)--(a5); 
\draw[cliqueedge](a4)--(a5); 

\draw(t2)--(b2); \draw(t2)--(b3); \draw(t2)--(b4); \draw(t2)--(b5);
\draw[cliqueedge](b2)--(b3); \draw[cliqueedge](b2)--(b4); \draw[cliqueedge](b2)--(b5); 
\draw[cliqueedge](b3)--(b4); \draw[cliqueedge](b3)--(b5); 
\draw[cliqueedge](b4)--(b5); 

\draw(t3)--(c2); \draw(t3)--(c3); \draw(t3)--(c4); \draw(t3)--(c5);
\draw[cliqueedge](c2)--(c3); \draw[cliqueedge](c2)--(c4); \draw[cliqueedge](c2)--(c5); 
\draw[cliqueedge](c3)--(c4); \draw[cliqueedge](c3)--(c5); 
\draw[cliqueedge](c4)--(c5); 

\coordinate (a2t2) at ($(a2)+(222:1)$);
\coordinate (a2t3) at ($(a2)+(282:1)$);

\node at ($(a2t2)+(222:0.3)$){\tiny$...$};
\node at ($(a2t3)+(302:0.25)$){\tiny$...$};

\coordinate (a3t2) at ($(a3)+(294:1)$);
\coordinate (a3t3) at ($(a3)+(354:1)$);

\node at ($(a3t2)+(270:0.25)$){\tiny$...$};
\node at ($(a3t3)+(14:0.25)$){\tiny$...$};

\coordinate (a4t2) at ($(a4)+(6:1)$);
\coordinate (a4t3) at ($(a4)+(66:1)$);

\node at ($(a4t2)+(-20:0.25)$){\tiny$...$};
\node at ($(a4t3)+(90:0.25)$){\tiny$...$};

\coordinate (a5t2) at ($(a5)+(78:1)$);
\coordinate (a5t3) at ($(a5)+(138:1)$);

\node at ($(a5t2)+(78:0.25)$){\tiny$...$};
\node at ($(a5t3)+(158:0.3)$){\tiny$...$};

\draw[fill=black] (a2t2) circle [radius=1.3pt];
\draw[fill=black] (a2t3) circle [radius=1.3pt];
\draw(a2)--(a2t2)--(a2t3)--(a2); 

\draw[fill=black] (a3t2) circle [radius=1.3pt];
\draw[fill=black] (a3t3) circle [radius=1.3pt];
\draw(a3)--(a3t2)--(a3t3)--(a3); 

\draw[fill=black] (a4t2) circle [radius=1.3pt];
\draw[fill=black] (a4t3) circle [radius=1.3pt];
\draw(a4)--(a4t2)--(a4t3)--(a4); 

\draw[fill=black] (a5t2) circle [radius=1.3pt];
\draw[fill=black] (a5t3) circle [radius=1.3pt];
\draw(a5)--(a5t2)--(a5t3)--(a5);

\coordinate (b2t2) at ($(b2)+(102:1)$);
\coordinate (b2t3) at ($(b2)+(162:1)$);

\node at ($(b2t2)+(92:0.3)$){\tiny$...$};
\node at ($(b2t3)+(182:0.25)$){\tiny$...$};

\coordinate (b3t2) at ($(b3)+(174:1)$);
\coordinate (b3t3) at ($(b3)+(234:1)$);

\node at ($(b3t2)+(154:0.25)$){\tiny$...$};
\node at ($(b3t3)+(264:0.25)$){\tiny$...$};

\coordinate (b4t2) at ($(b4)+(246:1)$);
\coordinate (b4t3) at ($(b4)+(306:1)$);

\node at ($(b4t2)+(246:0.25)$){\tiny$...$};
\node at ($(b4t3)+(326:0.25)$){\tiny$...$};

\coordinate (b5t2) at ($(b5)+(318:1)$);
\coordinate (b5t3) at ($(b5)+(18:1)$);

\node at ($(b5t2)+(278:0.25)$){\tiny$...$};

\draw[fill=black] (b2t2) circle [radius=1.3pt];
\draw[fill=black] (b2t3) circle [radius=1.3pt];
\draw(b2)--(b2t2)--(b2t3)--(b2); 

\draw[fill=black] (b3t2) circle [radius=1.3pt];
\draw[fill=black] (b3t3) circle [radius=1.3pt];
\draw(b3)--(b3t2)--(b3t3)--(b3); 

\draw[fill=black] (b4t2) circle [radius=1.3pt];
\draw[fill=black] (b4t3) circle [radius=1.3pt];
\draw(b4)--(b4t2)--(b4t3)--(b4); 

\draw[fill=black] (b5t2) circle [radius=1.3pt];
\draw[fill=black] (b5t3) circle [radius=1.3pt];
\draw(b5)--(b5t2)--(b5t3)--(b5); 

\coordinate (c2t2) at ($(c2)+(-18:1)$);
\coordinate (c2t3) at ($(c2)+(42:1)$);

\node at ($(c2t3)+(82:0.25)$){\tiny$...$};

\coordinate (c3t2) at ($(c3)+(54:1)$);
\coordinate (c3t3) at ($(c3)+(114:1)$);

\node at ($(c3t2)+(54:0.25)$){\tiny$...$};
\node at ($(c3t3)+(134:0.25)$){\tiny$...$};

\coordinate (c4t2) at ($(c4)+(126:1)$);
\coordinate (c4t3) at ($(c4)+(186:1)$);

\node at ($(c4t2)+(92:0.25)$){\tiny$...$};
\node at ($(c4t3)+(226:0.25)$){\tiny$...$};

\coordinate (c5t2) at ($(c5)+(198:1)$);
\coordinate (c5t3) at ($(c5)+(258:1)$);

\node at ($(c5t2)+(178:0.25)$){\tiny$...$};

\draw[fill=black] (c2t2) circle [radius=1.3pt];
\draw[fill=black] (c2t3) circle [radius=1.3pt];
\draw(c2)--(c2t2)--(c2t3)--(c2); 

\draw[fill=black] (c3t2) circle [radius=1.3pt];
\draw[fill=black] (c3t3) circle [radius=1.3pt];
\draw(c3)--(c3t2)--(c3t3)--(c3); 

\draw[fill=black] (c4t2) circle [radius=1.3pt];
\draw[fill=black] (c4t3) circle [radius=1.3pt];
\draw(c4)--(c4t2)--(c4t3)--(c4); 

\draw[fill=black] (c5t2) circle [radius=1.3pt];
\draw[fill=black] (c5t3) circle [radius=1.3pt];
\draw(c5)--(c5t2)--(c5t3)--(c5); 

\end{scope}

\draw[-{Latex}] ($(t1 |- 0,-3.5)$) -- ($(t1 |- 0,-4.5)$);
\node at (0.8,-4) {\small $\pi_e$};

\node at ($(t1 |- 0,3.5)$) {$\mathcal{Y}$};



\node at ($(t1 |- 0,-7)$) {$\wl$};

\coordinate (p0) at  ($ (t1) + (0,-5.5) $);
\coordinate (p1) at ($ (p0) + (1.5,0) $); 
\coordinate (p2) at ($ (p1) + (1.5,0) $); 
\coordinate (p3) at ($ (p2) + (1.5,0) $); 
\coordinate (etc) at ($ (p3) + (1.2,0) $); 

\coordinate (-p1) at  ($(t2 |- p0)$);
\coordinate (-p2) at ($ (-p1) + (-1.5,0) $); 
\coordinate (-p3) at ($ (-p2) + (-1.5,0) $); 
\coordinate (-etc) at ($ (-p3) + (-1.2,0) $); 

\node at ($ (p0) + (0,-0.4) $) {\small $0$};
\node at ($ (p1) + (0,-0.4) $) {\small $1$};
\node at ($ (p2) + (0,-0.4) $) {\small $2$};
\node at ($ (p3) + (0,-0.4) $) {\small $3$};

\node at ($ (-p1) + (0,-0.4) $) {\small $-1$};
\node at ($ (-p2) + (0,-0.4) $) {\small $-2$};
\node at ($ (-p3) + (0,-0.4) $) {\small $-3$};

\draw[fill=black] (p0) circle [radius=2pt];
\draw[fill=black] (p1) circle [radius=2pt];
\draw[fill=black] (p2) circle [radius=2pt];
\draw[fill=black] (p3) circle [radius=2pt];
\draw[fill=black] (-p1) circle [radius=2pt];
\draw[fill=black] (-p2) circle [radius=2pt];
\draw[fill=black] (-p3) circle [radius=2pt];

\draw(p0)--node[above]{\tiny{\color{black} $\oH\mwt$}}(p1); 
\draw(p1)--node[above]{\tiny{\color{black} $\oG \oH\mwo$}}(p2); 
\draw(p2)--node[above]{\tiny{\color{black} $\oG \oH^2\mwt$}}(p3); 
\draw(p3)--(etc); 

\draw(p0)--node[above]{\tiny{\color{black} $\oG\mwo$}}(-p1); 
\draw(-p1)--node[above]{\tiny{\color{black} $\oG \oH\mwt$}}(-p2); 
\draw(-p2)--node[above]{\tiny{\color{black} $\oG^2 \oH\mwo$}}(-p3); 
\draw(-p3)--(-etc);

\makeatletter
\tikzset{my loop/.style =  {to path={
  \pgfextra{\let\tikztotarget=\tikztostart}
  [looseness=6,min distance=10mm]
  \tikz@to@curve@path}
  }}  
\makeatletter 

\path (p1) edge[my loop] node[above] {\tiny{\color{black} $2\binom{\oH}{2}\mwt$}} (p1);
\path (p2) edge[my loop] node[above] {\tiny{\color{black} $2N\binom{\oG}{2}\mwo$}} (p2);
\path (p3) edge[my loop] node[above] {\tiny{\color{black} $2MN\binom{\oH}{2}\mwt$}} (p3);

\path (-p1) edge[my loop] node[above] {\tiny{\color{black} $2\binom{\oG}{2}\mwo$}} (-p1);
\path (-p2) edge[my loop] node[above] {\tiny{\color{black} $2M\binom{\oH}{2}\mwt$}} (-p2);
\path (-p3) edge[my loop] node[above] {\tiny{\color{black} $2MN\binom{\oG}{2}\mwo$}} (-p3);

\node at (7.2,-5.5) {$\dots$};
\node at (-5.1,-5.5) {$\dots$};

\end{tikzpicture}
\caption{$\mathcal{Y}$ covers the weighted line $\wl$ strongly and regularly. The distinguished vertex is $0$.} \label{fig:covering2}
\end{figure}

\begin{figure}[h]
\center
\begin{tikzpicture}[scale=0.8]

\tikzset{
  treeedge/.style={line width=0.6pt},
  cliqueedge/.style={gray!90, line width=0.3pt}
}

\begin{scope}[scale=0.8, transform shape]

\coordinate (C0) at (0,0); 
\coordinate (t1) at ($(C0)+(0:1.3)$);
\coordinate (t3) at ($(C0)+(120:1.3)$);
\coordinate (t2) at ($(C0)+(240:1.3)$);

\coordinate (C1) at ($(t1)+(0:1.3)$);
\coordinate (a1) at ($(C1)+(180:1.3)$);
\coordinate (a2) at ($(C1)+(252:1.3)$);
\coordinate (a3) at ($(C1)+(324:1.3)$);
\coordinate (a4) at ($(C1)+(36:1.3)$); 
\coordinate (a5) at ($(C1)+(108:1.3)$); 

\coordinate (C2) at ($(t2)+(240:1.3)$);
\coordinate (b1) at ($(C2)+(60:1.3)$);
\coordinate (b2) at ($(C2)+(132:1.3)$);
\coordinate (b3) at ($(C2)+(204:1.3)$);
\coordinate (b4) at ($(C2)+(276:1.3)$);
\coordinate (b5) at ($(C2)+(348:1.3)$);

\coordinate (C3) at ($(t3)+(120:1.3)$);
\coordinate (c1) at ($(C3)+(-60:1.3)$);
\coordinate (c2) at ($(C3)+(12:1.3)$);
\coordinate (c3) at ($(C3)+(84:1.3)$);
\coordinate (c4) at ($(C3)+(156:1.3)$);
\coordinate (c5) at ($(C3)+(228:1.3)$);

\node at (1.25,-0.25) {\small$e$};

\draw[fill=black] (t1) circle [radius=2pt];
\draw[fill=black] (t2) circle [radius=2pt];
\draw[fill=black] (t3) circle [radius=2pt];

\draw[fill=black] (a1) circle [radius=1.8pt];
\draw[fill=black] (a2) circle [radius=1.8pt];
\draw[fill=black] (a3) circle [radius=1.8pt];
\draw[fill=black] (a4) circle [radius=1.8pt];
\draw[fill=black] (a5) circle [radius=1.8pt];

\draw[fill=black] (b1) circle [radius=1.8pt];
\draw[fill=black] (b2) circle [radius=1.8pt];
\draw[fill=black] (b3) circle [radius=1.8pt];
\draw[fill=black] (b4) circle [radius=1.8pt];
\draw[fill=black] (b5) circle [radius=1.8pt];

\draw[fill=black] (c1) circle [radius=1.8pt];
\draw[fill=black] (c2) circle [radius=1.8pt];
\draw[fill=black] (c3) circle [radius=1.8pt];
\draw[fill=black] (c4) circle [radius=1.8pt];
\draw[fill=black] (c5) circle [radius=1.8pt];

\draw(t1)--(t2)--(t3)--(t1); 

\draw(t1)--(a2); \draw(t1)--(a3); \draw(t1)--(a4); \draw(t1)--(a5);
\draw[cliqueedge](a2)--(a3); \draw[cliqueedge](a2)--(a4); \draw[cliqueedge](a2)--(a5); 
\draw[cliqueedge](a3)--(a4); \draw[cliqueedge](a3)--(a5); 
\draw[cliqueedge](a4)--(a5); 

\draw(t2)--(b2); \draw(t2)--(b3); \draw(t2)--(b4); \draw(t2)--(b5);
\draw[cliqueedge](b2)--(b3); \draw[cliqueedge](b2)--(b4); \draw[cliqueedge](b2)--(b5); 
\draw[cliqueedge](b3)--(b4); \draw[cliqueedge](b3)--(b5); 
\draw[cliqueedge](b4)--(b5); 

\draw(t3)--(c2); \draw(t3)--(c3); \draw(t3)--(c4); \draw(t3)--(c5);
\draw[cliqueedge](c2)--(c3); \draw[cliqueedge](c2)--(c4); \draw[cliqueedge](c2)--(c5); 
\draw[cliqueedge](c3)--(c4); \draw[cliqueedge](c3)--(c5); 
\draw[cliqueedge](c4)--(c5); 

\coordinate (a2t2) at ($(a2)+(222:1)$);
\coordinate (a2t3) at ($(a2)+(282:1)$);

\node at ($(a2t2)+(222:0.3)$){\tiny$...$};
\node at ($(a2t3)+(302:0.25)$){\tiny$...$};

\coordinate (a3t2) at ($(a3)+(294:1)$);
\coordinate (a3t3) at ($(a3)+(354:1)$);

\node at ($(a3t2)+(270:0.25)$){\tiny$...$};
\node at ($(a3t3)+(14:0.25)$){\tiny$...$};

\coordinate (a4t2) at ($(a4)+(6:1)$);
\coordinate (a4t3) at ($(a4)+(66:1)$);

\node at ($(a4t2)+(-20:0.25)$){\tiny$...$};
\node at ($(a4t3)+(90:0.25)$){\tiny$...$};

\coordinate (a5t2) at ($(a5)+(78:1)$);
\coordinate (a5t3) at ($(a5)+(138:1)$);

\node at ($(a5t2)+(78:0.25)$){\tiny$...$};
\node at ($(a5t3)+(158:0.3)$){\tiny$...$};

\draw[fill=black] (a2t2) circle [radius=1.3pt];
\draw[fill=black] (a2t3) circle [radius=1.3pt];
\draw(a2)--(a2t2)--(a2t3)--(a2); 

\draw[fill=black] (a3t2) circle [radius=1.3pt];
\draw[fill=black] (a3t3) circle [radius=1.3pt];
\draw(a3)--(a3t2)--(a3t3)--(a3); 

\draw[fill=black] (a4t2) circle [radius=1.3pt];
\draw[fill=black] (a4t3) circle [radius=1.3pt];
\draw(a4)--(a4t2)--(a4t3)--(a4); 

\draw[fill=black] (a5t2) circle [radius=1.3pt];
\draw[fill=black] (a5t3) circle [radius=1.3pt];
\draw(a5)--(a5t2)--(a5t3)--(a5);

\coordinate (b2t2) at ($(b2)+(102:1)$);
\coordinate (b2t3) at ($(b2)+(162:1)$);

\node at ($(b2t2)+(92:0.3)$){\tiny$...$};
\node at ($(b2t3)+(182:0.25)$){\tiny$...$};

\coordinate (b3t2) at ($(b3)+(174:1)$);
\coordinate (b3t3) at ($(b3)+(234:1)$);

\node at ($(b3t2)+(154:0.25)$){\tiny$...$};
\node at ($(b3t3)+(264:0.25)$){\tiny$...$};

\coordinate (b4t2) at ($(b4)+(246:1)$);
\coordinate (b4t3) at ($(b4)+(306:1)$);

\node at ($(b4t2)+(246:0.25)$){\tiny$...$};
\node at ($(b4t3)+(326:0.25)$){\tiny$...$};

\coordinate (b5t2) at ($(b5)+(318:1)$);
\coordinate (b5t3) at ($(b5)+(18:1)$);

\node at ($(b5t2)+(278:0.25)$){\tiny$...$};

\draw[fill=black] (b2t2) circle [radius=1.3pt];
\draw[fill=black] (b2t3) circle [radius=1.3pt];
\draw(b2)--(b2t2)--(b2t3)--(b2); 

\draw[fill=black] (b3t2) circle [radius=1.3pt];
\draw[fill=black] (b3t3) circle [radius=1.3pt];
\draw(b3)--(b3t2)--(b3t3)--(b3); 

\draw[fill=black] (b4t2) circle [radius=1.3pt];
\draw[fill=black] (b4t3) circle [radius=1.3pt];
\draw(b4)--(b4t2)--(b4t3)--(b4); 

\draw[fill=black] (b5t2) circle [radius=1.3pt];
\draw[fill=black] (b5t3) circle [radius=1.3pt];
\draw(b5)--(b5t2)--(b5t3)--(b5); 

\coordinate (c2t2) at ($(c2)+(-18:1)$);
\coordinate (c2t3) at ($(c2)+(42:1)$);

\node at ($(c2t3)+(82:0.25)$){\tiny$...$};

\coordinate (c3t2) at ($(c3)+(54:1)$);
\coordinate (c3t3) at ($(c3)+(114:1)$);

\node at ($(c3t2)+(54:0.25)$){\tiny$...$};
\node at ($(c3t3)+(134:0.25)$){\tiny$...$};

\coordinate (c4t2) at ($(c4)+(126:1)$);
\coordinate (c4t3) at ($(c4)+(186:1)$);

\node at ($(c4t2)+(92:0.25)$){\tiny$...$};
\node at ($(c4t3)+(226:0.25)$){\tiny$...$};

\coordinate (c5t2) at ($(c5)+(198:1)$);
\coordinate (c5t3) at ($(c5)+(258:1)$);

\node at ($(c5t2)+(178:0.25)$){\tiny$...$};

\draw[fill=black] (c2t2) circle [radius=1.3pt];
\draw[fill=black] (c2t3) circle [radius=1.3pt];
\draw(c2)--(c2t2)--(c2t3)--(c2); 

\draw[fill=black] (c3t2) circle [radius=1.3pt];
\draw[fill=black] (c3t3) circle [radius=1.3pt];
\draw(c3)--(c3t2)--(c3t3)--(c3); 

\draw[fill=black] (c4t2) circle [radius=1.3pt];
\draw[fill=black] (c4t3) circle [radius=1.3pt];
\draw(c4)--(c4t2)--(c4t3)--(c4); 

\draw[fill=black] (c5t2) circle [radius=1.3pt];
\draw[fill=black] (c5t3) circle [radius=1.3pt];
\draw(c5)--(c5t2)--(c5t3)--(c5); 

\end{scope}

\draw[-{Latex}] (0.25,-4) -- (0.25,-5);
\node at (0.8,-4.5) {\small $/ S_{\oG}$};

\node at ($(t1 |- 0,3.5)$) {$\mathcal Y$};


\begin{scope}[scale=0.8, transform shape, shift={(0,-9)}]

\coordinate (C0) at (0,0); 
\coordinate (t1) at ($(C0)+(0:1.3)$);
\coordinate (t2) at ($(C0)+(240:1.3)$);
\coordinate (t2n) at ($(t2 |- 0,0)$);

\coordinate (C1) at ($(t1)+(0:1.3)$);
\coordinate (a1) at ($(C1)+(180:1.3)$);
\coordinate (a2) at ($(C1)+(252:1.3)$);
\coordinate (a3) at ($(C1)+(324:1.3)$);
\coordinate (a4) at ($(C1)+(36:1.3)$); 
\coordinate (a5) at ($(C1)+(108:1.3)$); 

\coordinate (C2) at ($(t2n)+(180:1.3)$);
\coordinate (b1) at ($(C2)+(0:1.3)$);
\coordinate (b2) at ($(C2)+(72:1.3)$);
\coordinate (b3) at ($(C2)+(144:1.3)$);
\coordinate (b4) at ($(C2)+(216:1.3)$);
\coordinate (b5) at ($(C2)+(288:1.3)$);

\node at (1.25,-0.25) {\small$v_0$};
\node at (-0.3,-0.25) {\small$v_{-1}$};

\draw[fill=black] (t1) circle [radius=2pt];
\draw[fill=black] (t2n) circle [radius=2pt];

\draw[fill=black] (a1) circle [radius=1.8pt];
\draw[fill=black] (a2) circle [radius=1.8pt];
\draw[fill=black] (a3) circle [radius=1.8pt];
\draw[fill=black] (a4) circle [radius=1.8pt];
\draw[fill=black] (a5) circle [radius=1.8pt];

\draw[fill=black] (b1) circle [radius=1.8pt];
\draw[fill=black] (b2) circle [radius=1.8pt];
\draw[fill=black] (b3) circle [radius=1.8pt];
\draw[fill=black] (b4) circle [radius=1.8pt];
\draw[fill=black] (b5) circle [radius=1.8pt];

\draw(t2n)--node[above]{\tiny{\color{black} $M\mwo$}}(t1); 

\draw(t1)--(a2); \draw(t1)--(a3); \draw(t1)--(a4); \draw(t1)--(a5);
\draw[cliqueedge](a2)--(a3); \draw[cliqueedge](a2)--(a4); \draw[cliqueedge](a2)--(a5); 
\draw[cliqueedge](a3)--(a4); \draw[cliqueedge](a3)--(a5); 
\draw[cliqueedge](a4)--(a5); 

\draw(t2n)--(b2); \draw(t2n)--(b3); \draw(t2n)--(b4); \draw(t2n)--(b5);
\draw[cliqueedge](b2)--(b3); \draw[cliqueedge](b2)--(b4); \draw[cliqueedge](b2)--(b5); 
\draw[cliqueedge](b3)--(b4); \draw[cliqueedge](b3)--(b5); 
\draw[cliqueedge](b4)--(b5);

\coordinate (a2t2) at ($(a2)+(222:1)$);
\coordinate (a2t3) at ($(a2)+(282:1)$);

\node at ($(a2t2)+(222:0.3)$){\tiny$...$};
\node at ($(a2t3)+(302:0.25)$){\tiny$...$};

\coordinate (a3t2) at ($(a3)+(294:1)$);
\coordinate (a3t3) at ($(a3)+(354:1)$);

\node at ($(a3t2)+(270:0.25)$){\tiny$...$};
\node at ($(a3t3)+(14:0.25)$){\tiny$...$};

\coordinate (a4t2) at ($(a4)+(6:1)$);
\coordinate (a4t3) at ($(a4)+(66:1)$);

\node at ($(a4t2)+(-20:0.25)$){\tiny$...$};
\node at ($(a4t3)+(90:0.25)$){\tiny$...$};

\coordinate (a5t2) at ($(a5)+(78:1)$);
\coordinate (a5t3) at ($(a5)+(138:1)$);

\node at ($(a5t2)+(78:0.25)$){\tiny$...$};
\node at ($(a5t3)+(158:0.3)$){\tiny$...$};

\draw[fill=black] (a2t2) circle [radius=1.3pt];
\draw[fill=black] (a2t3) circle [radius=1.3pt];
\draw(a2)--(a2t2)--(a2t3)--(a2); 

\draw[fill=black] (a3t2) circle [radius=1.3pt];
\draw[fill=black] (a3t3) circle [radius=1.3pt];
\draw(a3)--(a3t2)--(a3t3)--(a3); 

\draw[fill=black] (a4t2) circle [radius=1.3pt];
\draw[fill=black] (a4t3) circle [radius=1.3pt];
\draw(a4)--(a4t2)--(a4t3)--(a4); 

\draw[fill=black] (a5t2) circle [radius=1.3pt];
\draw[fill=black] (a5t3) circle [radius=1.3pt];
\draw(a5)--(a5t2)--(a5t3)--(a5);

\coordinate (b2t2) at ($(b2)+(42:1)$);
\coordinate (b2t3) at ($(b2)+(102:1)$);

\node at ($(b2t2)+(32:0.3)$){\tiny$...$};
\node at ($(b2t3)+(135:0.25)$){\tiny$...$};

\coordinate (b3t2) at ($(b3)+(114:1)$);
\coordinate (b3t3) at ($(b3)+(174:1)$);

\node at ($(b3t2)+(90:0.25)$){\tiny$...$};
\node at ($(b3t3)+(210:0.25)$){\tiny$...$};

\coordinate (b4t2) at ($(b4)+(186:1)$);
\coordinate (b4t3) at ($(b4)+(246:1)$);

\node at ($(b4t2)+(160:0.25)$){\tiny$...$};
\node at ($(b4t3)+(280:0.25)$){\tiny$...$};

\coordinate (b5t2) at ($(b5)+(258:1)$);
\coordinate (b5t3) at ($(b5)+(318:1)$);

\node at ($(b5t2)+(245:0.25)$){\tiny$...$};
\node at ($(b5t3)+(320:0.25)$){\tiny$...$};

\draw[fill=black] (b2t2) circle [radius=1.3pt];
\draw[fill=black] (b2t3) circle [radius=1.3pt];
\draw(b2)--(b2t2)--(b2t3)--(b2); 

\draw[fill=black] (b3t2) circle [radius=1.3pt];
\draw[fill=black] (b3t3) circle [radius=1.3pt];
\draw(b3)--(b3t2)--(b3t3)--(b3); 

\draw[fill=black] (b4t2) circle [radius=1.3pt];
\draw[fill=black] (b4t3) circle [radius=1.3pt];
\draw(b4)--(b4t2)--(b4t3)--(b4); 

\draw[fill=black] (b5t2) circle [radius=1.3pt];
\draw[fill=black] (b5t3) circle [radius=1.3pt];
\draw(b5)--(b5t2)--(b5t3)--(b5);

\makeatletter
\tikzset{my loop/.style =  {to path={
  \pgfextra{\let\tikztotarget=\tikztostart}
  [looseness=6,min distance=10mm]
  \tikz@to@curve@path}
  }}  
\makeatletter 

\path (t2n) edge[my loop] node[above] {\tiny{\color{black} $2\binom{M}{2}\mwo$}} (t2n);
\end{scope}

\end{tikzpicture}
\caption{Covering $\mathcal Y \to \mathcal Y/S_M$.} \label{fig:covering2_step1}
\end{figure}


At the identity vertex $e$, the graph splits into two branches starting with the complete graphs of the groups $G$ (to the left) and $H$ (to the right), respectively. 
The symmetric group $S_{\oG}$ acts by permuting the $\oG=|G|-1$ vertices adjacent to $e$ in the copy of $G$, together with their branches, while fixing $e$ and everything to the right. 
Quotienting by this action identifies these $\oG$ vertices into a single vertex $v_{-1}$ and produces a loop at $v_{-1}$ of weight $2\binom{\oG}{2}\mwo$, as shown in Figure~\ref{fig:covering2_step1}.

After this first quotient, the same configuration appears on both sides of the edge between $v_{-1}$ and $v_0$.
Starting from $v_{-1}$ on the left (resp. from $v_0$ on the right) there are $N=|H|-1$ equivalent vertices corresponding to the group $H$.
The group $S_{\oH}$ acts by permuting these vertices together with their branches, while fixing $v_{-1}$ (resp. $v_0$).
Quotienting by this action identifies these $\oH$ vertices and produces a loop of weight $2\oG \binom{\oH}{2}\mwt$ (resp. $2\binom{\oH}{2}\mwt$), see Figure~\ref{fig:covering2}.

Iterating this procedure infinitely by alternating with the actions of $S_{\oG}$ and $S_{\oH}$ yields the weighted line $\wl$.
Note that the weights on the left and right sides are symmetric with respect to the vertex $0$, with the roles of pairs $(\oG,\mwo)$ and $(\oH,\mwt)$ exchanged, see Figure~\ref{fig:covering2}. This construction gives the following weight function.

\begin{proposition} \label{weight_fct2}
The weight function $w=w_{\wl}$ of $\wl$ is given on the non-negative side by
\begin{align*}
& w(0,0)=0, \\
& w(2n+1,2n+1)=2\binom{\oH}{2}\oG^n\oH^n\mwt, \ \ \ n\ge 0,\\
& w(2n,2n+1)=\oG^n\oH^{n+1}\mwt, \qquad \qquad \ \; n\ge 0,\\
& w(2n,2n)=2\binom{\oG}{2}\oG^{n-1}\oH^n\mwo, \qquad \quad \; n\ge 1,\\
& w(2n-1,2n)=\oG^n\oH^n\mwo, \qquad \qquad \quad \ \; n\ge 1.
\end{align*}
The weights on the negative side are obtained from those on the positive side by exchanging the roles of pairs $(\oG,\mwo)$ and $(\oH,\mwt)$.
All other weights are $0$.
\end{proposition}

The covering map $\pi_e$ can be viewed as a signed version of the word-length map. Indeed, for $x\in \mathcal{B}$,
\begin{equation}\label{CY_proj2}
 \pi_e(x)=
\begin{cases} 
0 & \text{if $x=e$} \\
 |x| & \text{if the reduced form of $x$ starts with a letter in $H$} \\
 -|x| & \text{if the reduced form of $x$ starts with a letter in $G$}.
\end{cases}
\end{equation}

As in Section~\ref{family1}, Proposition~\ref{Prop2} and vertex transitivity imply that $\mathcal{Y}$ covers $\wl$ strongly and regularly, with distinguished vertex $0$, see Fig.~\ref{fig:covering2}.

\subsection{The projected Laplacian}

Let $\pi_e\colon \mathcal{B} \to \Z$ be the covering defined above.
By definition of the covering, see Section~\ref{projection2}, we have
\begin{equation}\label{fibers2}
 |\pi_e^{-1}(m)|=
\begin{cases}
 \oG^n \oH^n & \text{if } m=2n \\
 \oG^n \oH^{n+1} & \text{if } m=2n+1,
\end{cases}
\quad m\ge0,
\end{equation}
and the fibers on the negative side are obtained by exchanging $M$ and $N$.

The weight function $w$ is given in Proposition~\ref{weight_fct2}. 
Denote
\begin{equation}
 \kappa:= M\mwo+N\mwt.
\end{equation}
Since every vertex of $\mathcal{Y}$ has weighted degree $\kappa$, the degree of a vertex $m\in Q$ is
\begin{equation}
 d_m= \kappa |\pi_e^{-1}(m)|.
\end{equation}
Thus, for $m\ge0$,
$$
d_m=
\begin{cases} 
   \kappa \oG^n\oH^n &\text{if } m=2n  \\
   \kappa \oG^n \oH^{n+1} &\text{if } m=2n+1,
\end{cases} 
$$
and the same formulas hold on the negative side with $M$ and $N$ exchanged. 

Let $\mathcal{L}=\mathcal{L}_Q$ be the normalized Laplacian on $Q$, and denote $\mathcal{M}=I-\mathcal{L}$.
From
$$
\mathcal{M}(m,n)=\frac{w(m,n)}{\sqrt{d_md_n}},
$$
we obtain the following.
\begin{proposition} \label{propM2}
The operator $\mathcal{M}$ of the graph $\wl$ acts on $f\colon\Z \to \C$, as follows
\begin{align}
 & (\mathcal{M}f)(2n+1)=\frac{\sqrt{\oH}\mwt f(2n)+(\oH-1) \mwt f(2n+1)+\sqrt{\oG} \mwo f(2n+2)}{\kappa}, \quad n\ge0, \\
 & (\mathcal{M}f)(2n)=\frac{\sqrt{\oG} \mwo f(2n-1)+(\oG-1)\mwo f(2n)+\sqrt{\oH}\mwt f(2n+1)}{\kappa}, \qquad \ \, n\ge1, \\
 & (\mathcal{M}f)(0)=\frac{\sqrt{\oG} \mwo f(-1)+\sqrt{\oH} \mwt f(1)}{\kappa}.
\end{align}
\end{proposition}
The action on the negative side is obtained by exchanging the roles of pairs $(\oG,\mwo)$ and $(\oH,\mwt)$ in the corresponding positive-side formulas.

\subsection{Spectral analysis of the projected Laplacian}
As a step towards the proof of Theorem~\ref{K2}, we solve the spectral problem for the self-adjoint operator $\mathcal{L}$ in the Hilbert space $\ell^2(\Z)$.

\subsubsection{Eigenvalues and eigenfunctions}

\begin{proposition}\label{spectrum_L^pr2}
The $\pi_e$-projected Laplacian $\mathcal{L}$ on $\wl$ has two discrete eigenvalues, given by
\begin{equation}
 \lambda_1=\frac{(\oH+1)\mwt}{\kappa}, \quad
  \lambda_2=1+\frac{\mwo+\mwt}{\kappa}.
\end{equation}
The corresponding normalized real eigenfunctions $f_1,f_2\in\ell^2(\Z)$ are given by
\begin{equation}\label{disc_efct1}
 f_1(m)=\beta_1
\begin{cases}
  (-1)^n \big(\frac{\oG}{\oH}\big)^{|n|/2} & \text{if } m=2n \\ 
  (-1)^{n+1}\oG^{n/2}\oH^{-(n+1)/2} & \text{if } m=2n+1 \\ 
  (-1)^n\oG^{(n+1)/2}\oH^{-n/2} & \text{if } m=-2n-1, 
\end{cases}
\end{equation}
and
\begin{equation}\label{disc_efct2}
 f_2(m)=\beta_2
\begin{cases}
  (\oG\oH)^{-|n|/2} & \text{if } m=2n \\ 
  -\oG^{-n/2}\oH^{-(n+1)/2} & \text{if } m=2n+1 \\ 
  -\oG^{-(n+1)/2}\oH^{-n/2} & \text{if } m=-2n-1. 
\end{cases}
\end{equation}
Here $n\in\Z$ in the even cases and $n\ge0$ in the odd cases, and
\begin{equation}\label{betas12}
 \beta_1=\sqrt{\frac{\oH-\oG}{(\oG+1)(\oH+1)}}, \quad \beta_2=\sqrt{\frac{\oG\oH-1}{(\oG+1)(\oH+1)}}.
\end{equation}
It also has generalized eigenvalues, 
parameterized by $x\in(0,\pi)$ and $\epsilon\in\{\pm1\}$, given by
\begin{equation}
 \lambda_{\epsilon,x} 
 =\frac{(\oG+1)\mwo+(\oH+1)\mwt-\epsilon R_x}{2\kappa},
\end{equation}
where
\begin{equation}\label{R_x2}
R_x=\sqrt{\big((\oG+1)\mwo-(\oH+1)\mwt\big)^2+4(\oG+\oH)\mwo\mwt+8\sqrt{\oG\oH}\mwo\mwt\cos x}.
\end{equation}
For $x\in(0,\pi)$, denote
\begin{equation}\label{def:rho_ex}
 \rho_{\epsilon,x}:=\kappa(1-\lambda_{\epsilon,x})=\frac{(\oG-1)\mwo+(\oH-1)\mwt+\epsilon R_x}{2}
\end{equation}
and
\begin{equation}\label{def:AandB_ex}
 A_{\epsilon,x}:=\rho_{\epsilon,x}-(\oH-1)\mwt, \quad B_{\epsilon,x}:=\rho_{\epsilon,x}-(\oG-1)\mwo.
\end{equation}
Define also
\begin{equation}\label{def:D_ex}
D_{\epsilon,x}:=\frac{A_{\epsilon,x}B_{\epsilon,x}\Big((\oG-1)\mwo+(\oH-1)\mwt\Big) - (A_{\epsilon,x}-B_{\epsilon,x})(\oG\mwo^2-\oH\mwt^2)}{2 A_{\epsilon,x} \sqrt{\oG\oH}\mwo\mwt \sin x}.
\end{equation}
For each $\epsilon\in\{\pm1\}$ and $x\in(0,\pi)$, the generalized eigenspace corresponding to $\lambda_{\epsilon,x}$ is two-dimensional and is spanned by the real generalized eigenfunctions $f^{(1)}_{\epsilon,x}$ and $f^{(2)}_{\epsilon,x}$, given by
\begin{equation} \label{gen_efct1}
 f^{(1)}_{\epsilon,x}(m)=
\begin{cases}
  \cos nx & \text{if } m=2n \\ 
  \frac{\sqrt{\oH}\mwt \cos nx+\sqrt{\oG}\mwo \cos((n+1)x)}{A_{\epsilon,x}} & \text{if } m=2n+1 \\ 
  \cos nx+D_{\epsilon,x}\sin nx & \text{if } m=-2n \\ 
  \frac{\sqrt{\oG}\mwo \big(\cos nx+D_{\epsilon,x}\sin nx\big)+\sqrt{\oH} \mwt \big(\cos((n+1)x)+D_{\epsilon,x}\sin((n+1)x)\big)}{B_{\epsilon,x}} & \text{if } m=-2n-1 \\ 
\end{cases}
\end{equation}
and 
\begin{equation} \label{gen_efct2}
 f^{(2)}_{\epsilon,x}(m)=
\begin{cases}
  \sin nx & \text{if } m=2n \\ 
   \frac{\sqrt{\oH}\mwt \sin nx+\sqrt{\oG} \mwo \sin((n+1)x)}{A_{\epsilon,x}} & \text{if } m=2n+1 \\ 
  -\frac{B_{\epsilon,x}}{A_{\epsilon,x}} \sin nx & \text{if } m=-2n \\ 
  -\frac{\sqrt{\oG}\mwo \sin nx+\sqrt{\oH} \mwt \sin((n+1)x)}{A_{\epsilon,x}} & \text{if } m=-2n-1 
\end{cases}
\end{equation}
where $n\ge0$ in all cases.
\end{proposition}


\begin{proof}
Consider the equation $\mathcal{M}f=\mu f$,
and denote
\begin{equation}\label{def:k,rho}
 \rho:=\kappa\mu,
\end{equation}
and
\begin{equation}\label{def:nu,A,B}
  A:=\rho-(\oH-1)\mwt, \quad B:=\rho-(\oG-1)\mwo.
\end{equation}
For notational convenience, we also denote
\begin{equation} \label{def:a,b}
 a:=\sqrt{\oG}\mwo, \quad b:=\sqrt{\oH}\mwt.
\end{equation}
For $n\ge0$, denote
\begin{equation*}
 u_n:=f(2n), \quad v_n:=f(2n+1), \quad \check{u}_n:=f(-2n), \quad \check{v}_n:=f(-2n-1).
\end{equation*}
Since $\mathcal{L}=I-\mathcal{M}$, we have
\begin{equation} \label{rel_lambda_rk}
 \lambda=1-\mu=1-\frac{\rho}{\kappa}.
\end{equation}

By Proposition~\ref{propM2}, the equation $\mathcal{M}f=\mu f$ is equivalent to 
\begin{align}
& Av_n =b u_n +a u_{n+1}, \quad n\ge0, \label{l1_prop_eig_not} \\
& Bu_n=a v_{n-1}+b v_n, \quad n\ge1, \label{l2_prop_eig_not} \\
& B\check{v}_n=a \check{u}_n+b \check{u}_{n+1}, \quad n\ge0, \label{l1_prop_eig_neg_not} \\
& A\check{u}_n=b \check{v}_{n-1} + a \check{v}_n, \quad \, n\ge1, \label{l2_prop_eig_neg_not} 
\end{align}
with the sewing equations at the origin
\begin{align}
& \rho u_0=a \check{v}_0 +b v_0, \label{l3_prop_eig_not} \\
& u_0=\check{u}_0. \label{eig_not_0}
\end{align}

Multiplying~\eqref{l2_prop_eig_not} by $A$, using~\eqref{l1_prop_eig_not}, we obtain the following second-order linear equation with constant coefficients
\begin{align}\label{linear_equ_second_order2}
 AB u_n 
 &=ab (u_{n-1}+u_{n+1})+(a^2+b^2)u_n, \quad n\ge1,
\end{align}
and the same equation holds for $\check{u}_n$.
Seeking solutions of the form $u_n=\xi^n$, we obtain the spectral (dispersion) relation
\begin{equation}\label{spec_eq1}
 \frac{AB-(a^2+b^2)}{ab}= \xi+\xi^{-1}.
\end{equation}
Hence the general solution of~\eqref{linear_equ_second_order2} is of the form
\begin{equation}\label{gen_sol_u_n}
 u_n=\alpha\xi^n+\beta\xi^{-n}, \quad \alpha,\beta\in\C, \quad n\ge0,
\end{equation}
where $\xi$ satisfies~\eqref{spec_eq1}. 
Similarly, we also have
\begin{equation}\label{gen_sol_u_n_check}
 \check{u}_n=\check{\alpha}\xi^n+\check{\beta}\xi^{-n}, \quad \check{\alpha},\check{\beta}\in\C, \quad n\ge0.
\end{equation}

{\it The case $\vert \xi \vert \ne1$.}
Since~\eqref{spec_eq1} is invariant under $\xi \leftrightarrow \xi^{-1}$, we may assume that $\vert \xi \vert>1$. The condition $f\in \ell^2(\Z)$ forces 
$$
\alpha=\check{\alpha}=0.
$$
Therefore, from~\eqref{eig_not_0}, we have $\beta=\check{\beta}$ and we obtain
\begin{equation} \label{disc_u_both}
 u_n=\check{u}_n=\beta \xi^{-n}, \quad n\ge0.
\end{equation}

We first consider the cases $A=0$ and $B=0$.
If $A=0$, then 
\begin{equation} \label{Aneq0_rho}
 \rho=(\oH-1)\mwt,
\end{equation}
 and from~\eqref{l1_prop_eig_not}, we obtain
$$
u_{n+1}=- \frac{b}{a} u_n.
$$
Since $u_{n+1}=\xi^{-1}u_n$, we obtain
$$
\xi=- \frac{a}{b},
$$
and the condition $|\xi|>1$ requires $a>b$, or equivalently
\begin{equation} \label{Aneq0_w2_cond}
 \mwt<\sqrt{\frac{\oG}{\oH}}\mwo.
\end{equation}

 Since $B\neq0$ (otherwise, from \eqref{l1_prop_eig_not} and \eqref{l1_prop_eig_neg_not}, we would have $\xi=-\frac{a}{b}=-\frac{b}{a}$, which contradicts $|\xi|>1$), from \eqref{l2_prop_eig_not} and \eqref{l1_prop_eig_neg_not}, the square-summable solutions are 
\begin{equation}\label{Aneq0_v_both}
 v_n=-\frac{bB}{a^2-b^2}\beta\xi^{-n}, \quad \check{v}_n=\frac{a^2-b^2}{aB}\beta\xi^{-n}, \quad n\ge0.
\end{equation}
Substituting \eqref{disc_u_both} and \eqref{Aneq0_v_both} into the sewing equation \eqref{l3_prop_eig_not}, we obtain
$$
\rho=\frac{a^2-b^2}{B}-\frac{Bb^2}{a^2-b^2},
$$
and substituting the definitions of $B$, $a$, $b$ and $\rho$ from \eqref{def:nu,A,B}, \eqref{def:a,b}, and \eqref{Aneq0_rho}, this becomes
\begin{equation}\label{factorization}
 (\mwo-\mwt)(\oG\mwo-\oH\mwt)(\oG\mwo+\mwt)(\mwo+\oH\mwt)=0.
\end{equation}
Since $\mwo,\mwt>0$, the last two factors are positive. Furthermore, \eqref{Aneq0_w2_cond} and $\oG<\oH$ imply $\mwt<\mwo$, so we conclude
\begin{equation} \label{Aneq0_cond}
 \oG\mwo=\oH\mwt.
\end{equation}
Hence, in the case $A=0$, we obtain an $\ell^2$-eigenfunction when $\oG\mwo=\oH\mwt$,
with corresponding eigenvalue 
$$
\lambda=1-\frac{\rho}{\kappa}=1-\frac{\oH-1}{2\oH}=\frac{\oH+1}{2\oH},
$$
where we used \eqref{def:k,rho}, \eqref{rel_lambda_rk}, \eqref{Aneq0_rho} and \eqref{Aneq0_cond}.

The case $B=0$ is similar, with the roles of $(\oG,\mwo,a)$ and $(\oH,\mwt,b)$ exchanged. From the sewing equation~\eqref{l3_prop_eig_not}, we obtain the same factorization~\eqref{factorization}. Since in this case $a<b$, we conclude 
\begin{equation} \label{Bneq0_cond}
 \mwo=\mwt,
\end{equation}
 and the corresponding eigenvalue is
$$
\lambda=1-\frac{\oG-1}{\oG+\oH}=\frac{\oH+1}{\oG+\oH}.
$$
Note that the cases $A=0$ and $B=0$ are mutually exclusive. 

We now suppose that $AB\neq0$. From~\eqref{l1_prop_eig_not}, \eqref{l1_prop_eig_neg_not} and~\eqref{disc_u_both},
we obtain
\begin{align} \label{ABneq0_both_v}
 v_n=\frac{\beta \xi^{-n}}{A}\Big(b +a \xi^{-1} \Big), \quad  
 \check{v}_n=\frac{\beta \xi^{-n}}{B}\Big(a+b \xi^{-1} \Big), \quad n\ge0.  
\end{align}
Substituting~\eqref{disc_u_both} and \eqref{ABneq0_both_v}  into the sewing equation~\eqref{l3_prop_eig_not}, we obtain
\begin{align*}
   \rho =\frac{a}{B}\Big(a +b \xi^{-1} \Big) + \frac{b}{A}\Big(b+a \xi^{-1} \Big). 
\end{align*}
Multiplying by $AB$ and using~\eqref{def:nu,A,B}, \eqref{def:a,b} and the dispersion relation~\eqref{spec_eq1}, we obtain
\begin{align}
 \rho = -\frac{\oG(\oH-1)\mwo + \oH(\oG-1)\mwt + \sqrt{\oG\oH}\xi^{-1}\Big((\oG-1)\mwo+(\oH-1)\mwt \Big)}{\sqrt{\oG\oH}(\xi-\xi^{-1})}. \label{rho_xi}
\end{align}
Using~\eqref{rho_xi} in~\eqref{def:nu,A,B}, and substituting the resulting expressions into~\eqref{spec_eq1}, we obtain, after simplification and using $AB\neq0$,
\begin{equation*}
 \oG\oH(\xi+\xi^{-1})^2-\sqrt{\oG\oH}(\oG-1)(\oH-1)(\xi+\xi^{-1})-(\oG+\oH)(\oG\oH+1)=0,
\end{equation*}
factorizing, we obtain
\begin{equation*}
\Big(\sqrt{\oG\oH}(\xi+\xi^{-1})+\oG+\oH\Big) \Big(\sqrt{\oG\oH}(\xi+\xi^{-1})-(\oG\oH+1)\Big)=0.
\end{equation*}
Thus, the two solutions satisfying $|\xi|>1$ are
\begin{equation} \label{ABneq0_sol_xi}
 \xi_1=-\sqrt{\frac{\oH}{\oG}}, \quad \xi_2=\sqrt{\oG\oH}.
\end{equation}
Substituting~\eqref{ABneq0_sol_xi} into~\eqref{rho_xi}, we obtain
\begin{equation} \label{ABneq0_sol_rho}
 \rho_1=\oG\mwo-\mwt, \quad \rho_2=-(\mwo+\mwt),
\end{equation}
and from~\eqref{rel_lambda_rk}, the corresponding eigenvalues are
\begin{equation} \label{ABneq0_sol_lambda}
 \lambda_1=\frac{(\oH+1)\mwt}{\oG\mwo+\oH\mwt}, \quad  \lambda_2=1+\frac{\mwo+\mwt}{\oG\mwo+\oH\mwt}.
\end{equation}
Substituting $\rho_1$ into~\eqref{def:nu,A,B}, we have
$$
A=\oG\mwo-\oH\mwt, \quad B=\mwo-\mwt.
$$
Thus, the cases $A=0$ and $B=0$ considered above correspond to the two choices of weights $\oG\mwo=\oH\mwt$ and $\mwo=\mwt$, respectively, which were excluded from this computation. Therefore, combining all three cases, $\lambda_1$ in~\eqref{ABneq0_sol_lambda} is an eigenvalue for all $\mwo,\mwt>0$. 
For $\rho_2$, both $A$ and $B$ are nonzero for all positive weights, so $\lambda_2$ in~\eqref{ABneq0_sol_lambda} is also an eigenvalue.

We determine the corresponding $\ell^2$-eigenfunctions. From~\eqref{disc_u_both}, for $j=1,2$, we have
\begin{equation}\label{sol_even}
 f_j(2n)=\beta_j\xi_j^{-|n|}, \quad n\in\Z.
\end{equation}
Using~\eqref{def:nu,A,B}, \eqref{l1_prop_eig_not}, \eqref{l1_prop_eig_neg_not}, and the values of $\xi_j$ in~\eqref{ABneq0_sol_xi} and $\rho_j$ in~\eqref{ABneq0_sol_rho}, we obtain, for $n\ge0$,
\begin{equation}\label{sol_odd_pos}
 f_j(2n+1)=-\frac{\beta_j}{\sqrt{\oH}}\xi_j^{-n}, \quad j=1,2,
\end{equation}
and
\begin{equation}\label{sol_odd_neg}
 f_1(-2n-1)=\sqrt{\oG}\beta_1\xi_1^{-n}, \quad  f_2(-2n-1)=-\frac{\beta_2}{\sqrt{\oG}}\xi_2^{-n}.
\end{equation}
In the cases $A=0$ or $B=0$, the same formulas for $f_1$ are obtained directly from~\eqref{l1_prop_eig_not}--\eqref{l2_prop_eig_neg_not}.
Finally, summing the corresponding geometric series, and using~\eqref{ABneq0_sol_xi}, we obtain
\begin{equation*}
   \|f_{1}\|^2_{\ell^2(\Z)} = |\beta_1|^2 \frac{(\oG+1)(\oH+1)}{\oH-\oG}, \quad \|f_{2}\|^2_{\ell^2(\Z)}= |\beta_2|^2 \frac{(\oG+1)(\oH+1)}{\oG\oH-1}.
\end{equation*}
Thus, the choice of $\beta_1$ and $\beta_2$ given in~\eqref{betas12} normalizes both eigenfunctions. 

\vskip 0.3cm
{\it The case $\vert \xi \vert=1$.}
Writing $\xi=e^{ix}$, where $x\in(0,\pi)$, the spectral equation~\eqref{spec_eq1} becomes 
\begin{equation}\label{spec_eq_cos}
 AB=a^2+b^2+2ab\cos(x)>0,
\end{equation}
hence $A,B\neq0$.
Using the definitions of $A$ and $B$ in~\eqref{def:nu,A,B}, we obtain
\begin{equation}
 \rho=\rho_{\epsilon,x}=\frac{(\oG-1)\mwo+(\oH-1)\mwt+\epsilon R_x}{2}, \quad \epsilon\in\{\pm1\},
\end{equation}
where $R_x$ is defined in~\eqref{R_x2}
\begin{equation*}
R_x=\sqrt{\big((\oG+1)\mwo-(\oH+1)\mwt\big)^2+4(\oG+\oH)\mwo\mwt+8\sqrt{\oG\oH}\mwo\mwt\cos x}.
\end{equation*}
Therefore, from~\eqref{rel_lambda_rk}, the generalized eigenvalues are given by
\begin{equation}
 \lambda_{\epsilon,x} =1-\frac{\rho_{\epsilon,x}}{\kappa}
 =\frac{(\oG+1)\mwo+(\oH+1)\mwt-\epsilon R_x}{2\kappa}.
\end{equation}
Now, fix $\epsilon\in\{\pm1\}$ and $x\in(0,\pi)$, and denote 
$
\rho=\rho_{\epsilon,x}. 
$
Using~\eqref{gen_sol_u_n} in~\eqref{l1_prop_eig_not} and~\eqref{gen_sol_u_n_check} in~\eqref{l1_prop_eig_neg_not}, we obtain
\begin{equation}\label{gensol:v_n}
 v_n=\frac{\alpha(b+a\xi)\xi^n+\beta(b+a\xi^{-1})\xi^{-n}}{A}, \quad n\ge0,
\end{equation}
and
\begin{equation}\label{gensol:v_n_check}
 \check{v}_n=\frac{\check{\alpha}(a+b\xi)\xi^n+\check{\beta}(a+b\xi^{-1})\xi^{-n}}{B}, \quad n\ge0.
\end{equation}
The sewing equations~\eqref{l3_prop_eig_not}-\eqref{eig_not_0} give two linearly independent equations for $\check{\alpha}$ and $\check{\beta}$, since $\xi\neq\xi^{-1}$ for $x\in(0,\pi)$. Therefore, $\check{\alpha}$ and $\check{\beta}$ are uniquely determined in terms of $\alpha$ and~$\beta$, so the complex generalized eigenspace is two-dimensional. 

Let us first choose
$$
(\alpha,\beta)=(1,0).
$$
The positive entries of $f$ are given by~\eqref{gen_sol_u_n} and~\eqref{gensol:v_n} 
\begin{equation}
 f(2n)=u_n=\xi^n, \quad f(2n+1)=v_n=\frac{b+a\xi}{A}\xi^n, \quad n\ge0.
\end{equation}
The sewing equations~\eqref{eig_not_0} and~\eqref{l3_prop_eig_not} become
\begin{align}
 &\check{\alpha}+\check{\beta}=1, \label{sew1} \\
 &\check{\alpha}\xi+\check{\beta}\xi^{-1}=\frac{B}{ab}\Big(\rho-\frac{a^2}{B}-\frac{b^2}{A}\Big)-\frac{B}{A}\xi. \label{sew2}
\end{align}
Solving~\eqref{sew1}-\eqref{sew2}, using~\eqref{spec_eq_cos} and denoting $D=D_{\epsilon,x}$ defined in~\eqref{def:D_ex},
we obtain
\begin{equation}\label{ab_check}
 \check{\alpha}=\frac{1}{2}\Big(1-\frac{B}{A}\Big)-\frac{i}{2}D, \quad \check{\beta}=\frac{1}{2}\Big(1+\frac{B}{A}\Big)+\frac{i}{2}D.
\end{equation}
The negative entries of $f$ are then determined by~\eqref{gen_sol_u_n_check} and~\eqref{gensol:v_n_check}. 

Since $A$, $B$, $\rho$ are real and $\xi^{-1}=\overline{\xi}$, the second choice $(\alpha,\beta)=(0,1)$
gives the complex conjugate solution $\overline{f}$. Thus 
$$
f^{(1)}_{\epsilon,x}=\operatorname{Re}(f), \quad f^{(2)}_{\epsilon,x}=\operatorname{Im}(f),
$$
form a basis of the generalized eigenspace. Using 
$$
\xi^n=\cos nx+i\sin nx
$$
in the formulas above gives~\eqref{gen_efct1} and~\eqref{gen_efct2}.

\end{proof}

\subsubsection{Scalar products of (generalized) eigenfunctions}
We first compute the distributional scalar products of the generalized eigenfunctions. Together with the completeness relation below, these determine the continuous part of the spectral resolution.
\begin{proposition}\label{scalar_products2_prop}
We have the following scalar products, for $\epsilon,\epsilon'\in\{\pm1\}$, $x,y\in(0,\pi)$, $i,j\in\{1,2\}$,
\begin{equation}\label{sc_pr_matrix}
  \big \langle f^i_{\epsilon,x},  f^j_{\epsilon',y} \big \rangle=\mathcal{H}_{\epsilon}(x)_{ij}\delta_{\epsilon,\epsilon'}\delta(x-y)
\end{equation}
where $\delta(x)$ is Dirac's delta function,
\begin{equation}\label{Matrix_H}
 \mathcal{H}_{\epsilon}(x):=\frac{\pi}{2} \Bigl(1+\frac{B_{\epsilon,x}}{A_{\epsilon,x}}\Bigr)H_{\epsilon}(x),
\end{equation}
and $H_{\epsilon}(x)$ is a symmetric matrix whose entries are
\begin{align}
& H_{\epsilon}(x)_{11}=1+\frac{A_{\epsilon,x}}{B_{\epsilon,x}}+\frac{A_{\epsilon,x}}{B_{\epsilon,x}} D^2_{\epsilon,x}, \label{matrix_H11} \\
& H_{\epsilon}(x)_{12}=H_{\epsilon}(x)_{21}= -D_{\epsilon,x}, \label{matrix_H12} \\
& H_{\epsilon}(x)_{22}=1+\frac{B_{\epsilon,x}}{A_{\epsilon,x}}. \label{matrix_H22}
\end{align}
\end{proposition}

\begin{proof}
 Let $\epsilon,\epsilon'\in\{\pm1\}$ and $x,y\in(0,\pi)$. 
 If $\epsilon\neq \epsilon'$, then $\lambda_{\epsilon,x} \neq \lambda_{\epsilon',y}$, since the intervals corresponding to $\epsilon=1$ and $\epsilon=-1$ are disjoint. 
 Thus, by self-adjointness of $\mathcal{L}$, for $i,j\in\{1,2\}$, we have 
 $$
 \big \langle f^{(i)}_{\epsilon,x}, f^{(j)}_{\epsilon',y} \big \rangle =0.
 $$
 
 Let us now fix $\epsilon \in\{\pm1\}$, and denote 
 $$
 A_z:=A_{\epsilon,z}, \quad B_z:=B_{\epsilon,z}, \quad D_z:=D_{\epsilon,z}, \quad z\in(0,\pi).
 $$
 We also introduce 
\begin{equation}\label{not:abcd}
 a:=\sqrt{\oG}\mwo, \ b:=\sqrt{\oH}\mwt, \ c:=(\oG-1)\mwo+(\oH-1)\mwt, \ d:=(\oG-1)\mwo-(\oH-1)\mwt.
\end{equation}
From~\eqref{def:AandB_ex}, we have
\begin{equation} \label{A-B=d}
 A_z-B_z=d.
\end{equation}
We also have the dispersion relation~\eqref{spec_eq_cos},
\begin{equation}\label{disp_eq_cosx}
 A_zB_z=a^2+b^2+2ab\cos z,
\end{equation}
as well as, from the definition of $D_x$ in~\eqref{def:D_ex},
\begin{equation} \label{compl_D}
 2abD_z \sin z=cB_z-\frac{d(a^2-b^2)}{A_z}.
\end{equation}
We will use the distributional identities~\eqref{n,n},\eqref{n,n+1} together with 
\begin{align}
 & \sum_{n\ge0} \cos nx \cos ny =\frac{\pi}{2} \delta(x-y) + \frac12, \label{cos_n,n} \\
 & \sum_{n\ge0} \Big(\cos nx \cos((n+1)y) + \cos((n+1)x) \cos ny \Big)=\pi \cos x \delta(x-y), \label{cos_cos_shifted} \\
 & \sum_{n\ge0} \Big(\cos nx \sin((n+1)y) + \cos((n+1)x) \sin ny \Big)=(\cos x+\cos y)C(x,y)+\frac12 \sin y, \label{cos_sin_shifted}
\end{align}
where
\begin{equation}\label{def:sum_C}
 C(x,y):=\sum_{n\ge0} \cos nx \sin ny.
\end{equation}
We also have 
\begin{equation}\label{cos_sin_C}
 2(\cos x-\cos y)C(x,y)=\sin y, \quad 2(\cos x-\cos y)C(y,x)=-\sin x.
\end{equation}
The proofs of~\eqref{cos_n,n}--\eqref{cos_sin_shifted} and~\eqref{cos_sin_C} are given in Appendix~\ref{identities_prop2}.

Let us compute the different scalar products. 
Using~\eqref{gen_efct2} and collecting the different products of sines, we have
\begin{align*}
 \big \langle f^{(2)}_{\epsilon,x}, f^{(2)}_{\epsilon,y} \big \rangle 
 = &\Big(1+ \frac{2(a^2+b^2)+B_xB_y}{A_xA_y} \Big) \sum_{n\ge0} \sin nx \sin ny \\
 & + \frac{2ab}{A_xA_y} \sum_{n\ge0} \Big(\sin nx \sin((n+1)y) + \sin((n+1)x) \sin ny \Big).
\end{align*}
Using~\eqref{delta_cont_fct}
--\eqref{n,n+1}, we obtain
\begin{align*}
\big \langle f^{(2)}_{\epsilon,x}, f^{(2)}_{\epsilon,y} \big \rangle
 = \frac{\pi}{2} \Big(1+ \frac{B_x^2}{A_x^2}+ \frac{2(a^2+b^2+2ab\cos x)}{A_x^2} \Big) \delta(x-y) = \frac{\pi}{2} \Big(1+ \frac{B_x}{A_x} \Big)^2 \delta(x-y) ,
\end{align*}
where, in the last equality, we used~\eqref{disp_eq_cosx}. This is the formula for $i=j=2$.

Similarly, using~\eqref{gen_efct1},\eqref{gen_efct2}, and collecting the different products of sines and cosines, we have
\begin{align*}
 \big \langle f^{(1)}_{\epsilon,x}, f^{(2)}_{\epsilon,y} \big \rangle 
 = & \left(1 - \frac{B_y}{A_y} + \frac{a^2+b^2}{A_y} \Big( \frac{1}{A_x} - \frac{1}{B_x} \Big) \right) \sum_{n\ge0} \cos nx \sin ny 
  - \frac{D_x}{A_y} \Big(B_y+\frac{a^2+b^2}{B_x} \Big) \sum_{n\ge0} \sin nx \sin ny \\
 & - \frac{abD_x}{B_xA_y} \sum_{n\ge0} \Big(\sin nx \sin((n+1)y) + \sin((n+1)x) \sin ny \Big) \\
 & + \frac{ab}{A_y} \Big( \frac{1}{A_x}- \frac{1}{B_x} \Big) \sum_{n\ge0} \Big(\cos nx \sin((n+1)y) + \cos((n+1)x) \sin ny \Big). 
\end{align*}
Using~\eqref{delta_cont_fct}
--\eqref{n,n+1}, together with~\eqref{cos_sin_shifted} and \eqref{def:sum_C}, and collecting the coefficient of $\delta(x-y)$, we obtain
\begin{align*}
\big \langle f^{(1)}_{\epsilon,x}, f^{(2)}_{\epsilon,y} \big \rangle
 = - \frac{\pi}{2} D_x \Big(1+\frac{B_x}{A_x} \Big) \delta(x-y) + R_{12}(x,y),
\end{align*}
where we used~\eqref{disp_eq_cosx} to simplify the coefficient of $\delta(x-y)$, and
\begin{equation}\label{R_12}
R_{12}(x,y)
 = r_{12}(x,y) C(x,y) + \frac{ab}{2A_y} \Big( \frac{1}{A_x}- \frac{1}{B_x} \Big) \sin y,
\end{equation}
with
\begin{align*}
 r_{12}(x,y)=1 - \frac{B_y}{A_y} + \frac{1}{A_y} \Big( \frac{1}{A_x} - \frac{1}{B_x} \Big) \big(a^2+b^2+ab(\cos x+\cos y)\big).
\end{align*}
Using~\eqref{A-B=d} and~\eqref{disp_eq_cosx}, this simplifies to 
\begin{equation}\label{r_12}
 r_{12}(x,y)= \frac{ab}{A_y} \Big( \frac{1}{A_x} - \frac{1}{B_x} \Big) (\cos y-\cos x).
\end{equation}
Substituting~\eqref{r_12} into~\eqref{R_12} and using~\eqref{cos_sin_C}, we obtain $R_{12}(x,y)=0$. Thus
\begin{equation*}
 \big \langle f^{(1)}_{\epsilon,x}, f^{(2)}_{\epsilon,y} \big \rangle
 = - \frac{\pi}{2} D_x \Big(1+\frac{B_x}{A_x} \Big) \delta(x-y).
\end{equation*}
The case $i=2,j=1$ follows by symmetry. 

Finally, using~\eqref{gen_efct1}, collecting the different products of sines and cosines, we apply~\eqref{delta_cont_fct}--\eqref{n,n+1} and \eqref{cos_n,n}--\eqref{def:sum_C}. Collecting the coefficient of $\delta(x-y)$, we obtain
\begin{align*}
 \big \langle f^{(1)}_{\epsilon,x}, f^{(1)}_{\epsilon,y} \big \rangle
=&\frac{\pi}{2} \left(2+\frac{B_x}{A_x} + \frac{A_x}{B_x} + D_x^2\Big(1+\frac{A_x}{B_x}\Big) \right) \delta(x-y) + R_{11}(x,y) \\
=&\frac{\pi}{2} \Big(1+\frac{B_x}{A_x}\Big) \Big(1+\frac{A_x}{B_x}+\frac{A_x}{B_x}D_x^2\Big) \delta(x-y) + R_{11}(x,y),
\end{align*}
where we used~\eqref{disp_eq_cosx} to simplify the coefficient of $\delta(x-y)$, and
\begin{align}
  R_{11}(x,y) = & \Big(1+\frac{a^2+b^2+ab(\cos x+\cos y)}{B_xB_y} \Big) \big(D_yC(x,y) +D_xC(y,x)\big) \label{line1_R11} \\
 & + \frac{ab}{2B_xB_y} (D_x\sin x+D_y \sin y) +\frac{b^2-a^2}{2} \Big(\frac{1}{A_xA_y}-\frac{1}{B_xB_y} \Big).  \notag 
\end{align}
Let us now show that $R_{11}(x,y)=0$. 
We start by rewriting the first factor in the first term of~\eqref{line1_R11}.
Using~\eqref{disp_eq_cosx} and~\eqref{A-B=d}, we have
\begin{equation}\label{1stfactor_l1}
 1+\frac{a^2+b^2+ab(\cos x+\cos y)}{B_xB_y}
 =\frac{(A_x+A_y-d)(B_x+B_y)}{2B_xB_y}.
\end{equation}
For the second factor in the first term of~\eqref{line1_R11}, by~\eqref{cos_sin_C}, we have
\begin{equation}\label{1st_simpl}
 2(\cos x-\cos y)\left(D_yC(x,y) +D_xC(y,x)\right)=D_y\sin y -D_x\sin x.
\end{equation}
Using~\eqref{compl_D}, we rewrite the right-hand side of~\eqref{1st_simpl} as
\begin{multline} \label{factor_rhs}
 D_y\sin y -D_x\sin x 
 =\frac{1}{2ab} \left(c(B_y-B_x)-d(a^2-b^2)\Big(\frac{1}{A_y}-\frac{1}{A_x}\Big) \right) \\
 =\frac{A_y-A_x}{2ab} \left(c+\frac{d(a^2-b^2) }{A_xA_y} \right),
\end{multline}
where, in the second equality, we used $B_y-B_x=A_y-A_x$, which follows from~\eqref{A-B=d}.
Using~\eqref{disp_eq_cosx} and then~\eqref{A-B=d}, we also have
\begin{equation}\label{1stfactor_lhs}
2ab(\cos x-\cos y)=A_xB_x-A_yB_y
=(A_x-A_y)(A_x+A_y-d).
\end{equation}
Therefore, using~\eqref{factor_rhs} and~\eqref{1stfactor_lhs} in~\eqref{1st_simpl}, we obtain
\begin{equation}\label{2ndfactor_l1}
D_yC(x,y) +D_xC(y,x)=-\frac{c+\frac{d(a^2-b^2) }{A_xA_y} }{2(A_x+A_y-d)}.
\end{equation}
Thus, substituting~\eqref{1stfactor_l1} and~\eqref{2ndfactor_l1} into the first term of~\eqref{line1_R11}, and using~\eqref{compl_D} in the second term, we obtain
\begin{align*}
  R_{11}(x,y) = &-\frac{B_x+B_y}{4B_xB_y} \left(c+\frac{d(a^2-b^2)}{A_xA_y} \right) + \frac{c(B_x+B_y)-d(a^2-b^2) \Big(\frac{1}{A_x}+\frac{1}{A_y}\Big)}{4B_xB_y} \\
  & +\frac{(a^2-b^2)(A_xA_y-B_xB_y)}{2A_xA_yB_xB_y} \\
 = &-\frac{d(a^2-b^2)(A_x+A_y+B_x+B_y)}{4A_xA_yB_xB_y} +\frac{(a^2-b^2)(A_xA_y-B_xB_y)}{2A_xA_yB_xB_y}.
\end{align*}
Since, by \eqref{A-B=d}, we have
$$
A_x+A_y+B_x+B_y=2(A_x+A_y-d), \quad 
A_xA_y-B_xB_y
=d(A_x+A_y-d),
$$
we conclude that $R_{11}(x,y)=0$ and therefore
\begin{equation*}
 \big \langle f^{(1)}_{\epsilon,x}, f^{(1)}_{\epsilon,y} \big \rangle = \frac{\pi}{2} \Big(1+\frac{B_x}{A_x}\Big) \Big(1+\frac{A_x}{B_x}+\frac{A_x}{B_x}D_x^2\Big) \delta(x-y).
\end{equation*}

\end{proof}

\subsubsection{Completeness}
\begin{proposition}\label{completeness2_prop}  
The following completeness condition is satisfied: for all $m,n\in\Z$, 
\begin{equation}\label{completeness2}
\delta_{m,n}= \sum_{k=1}^2f_k(m)f_k(n)+ \int_0^\pi \sum_{\epsilon\in\{\pm1\}} F_{\epsilon,x}(m)^T \mathcal{H}_{\epsilon}(x)^{-1} F_{\epsilon,x}(n) \mathrm{d}x,
\end{equation}
where
$$
F_{\epsilon,x}(m)=
\begin{pmatrix}
f^{(1)}_{\epsilon,x}(m) \\
f^{(2)}_{\epsilon,x}(m)
\end{pmatrix},
$$
and $f_1$, $f_2$, $f^{(1)}_{\epsilon,x}$, $f^{(2)}_{\epsilon,x}$ are given in Proposition~\ref{spectrum_L^pr2}, and $\mathcal{H}_{\epsilon}(x)$ is given in Proposition~\ref{scalar_products2_prop}.
\end{proposition}
The proof is given in Appendix~\ref{appC}.

\subsection{Heat kernel on the projected graph and proof of Theorem~\ref{K2}}

\begin{proposition}\label{heatkernel2}  
The heat kernel of the $\pi_e$-projected Laplacian $\mathcal{L}$ of $\wl$ is given by
\begin{equation}\label{heatkernel2} 
h^{\wl}_t(m,n)= \sum_{k=1}^2 e^{-t\lambda_k} f_k(m)f_k(n)+ \int_0^\pi \sum_{\epsilon\in\{\pm1\}}  e^{-t\lambda_{\epsilon,x}} F_{\epsilon,x}(m)^T \mathcal{H}_{\epsilon}(x)^{-1} F_{\epsilon,x}(n) \mathrm{d}x,  
\end{equation}
where $\lambda_k$, $f_k$, $\lambda_{\epsilon,x}$, $F_{\epsilon,x}$ are given in Proposition~\ref{spectrum_L^pr2} and $\mathcal{H}_{\epsilon}(x)$ is given in Proposition~\ref{scalar_products2_prop}. 
\end{proposition}

We can now prove Theorem~\ref{K2}.
\begin{proof}[Proof of Theorem~\ref{K2}]
Let $h^{\mathcal{Y}}_t=e^{-t\mathcal{L}_{\mathcal{Y}}}$ be the heat kernel of the Laplacian $\mathcal{L}_{\mathcal{Y}}$ of $\mathcal{Y}$. 
Formula~\eqref{eq6} implies that the function $k^{\mathcal{Y}}_t(x)$ in~\eqref{heat_kernel_det} is given by
$$
k^{\mathcal{Y}}_t(x)=K^{\mathcal{Y}}_t(\pi_e(x)), \quad K^{\mathcal{Y}}_t(n)=\frac{1}{\sqrt{|\pi_e^{-1}(n)|}} h^{\wl}_t(0,n),
$$
where $|\pi_e^{-1}(n)|$ is given in~\eqref{fibers2} and $h^{\wl}_t$ is given in~\eqref{heatkernel2}.
\end{proof}

\appendix
\section{Distributional identities} 
\subsection{Proof of~\eqref{n+1,n-1}} \label{appA}
\begin{proof}
Let $x,y\in(0,\pi)$. 
From the distributional form of the Poisson summation formula~\eqref{poisson_summation} and $\delta(\frac{t}{a})=a\delta(t)$, for $a\in\R_{>0}$, we have
\begin{align}
 2\sum_{n\ge0} \cos nt 
 =1+\sum_{n\in\Z} e^{int}
 =1+2\pi \sum_{k\in\Z} \delta(t+2\pi k). \label{sum_cos}
\end{align}
Using \eqref{2sin_id}, and then the cosine addition formula, we have
\begin{multline*}
  2\sum_{n\ge0} \Big(\sin((n+1)x)\sin((n-1)y) + \sin((n+1)y)\sin((n-1)x) \Big) \label{2(n+1,n-1)} \\
  = 2\cos(x+y) \sum_{n\ge0} \cos(n(x-y))-2\cos(x-y) \sum_{n\ge0} \cos(n(x+y)),
 \end{multline*} 
 using \eqref{sum_cos}, we obtain
  \begin{equation*}
  = \cos(x+y) \Big(1+2\pi \sum_{k\in\Z} \delta(x-y+2\pi k)\Big) -\cos(x-y) \Big(1+2\pi \sum_{k\in\Z} \delta(x+y+2\pi k)\Big).
   \end{equation*} 
   Since $x,y\in(0,\pi)$, only the term $k=0$ contributes to the first sum, while the second sum vanishes, thus we obtain
   \begin{equation*}
  = \cos(x+y) \Big(1+2\pi \delta(x-y)\Big) -\cos(x-y) = \cos(x+y)-\cos(x-y) + 2\pi\cos(2x) \delta(x-y),
   \end{equation*} 
   using~\eqref{2sin_id} and~\eqref{delta_cont_fct}, we obtain
   \begin{equation*}
  = -2\sin x \sin y + 2\pi\cos(2x) \delta(x-y).
   \end{equation*} 
   Dividing by 2 proves~\eqref{n+1,n-1}.
\end{proof}

\subsection{Identities used in the proof of Proposition~\ref{scalar_products2_prop}} \label{identities_prop2}
Let $x,y\in(0,\pi)$.

Identity~\eqref{cos_n,n} is obtained from 
 $2\cos a \cos b=\cos(a-b)+\cos(a+b)$,
 \eqref{sum_cos}, and the argument used in the proof of~\eqref{n+1,n-1}.

\begin{proof}[Proof of~\eqref{cos_cos_shifted}]
Using 
\begin{equation} \label{2_cos_cos_sum}
 2\cos x \cos nx=\cos((n+1)x)+\cos((n-1)x),
\end{equation}
and reindexing the second sum, we have
$$
2\cos x \sum_{n\ge0} \cos nx \cos ny =\cos x + \sum_{n\ge0} \Big( \cos((n+1)x)\cos ny + \cos nx\cos ((n+1)y) \Big).
$$
Thus, using~\eqref{cos_n,n}, we obtain
\begin{align*}
 \sum_{n\ge0} \Big( \cos((n+1)x)\cos ny + \cos nx\cos ((n+1)y) \Big) 
 =\pi \cos x \delta(x-y).
\end{align*}
\end{proof}

\begin{proof}[Proof of~\eqref{cos_sin_shifted}]

Using the sine and cosine addition formulas and definition~\eqref{def:sum_C} of $C(x,y)$, we have 
\begin{align*} 
&\sum_{n\ge0} \Big( \cos nx \sin((n+1)y) +\cos((n+1)x)\sin ny \Big) \\ 
&\qquad= (\cos x+\cos y)C(x,y) +\sin y\sum_{n\ge0} \cos nx \cos ny -\sin x\sum_{n\ge0} \sin nx \sin ny,
\end{align*}
using~\eqref{delta_cont_fct},\eqref{n,n}, and~\eqref{cos_n,n}, we obtain
$$ 
\sum_{n\ge0} \Big(\cos nx \sin((n+1)y) +\cos((n+1)x)\sin ny \Big) = (\cos x+\cos y)C(x,y)+\frac12\sin y. 
$$
\end{proof}

\begin{proof}[Proof of~\eqref{cos_sin_C}]
For $N\ge0$, using \eqref{2_cos_cos_sum}, and
\begin{equation*}
 2\cos y \sin ny = \sin(n+1)y)+\sin((n-1)y),
\end{equation*}
and reindexing two sums, we have 
\begin{align*}
 2(\cos x-\cos y)\sum_{n=0}^N \cos nx \sin ny 
 &=\sin y+ \cos((N+1)x) \sin Ny - \cos Nx \sin((N+1)y).
\end{align*}
Taking the limit $N\to\infty$, the last two terms tend to 0 distributionally. Thus, we obtain
$$
 2(\cos x-\cos y)C(x,y)=\sin y.
$$
The second identity is obtained by exchanging $x$ and $y$.
\end{proof}

\section{Proof of Proposition~\ref{completeness1_prop}} \label{appB}
\begin{proof} 
For convenience, we introduce
$$
\alpha=\frac{1}{\sqrt{(\mN-1)(r-1)}}, \ \beta=\sqrt{\frac{r-1}{\mN-1}},
$$
so that $\mN-1=\frac{1}{\alpha\beta}$ and $r-1=\frac{\beta}{\alpha}$. Note that $\alpha <1$, while $\beta <1$ if and only if $\mN>r$.
In the case $\mN>r$, formula~\eqref{disc_efct} becomes
\begin{equation}
 f_*(n)=
\begin{cases}
 (-\beta)^n \sqrt{\frac{(\alpha+\beta)(1-\beta^2)}{(1+\alpha \beta)\beta}} & \text{if } n\ge1 \\
  \sqrt{\frac{1-\beta^2}{1+\alpha \beta}} & \text{if } n=0.
\end{cases}
\end{equation}
Formula~\eqref{gen_efct} becomes
\begin{equation}
 f_x(n)=
\begin{cases}
 \frac{\beta}{\alpha}\sin((n+1)x)+\frac{1-\alpha \beta}{\alpha}\sin(nx)-\sin((n-1)x) & \text{if } n\ge1 \\
 \frac{\sqrt{(\alpha+\beta)\beta}}{\alpha}\sin(x) & \text{if } n=0.
\end{cases}
\end{equation}
We also rewrite the positive function $H$ from Proposition~\ref{scalar_products1} in terms of $\alpha$, $\beta$:
$$
 H(x)= \frac{\pi}{2\alpha^2} \Big(\big(1-\alpha\beta +(\beta-\alpha)\cos(x)\big)^2 + (\alpha+\beta)^2\sin^2(x) \Big).
 $$
Using $\sin(x)=\frac{e^{ix}-e^{-ix}}{2i}$, we obtain
  \begin{align}
   H(x) &= \frac{\pi}{2\alpha^2}
  \Big\vert 1-\alpha\beta+\beta e^{ix}- \alpha e^{-ix} \Big\vert^2 
 =\frac{\pi}{2\alpha^2} \vert (e^{ix}-\alpha) (1+\beta e^{ix}) \vert^2. \label{newH(x)}
\end{align}
In particular, the factors show that after setting $z=e^{ix}$, the denominator contributes poles at $z=\alpha$, $z=\frac{1}{\alpha}$, $z=-\beta$ and $z=-\frac{1}{\beta}$.

Proving the completeness identity~\eqref{completeness1} reduces to evaluating the integral in that expression. We compute it by using the generating functions, rewriting the integrand as a rational function on the unit circle via $z=e^{ix}$ and applying the residue theorem.

We consider the generating functions, for $|w|<1$,
$$
g_*(w)=\sum_{n\ge0} f_*(n)w^n \quad \text{and} \quad g_x(w)=\sum_{n\ge0} f_x(n)w^n.
$$
Let us fix $u,v\in\C$ with $|u|,|v|<1$. Multiplying the completeness identity~\eqref{completeness1} by $u^mv^n$ and summing over $m,n\ge0$, the left-hand side becomes 
$$
\sum_{m,n\ge0} \delta_{m,n}u^mv^n=\sum_{n\ge0} (uv)^n=\frac{1}{1-uv},
$$	
and the right-hand side produces the generating functions. Therefore, expression~\eqref{completeness1} becomes
\begin{equation}\label{completeness_genfct}
\frac{1}{1-uv}=\int_0^\pi g_x(u)g_x(v)\frac{\mathrm{d}x}{H(x)}+
\begin{cases}
 g_*(u)g_*(v)&\text{if } \mN>r\\
 0&\text{if } \mN\le r.
\end{cases}
\end{equation}
Now, let us compute $g_*(w)$ and $g_x(w)$. We have
\begin{align*}
 g_*(w)
 = \sqrt{\frac{1-\beta^2}{1+\alpha \beta}} + \sqrt{\frac{(\alpha+\beta)(1-\beta^2)}{(1+\alpha \beta)\beta}} \sum_{n\ge1}(-\beta w)^n 
 = \sqrt{\frac{1-\beta^2}{1+\alpha \beta}} \Big(1-\sqrt{\frac{\alpha+\beta}{\beta}} \frac{\beta w}{1+\beta w} \Big),
\end{align*}
where, in the last equality, we used the geometric series with $|\beta|, |w|<1$. We also have
\begin{align*}
  g_x(w) 
  = \frac{\sqrt{(\alpha+\beta)\beta}}{\alpha}\sin(x) +  \sum_{n\ge1}\Big(\frac{\beta}{\alpha}\sin((n+1)x)+\frac{1-\alpha \beta}{\alpha}\sin(nx)-\sin((n-1)x)\Big) w^n 
\end{align*}
reindexing the sum by $n\mapsto n+1$ and using $\sin(mx)=\frac{z^m-z^{-m}}{2i}$, where $z=e^{ix}$, 
we factor out the sums
\begin{multline*}
 = \frac{1}{2i} \Bigg( \frac{\sqrt{(\alpha+\beta)\beta}}{\alpha}(z-z^{-1})
+ \Big( \frac{\beta}{\alpha}z^{2} +\frac{1-\alpha \beta}{\alpha}z - 1 \Big)w \sum_{n\ge0} (zw)^n  \\
 - \Big( \frac{\beta}{\alpha}z^{-2} +\frac{1-\alpha \beta}{\alpha}z^{-1} - 1 \Big)w \sum_{n\ge0} (z^{-1}w)^n \Bigg) \Bigg\vert_{z=e^{ix}},
\end{multline*}
using the geometric series and factoring out $\frac{1}{\alpha z}$, we obtain
\begin{align*}
 g_x(w)
 &=\frac{1}{2i \alpha z} \Big(\sqrt{(\alpha+\beta)\beta}(z^2-1)
+ \frac{(z-\alpha)(1+\beta z)zw}{1-zw}
 - \frac{(z+\beta)(1-\alpha z)w}{z-w} \Big) \Big\vert_{z=e^{ix}}.
\end{align*}
Let us define 
\begin{equation}\label{part_g}
 h(z,w):=\sqrt{(\alpha+\beta)\beta}(z^2-1)
+ \frac{(z-\alpha)(1+\beta z)zw}{1-zw}
 - \frac{(z+\beta)(1-\alpha z)w}{z-w},
\end{equation}
where $z=e^{ix}$ and $|w|<1$, so that 
\begin{equation}\label{short_g}
 g_x(w)=\frac{1}{2i \alpha z} h(z,w) \vert_{z=e^{ix}}.
\end{equation}
From~\eqref{newH(x)},
denoting $z=e^{ix}$, and factoring $\frac{1}{z^2}$, we obtain
\begin{equation}\label{newH}
 H(x)=\frac{\pi}{2\alpha^2z^2}(z-\alpha)(1-\alpha z)(1+\beta z)(z+\beta) \big\vert_{z=e^{ix}}.
\end{equation}
Let us now compute the integral in~\eqref{completeness_genfct}. Since $g_{-x}(n)=-g_{x}(n)$ and $H(x)=H(-x)$, we have
\begin{align*}
I(u,v):=\int_0^\pi g_x(u)g_x(v)\frac{\mathrm{d}x}{H(x)} &= \frac{1}{2} \int_{-\pi}^\pi g_x(u)g_x(v)\frac{\mathrm{d}x}{H(x)}.
\end{align*}
Using~\eqref{short_g}, \eqref{newH} and the complex variable $z=e^{ix}$, so that $\mathrm{d}x=\frac{\mathrm{d}z}{iz}$, the last integral becomes a contour integral over the unit circle 
\begin{align*}
I(u,v) &=-\frac{1}{4\pi i} \int_{|z|=1}\frac{h(z,u)h(z,v)}{(z-\alpha)(1-\alpha z)(1+\beta z)(z+\beta)z}\mathrm{d}z.
\end{align*}
We compute this integral using the residue theorem. Denote the integrand
\begin{equation}\label{integrand_F}
 F(z):=\frac{h(z,u)h(z,v)}{(z-\alpha)(1-\alpha z)(1+\beta z)(z+\beta)z},
\end{equation}
which is a rational function of $z$.

\subsection{The case $\mN>r$} 
The denominator of $F(z)$ contributes simple poles at $$z\in\{0,\alpha^{\pm1}, -\beta^{\pm1}\}.$$
Since $0<\alpha<1$, and for $\mN>r$, $0<\beta <1$, only 
$$
z\in\{0,\alpha,-\beta\}
$$ 
lies inside the unit disk. Moreover, $h(z,u)$ and $h(z,v)$ have simple poles at $z\in\{u^{\pm1}, v^{\pm1}\}$. For $|u|,|v|<1$, only 
$$
z\in\{u,v\}
$$
lies inside the unit disk. Therefore, we obtain

\begin{align*}
 I(u,v)&=-\frac{1}{4\pi i} \int_{|z|=1}F(z)\mathrm{d}z
 =-\frac{1}{2} \sum_{z_0\in\{0,\alpha,-\beta,u,v\}}\mathrm{Res}(F,z_0).
\end{align*}
The corresponding residues are given by
\begin{align}
 & \mathrm{Res}(F,u) =-\frac{h(u,v)}{(u-\alpha)(1+\beta u)} \label{Res_u}\\
 & \mathrm{Res}(F,v) =-\frac{h(v,u)}{(v-\alpha)(1+\beta v)} \label{Res_v}\\
 & \mathrm{Res}(F,0) =\frac{h(0,u)h(0,v)}{-\alpha \beta}=-\frac{(\sqrt{(\alpha+\beta)\beta}+\beta)^2}{\alpha \beta} \label{Res_0}\\
 & \mathrm{Res}(F,\alpha) =\frac{h(\alpha,u)h(\alpha,v)}{(1-\alpha^2)(1+\alpha\beta)(\alpha+\beta)\alpha} \label{Res_alpha}\\
 & \mathrm{Res}(F,-\beta) =\frac{h(-\beta,u)h(-\beta,v)}{(\alpha+\beta)(1+\alpha \beta)(1-\beta^2)\beta}. \label{Res_-beta}
 \end{align}
and their sum is a rational function of $u$ and $v$. After simplifications, we obtain
$$
I(u,v)=\frac{1}{1-uv}- g_*(u)g_*(v).
$$
\subsection{The case $\mN<r$} 
In this case, $\beta>1$, while $\alpha<1$, so 
only 
$$
z\in\{0,\alpha,-\beta^{-1},u,v\}
$$ 
lies inside the unit disk. Therefore, we obtain
\begin{align*}
 I(u,v)&=-\frac{1}{2} \sum_{z_0\in\{0,\alpha,-\beta^{-1},u,v\}}\mathrm{Res}(F,z_0).
\end{align*}
The corresponding residues are given in  \eqref{Res_u}--\eqref{Res_alpha}, and
\begin{align}\label{Res_-1/beta}
 & \mathrm{Res}(F,-\beta^{-1}) =\frac{h(-\beta^{-1},u)h(-\beta^{-1},v)}{(\alpha+\beta^{-1})(1+\alpha\beta^{-1})(\beta-\beta^{-1})\beta^{-1}}.
 \end{align}
After simplifications, we obtain
$$
I(u,v)=\frac{1}{1-uv}.
$$

\subsection{The case $\mN=r$} 
In this case, $\beta=1$, while $\alpha=\frac{1}{\mN-1}<1$, so~\eqref{part_g} and \eqref{integrand_F} become
\begin{align*}
  h(z,w)&=(z+1)\Big(\sqrt{(\alpha+1)}(z-1)+\frac{(z-\alpha)zw}{1-zw}- \frac{(1-\alpha z)w}{z-w} \Big) 
\end{align*}
$$
 F(z)=\frac{h(z,u)h(z,v)}{(z-\alpha)(1-\alpha z)(z+1)^2z}.
$$
Thus, $F(z)$ has a removable singularity at $z=-1$ and the only simple poles inside the unit disk are at $$z\in\{0,\alpha,u,v\}.$$
The corresponding residues are given in  \eqref{Res_u}--\eqref{Res_alpha}. After simplifications, we obtain
$$
I(u,v)=\frac{1}{1-uv}.
$$
\end{proof}

\section{Proof of Proposition~\ref{completeness2_prop}} \label{appC}
\begin{proof} 
We start by determining $\mathcal{H}_{\epsilon}(x)^{-1}$. Fix $\epsilon \in \{\pm1\}$, $x\in (0,\pi)$, and for simplicity, denote
\begin{equation*}
 A:=A_{\epsilon,x}, \quad B:=B_{\epsilon,x}, \quad D:=D_{\epsilon,x}.
\end{equation*}
From~\eqref{Matrix_H}--\eqref{matrix_H22}, we rewrite 
$\mathcal{H}_{\epsilon}(x)$ as
\begin{equation*}
 \mathcal{H}_{\epsilon}(x)= \frac{\pi(A+B)}{2A}
 \begin{pmatrix}
 1+\frac{A}{B}+\frac{A}{B}D^2 & -D \\
 -D & 1+\frac{B}{A}
\end{pmatrix}.
\end{equation*}
Using that the determinant of the matrix on the right is $\frac{(A+B)^2+A^2D^2}{AB}$, we obtain 
\begin{equation}\label{H_inv}
\mathcal{H}_\epsilon(x)^{-1}=\frac{2A^2B}{\pi(A+B)\big((A+B)^2+A^2D^2\big)}
 \begin{pmatrix}
 1+\frac{B}{A} & D \\
 D & 1+\frac{A}{B}+\frac{A}{B}D^2
\end{pmatrix}.
\end{equation}
Using the notation introduced in~\eqref{not:abcd}, together with~\eqref{def:rho_ex}, \eqref{def:AandB_ex} and~\eqref{disp_eq_cosx}, we have 
\begin{equation}\label{not:A,B}
 A-B=d, \quad A+B=\epsilon R_x, \quad AB=a^2+b^2+2ab\cos x.
\end{equation}
From~\eqref{compl_D}, we have
\begin{equation} 
 2abAD \sin x=cAB-d(a^2-b^2),
\end{equation}
so $AD$ is independent of $\epsilon$, and we denote 
\begin{equation}\label{not:E_x} 
 E_x:=A_{\epsilon,x}D_{\epsilon,x}=AD.
\end{equation}
Using~\eqref{not:A,B} and~\eqref{not:E_x}, we rewrite~\eqref{H_inv} as
\begin{equation}
 \mathcal{H}_\epsilon(x)^{-1}=\frac{2}{\pi \big(R_x^2+E_x^2\big)}
 \begin{pmatrix}
 AB & \frac{ABE_x}{\epsilon R_x} \\
 \frac{ABE_x}{\epsilon R_x} & A^2+\frac{AE_x^2}{\epsilon R_x}
\end{pmatrix}.
\end{equation}
Denote
\begin{equation*}
 C_x(m,n):= \sum_{\epsilon\in\{\pm1\}} F_{\epsilon,x}(m)^T \mathcal{H}_{\epsilon}(x)^{-1} F_{\epsilon,x}(n).
\end{equation*}
We compute $C_x(m,n)$ according to the signs and parities of $m$ and $n$. We give the details for the case $m=2k$ and $n=2l$, with $k,l\ge0$, and the other cases are similar.

For $m,n\ge0$, we have
\begin{equation}\label{C_even_even_pos}
 C_x(2k,2l) 
 = \frac{2}{\pi} \sin kx \sin lx + \frac{4(a^2+b^2+2ab\cos x)}{\pi(R_x^2+E_x^2)} \cos((k+l)x),
\end{equation}
Let us rewrite the second term. From~\eqref{not:A,B}--\eqref{not:E_x}, we have
\begin{equation*} 
 E_x=\frac{c(a^2+b^2+2ab\cos x)-d(a^2-b^2)}{2ab\sin x}.
\end{equation*}
Also, from~\eqref{R_x2}, together with~\eqref{not:abcd}, we have
\begin{equation*} 
 R_x^2
 =d^2+4(a^2+b^2+2ab\cos x).
\end{equation*}
Thus, after simplification, we obtain
\begin{equation*} \label{denom_exp}
 R_x^2+E_x^2=\frac{(a^2+b^2+2ab\cos x)(\oG+\oH+2\sqrt{\oG\oH}\cos x)(\oG\oH+1-2\sqrt{\oG\oH}\cos x)}{\oG\oH\sin^2 x}.
\end{equation*}
Using this,
the second term in~\eqref{C_even_even_pos} becomes
\begin{equation} \label{simpl_second_term}
 \frac{4\oG\oH\sin^2 x}{\pi(\oG+\oH+2\sqrt{\oG\oH}\cos x)(\oG\oH+1-2\sqrt{\oG\oH}\cos x)} \cos((k+l)x).
\end{equation}
Using the partial fraction decomposition, we have
\begin{multline}
 \frac{1}{(\oG+\oH+2\sqrt{\oG\oH}\cos x)(\oG\oH+1-2\sqrt{\oG\oH}\cos x)} \\ 
 =\frac{1}{(\oG+1)(\oH+1)}\left(\frac{1}{\oG+\oH+2\sqrt{\oG\oH}\cos x}+\frac{1}{\oG\oH+1-2\sqrt{\oG\oH}\cos x}\right).
\end{multline}
Using the identity for $|r|<1$, 
\begin{equation*}
\frac{1}{1+r^2-2r\cos x}=\frac{1}{1-r^2}\Bigg(1+2\sum_{j=1}^\infty r^j \cos jx \Bigg),
\end{equation*}
we obtain, since $\oG<\oH$,
\begin{equation}\label{first_frac}
\frac{1}{\oG+\oH+2\sqrt{\oG\oH}\cos x}=\frac{1}{\oH-\oG} \Bigg(1+2\sum_{j=1}^\infty \left(-\sqrt{\frac{\oG}{\oH}}\right)^j \cos jx \Bigg)
\end{equation}
and 
\begin{equation}\label{second_frac}
\frac{1}{\oG\oH+1-2\sqrt{\oG\oH}\cos x}=\frac{1}{\oG\oH-1} \Bigg(1+2\sum_{j=1}^\infty \left(\frac{1}{\sqrt{\oG\oH}}\right)^j \cos jx \Bigg).
\end{equation}
We can now integrate~\eqref{C_even_even_pos}. Using 
$$
\int_0^{\pi} \sin kx \sin lx \mathrm{d}x =
\begin{cases}
  \frac{\pi}{2} & \text{if } k=l\ge1 \\
  0 & \text{otherwise},
\end{cases}
$$
the integral of the first term in~\eqref{C_even_even_pos} is $\delta_{k,l}-\delta_{k,0}\delta_{l,0}$.
For the second term in~\eqref{C_even_even_pos},
using~\eqref{simpl_second_term}--\eqref{second_frac}, together with
\begin{equation*}
\sin^2 x \cos nx= \frac12 \cos nx-\frac14 \cos((n+2)x)-\frac14 \cos((n-2)x),
\end{equation*}
and the orthogonality relations
$$
 \int_0^{\pi} \cos jx \cos nx \mathrm{d}x =
\begin{cases}
  \pi & \text{if } j=n=0 \\
  \frac{\pi}{2} & \text{if } j=n\ge1 \\
  0 & \text{if } j\neq n
\end{cases},
\quad
 \int_0^{\pi} \cos nx \mathrm{d}x =0, \ \ n\ge1,
$$
we obtain, for $|r|<1$,
\begin{equation*}
 \int_0^{\pi} \sin^2x \cos nx \left(1+2\sum_{j=1}^\infty r^j \cos jx \right) \mathrm{d}x =
\begin{cases}
 \frac{\pi}{2}(1-r^2) & \text{if } n=0 \\
 \frac{\pi}{4}r(1-r^2) & \text{if } n=1 \\
 -\frac{\pi}{4}(1-r^2)^2r^{n-2} & \text{if } n\ge2.
\end{cases}
\end{equation*}
Thus, the integral of the second term is
\begin{equation*}
 \delta_{k+l,0} -\frac{\oH-\oG}{(\oG+1)(\oH+1)}\left(-\sqrt{\frac{\oG}{\oH}}\right)^{k+l} - \frac{\oG\oH-1}{(\oG+1)(\oH+1)} \left(\frac{1}{\sqrt{\oG\oH}} \right)^{k+l} 
\end{equation*}
Putting these two terms together and using $\delta_{k+l,0}=\delta_{k,0}\delta_{l,0}$, since $k,l\ge0$, we obtain
\begin{align*}
 \int_0^{\pi} C_x(2k,2l) \mathrm{d}x = \delta_{k,l}-\frac{\oH-\oG}{(\oG+1)(\oH+1)}\left(-\sqrt{\frac{\oG}{\oH}}\right)^{k+l} - \frac{\oG\oH-1}{(\oG+1)(\oH+1)} \left(\frac{1}{\sqrt{\oG\oH}} \right)^{k+l}.
\end{align*}
For the discrete contribution, from~\eqref{disc_efct1}--\eqref{betas12}, we have
\begin{equation*}
 f_1(2k)f_1(2l)+f_2(2k)f_2(2l)= \frac{\oH-\oG}{(\oG+1)(\oH+1)} \left(-\sqrt{\frac{\oG}{\oH}}\right)^{k+l} + \frac{\oG\oH-1}{(\oG+1)(\oH+1)} \left(\frac{1}{\sqrt{\oG\oH}}\right)^{k+l}.
\end{equation*}
Thus, we obtain
\begin{equation*}
 f_1(2k)f_1(2l)+f_2(2k)f_2(2l)+ \int_0^\pi C_x(2k,2l) \mathrm{d}x=\delta_{k,l}=\delta_{2k,2l},
\end{equation*}
proving~\eqref{completeness2} in this case.
\end{proof}

\section*{Funding}
This work was supported by the Swiss NSF Grant 200021-212864.

\section*{Acknowledgments} 
I would like to thank my advisor, Anders Karlsson, for his guidance and support throughout this work. I am also grateful to Rostislav Grigorchuk for valuable discussions.


\end{document}